\documentclass[listof=flat]{scrartcl}

\usepackage{graphicx}
\graphicspath{{figures/}}
\usepackage{color, colortbl}
\usepackage{multirow}
\usepackage{caption}
\usepackage{tabularx}
\usepackage{aligned-overset}
\usepackage{amssymb,amsmath,amsthm,mathtools, bbm, mathrsfs} 
\usepackage{nicefrac}
\usepackage[margin=1.25in]{geometry}
\usepackage[utf8]{inputenc}
\usepackage{scrlayer} 
\usepackage{enumerate}
\usepackage[ngerman,english]{babel}
\usepackage[T1]{fontenc}
\usepackage{bbm}
\usepackage{tabularx}
\usepackage{mathdots}
\usepackage[makeroom]{cancel}
\usepackage{tikz}
\usepackage[draft]{todonotes}
\usepackage{tikz-imagelabels}
\usepackage{caption}
\usepackage[labelformat=simple]{subcaption}

\usepackage[wby]{callouts}
\usepackage{hyperref}
\usepackage{varioref}
\usepackage{comment}
\newtheorem{theorem}{Theorem}[section]

\newtheorem{lemma}[theorem]{Lemma}
\newtheorem{proposition}[theorem]{Proposition}

\newcommand{\dist}{\operatorname{dist}}
\newcommand{\vol}{\operatorname{Vol}}
\newcommand{\rot}{\operatorname{rot}}

\theoremstyle{definition} 
\newtheorem{mydef}{Definition}[section]
\newtheorem{myrem}{Remark}[section]

\definecolor{pgchange}{rgb}{0.5,0,0.55}

\definecolor{ao}{rgb}{0.0, 0.5, 0.0}
\definecolor{tuscanred}{rgb}{0.51, 0.21, 0.21}
\definecolor{babypink}{rgb}{0.96, 0.76, 0.76}
\definecolor{candypink}{rgb}{1.0, 0.67, 0.79}
\definecolor{darkpastelgreen}{rgb}{0.01, 0.75, 0.24}
\definecolor{red(pigment)}{rgb}{0.93, 0.11, 0.14}
\definecolor{unmellowyellow}{rgb}{1.0, 1.0, 0.5}
\definecolor{turquoiseblue}{rgb}{0.0, 1.0, 0.94}
\definecolor{blue(pigment)}{rgb}{0.23, 0.3, 0.92}
\begin{document}
	
	\title{Robustness in the Poisson Boolean model with convex grains}
	
	\author{Peter Gracar, Marilyn Korfhage and Peter M\"orters}

    \date{}
	
	\maketitle

	\begin{abstract}\noindent\textbf{Abstract:} We study the Poisson Boolean model where the grains are random convex bodies with a rotation-invariant distribution. A grain distribution is \emph{dense} if the union of the grains covers the entire space
 and \emph{robust} if the union of the grains has an unbounded connected component irrespective of the intensity of the underlying Poisson process. If the grains are balls of random radius, then density and robustness are equivalent, but in general this is not the case. We show that in any dimension $d\ge2$ there are grain distributions that are robust but not dense and give criteria for density, robustness and non-robustness of a grain distribution. 
 For robustness, we provide a 
 criterion that requires only information on the tail index of the marginal diameter distributions of the grains, as well as 
 criteria that
 apply when diameters are sufficiently positively correlated. We also show the existence of grains, where the marginal distribution of the diameters does not provide sufficient information to determine whether the grain distribution is robust.
 
	\end{abstract}

	\renewcommand{\contentsname}{Contents}
	\tableofcontents
	
	%

	%
	\section{Introduction and statement of the results}\label{Sect:introduction}
    
	Take a homogeneous Poisson point process $\mathscr{P}$ in $\mathbb{R}^d$ of dimension $d\ge2$ with positive intensity~$u$ and mark every point $x$ of $\mathscr{P}$ with an independent copy $\widetilde{C}_x$ of a  random convex body $C\subset \mathbb{R}^d$  such that
 the distribution of $C$ is rotation invariant about the origin. We will write $\vol(C)$ for the Lebesgue measure, i.e.\ the volume of the convex body $C$. We denote 
	\begin{equation*}
		C_x:= x+\widetilde{C}_x = \{y\in\mathbb{R}^d \colon y-x \in \widetilde{C}_x\},
	\end{equation*}
	and following~\cite{Hall1985} let
	$\mathscr{C}$ be the union of the convex bodies $(C_x)_{x\in\mathscr{P}}$, i.e.
		\begin{equation*}
			\mathscr{C}:= \bigcup_{x\in\mathscr{P}} C_x.
		\end{equation*}
	The set $\mathscr{C}$ is called the \emph{Poisson Boolean model} with convex grain $C$ and intensity~$u$. There is also a natural graph $\mathscr G=(\mathscr P, \mathscr E)$ associated with this model, where the points of the Poisson process $\mathscr P$ constitute the vertex set and there is an (unoriented) edge connecting distinct points $x,y\in\mathscr P$ if and only if 
	$C_x \cap C_y\not=\emptyset$. A survey on Boolean models can be found at~\cite{Hug2023} and recent or important works on percolation 
 problems for the Boolean model with general convex grains are~\cite{Gouere2023, Hall1985, Roy2002} and~\cite{Hilario2021, Teixeira2017}.
	\medskip


 \begin{figure}[h]
           \begin{subfigure}{1\linewidth} 
              \includegraphics[width=\linewidth]{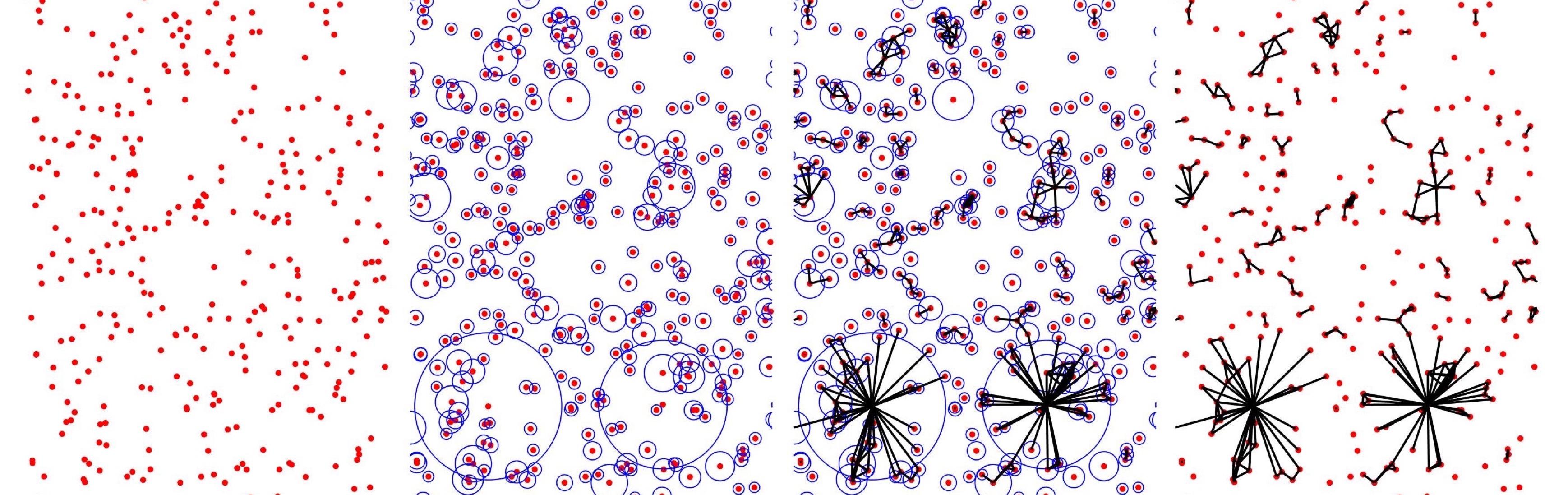}
              \subcaption{Example in $\mathbb{R}^2$ for the Poisson Boolean model with balls as grains.}
           \end{subfigure}
           \hfill
           \begin{subfigure}{1\linewidth}
              \includegraphics[width=\linewidth]{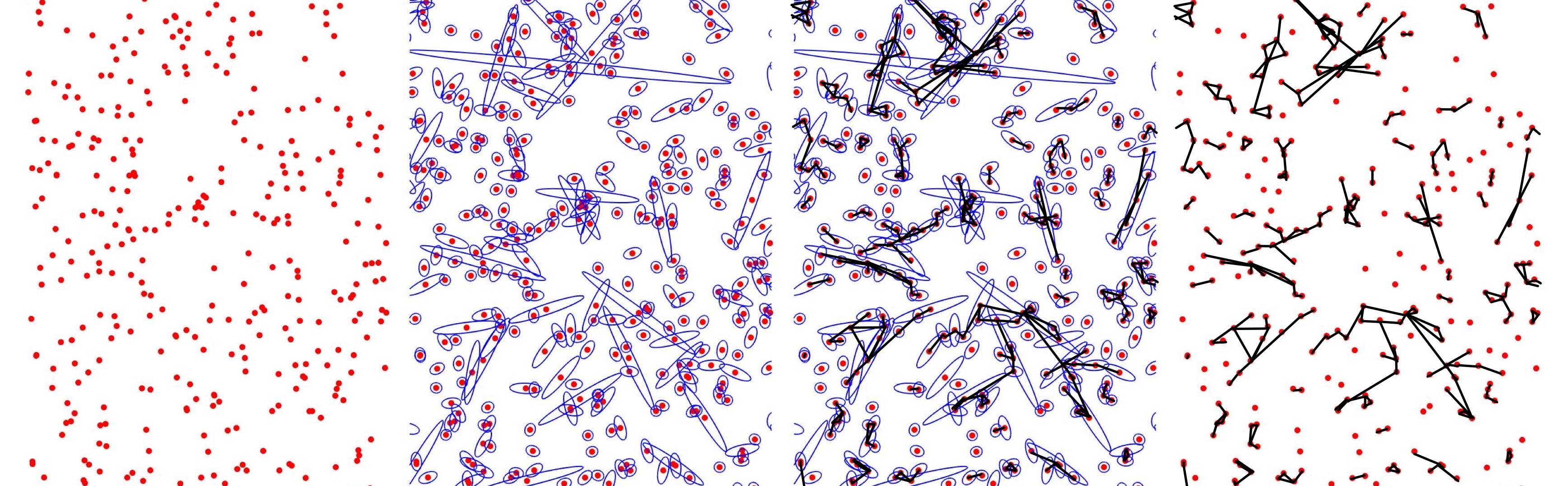}
              \subcaption{Example in $\mathbb{R}^2$ for the Poisson Boolean model with ellipses as grains.}
           \end{subfigure}
           \caption{Example in $\mathbb{R}^2$ for the Poisson Boolean model with convex grains. From left to right: Points placed by the Poisson point process, covering of the process by the random marks, edges constructed via the corresponding marks and on the right side the Gilbert graph.}     
           \label{fig:Boolean_model}
        \end{figure}

	There is little interest in the Boolean model if $\mathscr C=\mathbb R^d$ and we call the grain distribution \emph{dense} if this is the case almost surely, for any Poisson intensity $u>0$. Otherwise it is called \emph{sparse}. We now give a few equivalent characterisations of a grain distribution being sparse. 
 Define
	\begin{itemize}
		\item $M_0$ as the number of grains containing the origin,
		\begin{equation*}
			M_0 := \sum\limits_{x\in\mathscr{P}} \mathbbm{1}_{0\in C_x}.
		\end{equation*}
		\item $N_x$ as the degree of the vertex $x\in \mathscr{P}$ in the graph~$\mathscr G$,
		\begin{equation*}
			N_x := \sum\limits_{\substack{y\in\mathscr{P} \\ y\neq x}}  \mathbbm{1}_{C_x \cap C_y \neq \emptyset} 
		\end{equation*}
		\item $N_A$ as the degree of the set $A\in\mathcal{B}(\mathbb{R}^d)$,
		\begin{equation*}
			N_A:=\sum\limits_{x\in\mathscr{P}} \mathbbm{1}_{A \cap C_x \neq \emptyset}.
		\end{equation*}
		With that we get $N_x =N_{C_x}$ for $x\in\mathscr{P}$.
	\end{itemize}

	\begin{proposition}\label{Prop:one}
	Let $\mathscr{P}$ be a homogeneous Poisson point process in $\mathbb R^d$ and mark the points with independent random convex bodies $C\subset \mathbb{R}^d$ containing 
 a ball of fixed radius. 
 Let $\mathscr{P}_0$ be the Palm version of the marked Poisson point process, 
 and denote its law by $\mathbb{P}_0$ and by $\mathbb{E}_0$ the corresponding expectation.  The statement
 \begin{equation}\label{sparse}
 \mathbb{E} [\vol(C)]<\infty
 \end{equation}
 and the following statements are equivalent:\\[3mm]
	\begin{tabular}{lll}
	(a) $\mathbb{P} (\mathscr{C}=\mathbb{R}^d)<1$ 
	\phantom{gobldigobl}&
(b) $\mathbb{P} (\mathscr{C}=\mathbb{R}^d)=0$\phantom{gobldigobl} &
		(c) $\mathbb{P} (0\in \mathscr{C})<1$ \\[3mm]
		(d) $\mathbb{E}_0 [N_{0}]<\infty$ & 
		(e) $\mathbb{P}_0 (N_{0} = \infty)<1$ 
		&	(f) $\mathbb{P}_0 (N_{0} = \infty)=0$ \\[3mm]
		(g) $\mathbb{E} [M_0]<\infty$& (h) $\mathbb{P} (M_0 = \infty)<1$ 
		 &(i) $\mathbb{P} (M_0 = \infty)=0$
		\\[3mm]
	\end{tabular}
	\end{proposition}
	\textit{Remark:} {Statement \eqref{sparse} is only about the grain distribution and therefore does not depend on the Poisson intensity $u$.
 Hence all other statements are independent of $u$ as well. The statements $(a), (b)$ and $(c)$ are about the random sets $\mathscr{C}$, the statements 
 $(g), (h)$ and $(i)$ about the covering and $(d), (e), (f)$ about the corresponding random graph. 
	Recall that the grain distribution is sparse if one, and hence all, of the conditions in Proposition~\ref{Prop:one} hold for one, and hence all, values of  $u>0$.}
 \pagebreak[3]
	\medskip

 We say that the grain distribution is \emph{robust} if, for all $u>0$, the set $\mathscr{C}$ has an unbounded component, i.e.\ a component of infinite volume. 
 This is easily seen to be equivalent to the fact that,  for every intensity $u>0$, the graph $\mathscr G$ percolates, or in the Palm version there is a positive probability that there exists an infinite self-avoiding path in $\mathscr G$ starting at the origin. Conversely, the grain distribution is \emph{non-robust} if there exists $u_c>0$ such that for all $0<u<u_c$ every component of $\mathscr{C}$ is bounded.
 \medskip
 
 In the case that $C$ is a centred ball of random radius, Gou\'er\'e~\cite{Gouere2008} has shown that non-robustness is equivalent to sparseness , i.e. to the statements in Proposition~\ref{Prop:one}. This is however not true in the case of general convex grains. Teixeira and Ungaretti~\cite{Teixeira2017} have shown that for grains in $\mathbb R^2$  which are ellipses with a major axis of heavy tailed random length {with tail index $-\alpha$, for $1<\alpha<2$,} and a minor axis of unit length, for every $u>0$, we have $\mathscr{C}\not=\mathbb{R}^2$ even though  $\mathscr{C}$ has an unbounded component, almost surely.\smallskip\pagebreak[3]
 
 {Our main results give  criteria for a grain distribution to be  robust
 or non-robust, depending on a $d$-dimensional vector of diaemeters, which we now define.}

\begin{mydef}\label{def:diameters}
     Let $K\subset\mathbb R^d$ be a convex body, i.e. a compact convex set with nonempty interior. The \emph{first diameter}, or just \emph{diameter} of $K$, is defined as 
    \begin{equation*}
         D_K^{_{(1)}}:= \operatorname{diam}(K) = \max\{|x-y|\,:\, x,y \in K\}.
    \end{equation*}
    Let $p^{_{(1)}}_{K}$ be the orientation of $D_K^{_{(1)}}$, i.e. 
    {$p^{_{(1)}}_K=\tfrac{x-y}{|x-y|}$  }where $x,y$ are any (measurable) choice of maximizers in the definition of the diameter
    $D_K^{_{(1)}}$. We define the \emph{second diameter} as 
    \begin{equation*}  D_K^{_{(2)}}:=\operatorname{diam}\Bigl(P_{H_{p^{_{(1)}}_K}}(K)\Bigr), 
    \end{equation*} 
    where $P_H(B)$ is the orthogonal projection of $B\subset \mathbb{R}^d$ onto the linear subspace $H\subset\mathbb R^d$, and where \smash{$H_{p^{_{(1)}}_K}$} is the hyperplane perpendicular to $p^{_{(1)}}_{K}$. 
    Denote the orientation of $D_K^{_{(2)}}$ by $p^{_{(2)}}_{K}\!\in H_{p^{_{(1)}}_K}$ and let \smash{$H_{p^{_{(2)}}_K}$} be the hyperplane in \smash{$ H_{p^{_{(1)}}_K}$} perpendicular to $p^{_{(2)}}_{K}$.\medskip
    
    Iterating this procedure, given \smash{$H_{p^{_{(i)}}_K}$} for some $2\leq i <d$, the $(i+1)^{st}$ diameter is 
    \begin{equation*}
        D_K^{_{(i+1)}} := \operatorname{diam} \Bigl( P_{H_{p^{{(i)}}_K}}(K)\Bigr)
    \end{equation*}
    and we denote by $p^{_{(i+1)}}_K$ its orientation.
    By construction we have $D_K^{_{(i+1)}}\leq D_K^{_{(i)}}$ for all $i\in\{1,\dots,d-1\}$. 
    Although the sequence of diameters $D_K^{_{(1)}}, \ldots, D_K^{_{(d)}}$
    thus defined involves the possibly non-unique choice of orientations, our results do not depend on any of these choices.%
\end{mydef}

    We assume for our grain distribution that the tails of the random variables $D_C^{_{(i)}}$ are regularly varying with index $-\alpha_i$, i.e. 
	\begin{equation*}
	    \lim\limits_{r \rightarrow \infty} \frac{\mathbb{P}(D_C^{_{(i)}}\geq cr)} {\mathbb{P}(D_C^{_{(i)}}\geq r)} = c^{-\alpha_i} \text{ for all $c>1$.}
	\end{equation*}   
We include the case of bounded diameters $D_C^{_{(i)}}$, in which case we put $\alpha_i=\infty$. Note that, by definition, $\alpha_1\leq \alpha_2 \leq \cdots\leq \alpha_d$.\bigskip

{Our first main result, Theorem~\ref{Thm:one}, gives a \emph{universal criterion} for robustness of a grain distribution, which depends only on the tail indices of the individual diameters.  This is followed by Theorem~\ref{Thm:one'}, which  gives a \emph{specific criterion}, that relies on the diameters being sufficiently strongly correlated, which applies in cases not covered by the universal criterion. Our second main result, Theorem~\ref{Thm:two}, gives a sufficient condition for non-robustness for grain distributions in $d$ dimensions with $k$ `large' diameters and $d-k$ diameters of bounded length. This includes  
    a higher dimensional generalisation of the two-dimensional
     ellipses model studied in \cite{Teixeira2017}.
    Finally, we present some examples, including one that shows the existence of grain distributions where the 
    marginal distributions of the 
    diameters do not suffice to determine whether the grain distribution is robust or not, a phenomenon that to the best of our knowledge has not been observed in any other model so far.}

    \begin{theorem}\label{Thm:one}
        In the Poisson Boolean model given by a $d$-dimensional random convex body $C$ with rotation invariant distribution, containing a ball of fixed radius such that, for all $1\leq k\leq d$, the $k$-th diameter $D^{_{(k)}}_C$ has regularly varying tail with index $-\alpha_k$, we have that the grain distribution is \emph{robust} if
        \begin{enumerate}[(a)]
            \item there exists $1\leq k \leq d$ such that $\alpha_k < \min\{2k,d\}$.   
 \end{enumerate}
 \end{theorem}
       \pagebreak[3]

\noindent\textit{Remarks:}\ \\[-.5cm]
\begin{itemize}
\item The criterion in Theorem~\ref{Thm:one} only refers to the individual tail indices of the diameters. In particular, no further assumption is made on the \emph{joint distribution} of the diameters $D^{_{(1)}}_C, \ldots, D^{_{(d)}}_C$.
\item  Below we give an example of a grain distribution in dimension $d=3$ where information about the \emph{joint distribution} of diameters is needed to determine whether it is robust or not. This shows that outside of the universal criteria for robustness, further information is required to ensure the grain distribution is robust, which leads us to conjecture that Theorem~\ref{Thm:one} is in fact best possible universal criterion in dimension $d\ge3$.
\end{itemize}

  A specific criterion for robustness is based on the diameters being sufficiently positively correlated.      
           
        \begin{theorem}\label{Thm:one'}     
        Take the Poisson Boolean model given by a $d$-dimensional random convex body $C$ with rotation invariant distribution, containing a ball of fixed radius such that there exist $c>0$, positive integer {$k\leq d/2$} and $\gamma_k, \ldots, \gamma_{d-k} \ge 0$ with
        \[
            \mathbb{P}_C\big(D^{_{(j)}}_C
            \ge c (D^{_{(k)}}_C)^{\gamma_j}
            \mbox{ for all } d-k\ge j >k, D^{_{(k)}}_C \ge x\big)
        \]
        regularly varying in $x$ with index $-\gamma_k$. Then the grain distribution is robust if 
        \begin{enumerate}[(b)]
            \item $\gamma_k< 2k + \sum_{j=k+1}^{d-k} \gamma_j$.
        \end{enumerate}
        \end{theorem}
  
We now consider criteria for non-robustness. The following result is due to \cite{Gouere2008} and follows from an analysis of grains which are balls. 

\begin{proposition}\label{Prop:Gouere}
In the Poisson Boolean model given by a $d$-dimensional random convex body $C$ with rotation invariant distribution containing a ball of fixed radius, 
the grain distribution is non-robust
if $D^{_{(1)}}_C \in \mathcal{L}^d$.  
\end{proposition}

  Next, we consider a large class of grain distributions for which non-robustness can be shown under sufficient conditions on the tail exponents of the grains.

    \begin{theorem}\label{Thm:two}
    Take the Poisson Boolean model given by a d-dimensional random convex body $C$ with rotation invariant distribution and $D_C^{_{(\ell)}}$ regularly varying with index $-\alpha_{\ell}$ for all $\ell\in\{1,\dots,d\}$. 
    {Suppose that, for some $c>0$ and $\beta\ge0$, almost surely there exists  $k \in \{1,\dots,d -1\}$ such that 
    \begin{enumerate}[(c)]
            \item $0\le \beta < \min\big\{\frac{\alpha_k-2k}{2(d-k)}, \frac{\alpha_k}{2d}\big\}.$ 
        \end{enumerate} 
    and \smash{$D^{_{(k)}}_C> c D^{_{(1)}}_C$} and,  for all $d \ge j>k$, we have
    $D_C^{_{(j)}}\leq \frac1c \, (D_C^{_{(k)}})^{\beta}.$
Then the grain distribution is \emph{non-robust}.}
    \end{theorem}
    
\pagebreak[3]

\noindent\textit{Remarks:}\ 
\begin{itemize}
\item The assumptions of Theorem~\ref{Thm:two} imply that the product of diameters of $C$, and hence $\vol(C)$, has finite second moment. Indeed, if the assumption holds for $k$, then $D^{_{(j)}}_C\leq D^{_{(1)}}_C< \frac1c D^{_{(k)}}_C$ for $j\leq k$, so the product of diameters is at most $c^{-d}(D^{_{(k)}}_C)^{k+(d-k)\beta}$. By $(c)$ we have $2k+2(d-k)\beta<\alpha_k$, and moments of $D^{_{(k)}}_C$ of order smaller than $\alpha_k$ are finite. Moreover, $(c)$ implies $\alpha_k>2k$, and this condition is close to necessary for $\vol(C)\in\mathcal{L}^2$: If $C$ contains a ball of fixed radius, then $\vol(C)\in\mathcal{L}^2$ implies $\alpha_k \geq 2k$ for all $1\leq k\leq d$. However, $\alpha_k>2k$ for all $k$ does not imply $\vol(C)\in\mathcal{L}^2$. For example, the ellipsoids with strongly dependent axes below with $\gamma=\frac12$ and $\alpha\in(2,3]$ satisfy $\alpha_k>2k$ for all $k$, but $\vol(C)=cD^{3/2}\notin\mathcal{L}^2$. For this reason Theorem~\ref{Thm:two} also requires the upper bound on the diameters $D^{_{(j)}}_C$ for $j>k$.
    \item The condition in Theorem~\ref{Thm:one} implies $\vol(C)\not\in\mathcal{L}^2$ and $D_C^{_{(1)}}\not\in \mathcal{L}^d$, while the non-robustness condition in Theorem~\ref{Thm:two} implies $\vol(C)\in\mathcal{L}^2$ and the one in Proposition~\ref{Prop:Gouere}
    implies
    $D_C^{_{(1)}}\in \mathcal{L}^d$. 
{One might therefore be tempted to conjecture that $\vol(C)\not\in\mathcal{L}^2$ and 
    $D_C^{_{(1)}}\not\in \mathcal{L}^d$ is equivalent to robustness, but already \cite{Teixeira2017} provides an example of a non-robust grain distribution in $\mathbb R^2$ satisfying these conditions.}
\end{itemize}

We now describe several examples where we state criteria for different behaviour of the grain distributions, i.e.\ density, robustness and non-robustness. In our first and fourth example the criteria are sharp up to a boundary case where we expect that robustness depends 
on finer details of the distribution of $C$. In the second example Proposition~\ref{Prop:one} and Theorem~\ref{Thm:one} give criteria on density and robustness while in the third example Theorem~\ref{Thm:one'} give the criteria for robustness. The condition on the grain distribution being non-robust is given due to the known results on the Boolean model with balls as grains given in \cite{Gouere2008}, which is also mentioned in Proposition~\ref{Prop:Gouere} and Theorem~\ref{Thm:two}.

\subsubsection*{Ellipsoids with long and short axes.}

We sample a random radius $R>1$ from a distribution which has a tail distribution that is regularly varying with 
index~$-\alpha$. For integers $0\leq m \leq d$ we define a random ellipsoid $K$ with centre in the origin and $d-m$ long axes of length $R$ and $m$
short axes of length one. We let $C$ be the random convex set obtained by rotating $K$ about its centre by an independent 
uniform angle $\vartheta\in S^{d-1}$. Then
\begin{itemize}
\item $C$ is sparse if $\alpha > d-m$.
\item $C$ is robust if $\alpha < \min\{2(d-m),d\}$.
\item $C$ is non-robust if $\alpha > \min\{2(d-m),d\}$.
\end{itemize}
Observe that in the case $m=0$ the ellipsoid $C$ is a ball of random radius and in this case there exists no grain distribution that is both robust and sparse, as observed earlier
by Gou\'ere~\cite{Gouere2008}. The case $m=1$, $d=2$ corresponds to the case studied by Teixeira and Ungaretti~\cite{Teixeira2017} and we recover their result. Note that in all dimensions $d\ge 2$ whenever $0<m<d$ there exists $\alpha$ such that $C$ is both robust and sparse.

\subsubsection*{Ellipsoids with independent axes.}
We sample $d$ independent random radii $R_1,\ldots, R_d\ge1$ from distributions which have regularly varying 
tail distributions with index~$-\beta_i$. We define a random ellipsoid $K\subset \mathbb R^d$ with axes of length $R_1,\ldots, R_d$
and let $C$ be the random convex set obtained by rotating $K$ about its centre by an independent 
uniform angle $\vartheta\in S^{d-1}$. Then
\begin{itemize}
\item $C$ is sparse if $\beta_i > 1$ for all $1\leq i\leq d$.
\item $C$ is robust if there exists $1\leq i \leq d$ such that $\beta_i<2$.
\item $C$ is non-robust if $\beta_1 >d$.
\end{itemize}

\subsubsection*{Ellipsoids with strongly dependent axes.}

Let $0\leq \gamma \leq 1$ and $\alpha\ge 1$.  Let $D$ be Pareto distributed on $(1,\infty)$ with tail index $-\alpha$ and consider an ellipsoid $C\subset \mathbb{R}^3$ with diameters
\[
 D^{_{(1)}}=D, D^{_{(2)}}= D^{\gamma}, D^{_{(3)}}=1.
\]
Then, $D^{_{(1)}}$ is regularly varying with index $-\alpha$, $D^{_{(2)}}$ with index $-\alpha/\gamma$,
and
\begin{itemize}
    \item $C$ is sparse if $\alpha>1+\gamma$,
    \item $C$ is robust if $\alpha<\max\{2+\gamma,3\gamma\}$, 
    \item $C$ is non-robust if {$\alpha>\min\{3,2+4\gamma\}$}.
\end{itemize}

\subsubsection*{Random triangles.}

Let $R>1$ be random with regularly varying tail of index $-\alpha$, for $\alpha >0$. Take the \emph{right} triangle $K\subset 
\mathbb R^2$ such that $R$ is the length of the hypotenuse and $\mathrm{Vol}(K)=\frac14 R^{1+\beta}$, for some $\beta\in(0,1)$. 
Note that this describes the triangle uniquely up to symmetries. Now choose the origin uniformly as one of the corners of $K$
and let $C$ be the random set obtained by a uniform rotation via this point. 
Then 
\begin{itemize}
\item $C$ is sparse if $\alpha> 1+\beta$,
\item $C$ is robust if $\alpha < 2$,
\item $C$ is non-robust if $\alpha >2$.
\end{itemize}

\subsubsection*{Marginals of the diameter do not suffice.}

   Let $0<\gamma<\frac12$ and $2<\alpha<2+\gamma$.  Let $D$ be Pareto distributed on $(1,\infty)$ with tail index $-\alpha$ and $X\in\{H,T\}$ an independent fair coin flip. Consider an ellipsoid $C\subset \mathbb R^3$ with
    \begin{align*}
    D^{_{(1)}}=D, D^{_{(2)}}=D^{_{(3)}}=1
    & \text{ if $X=H$,}\\
    D^{_{(1)}}=D^{_{(2)}}=D^{\gamma}, D^{_{(3)}}=1
    & \text{ if $X=T$.}
    \end{align*}
    Then, $D^{_{(1)}}$ is regularly varying with index $-\alpha$ and $D^{_{(2)}}$ with index $-\alpha/\gamma$ and the Boolean model with this grain distribution is non-robust. \smallskip
    
    If instead
    we consider an ellipsoid with
\begin{align*}
    D^{_{(1)}}=D, D^{_{(2)}}= D^{\gamma}, D^{_{(3)}}=1
    & \text{ if $X=H$,}\\
    D^{_{(1)}}=D^\gamma, D^{_{(2)}}= D^{_{(3)}}=1
    & \text{ if $X=T$,}
    \end{align*}
    then the diameters of the resulting grain have the same marginal distribution, but the Boolean model with this grain is now robust.
\pagebreak[3]

\subsection*{Structure of the paper}
The proofs of Proposition~\ref{Prop:one} and Theorem~\ref{Thm:one}, \ref{Thm:one'} and~\ref{Thm:two} are given in Section~\ref{Sect:proofs}.
More details and proofs relating to the examples above can be found in Section~\ref{Sect:examples}. Throughout the article, we will allow ourselves a short abuse of notation and refer to \emph{hyperrectangles} (i.e. generalisations of rectangles to higher dimensions) simply as rectangles or boxes, regardless of the choice of dimension.

    \section{Proofs}\label{Sect:proofs}
    In order to prove our result we first formalise our setup. Denote by $\mathcal C^d$ the space of convex bodies in $\mathbb R^d$ with the Hausdorff metric. Recall that we assume that, for some fixed $\epsilon>0$, the $\epsilon$-interior of  $C$
    is nonempty almost surely. We assume that $\mathbb{P}_{C}$ is a law on
    $\mathcal C^d \times \mathbb R^d$ such that, for 
    $\mathbb{P}_{C}$-almost every $(C,m)$ the point $m$ is in the $\epsilon$-interior of  $C$. We further assume that $\mathbb{P}_{C}$ is invariant under simultaneous rotations of 
    $C$ and $m$ about the origin.
 %
We now define the Poisson Boolean base model, which we use in our proofs.

    \begin{mydef}  
    The \textit{Poisson Boolean base model} is the Poisson point process on $$\mathcal{S}:=\mathbb{R}^d\times \big(\mathcal C^d\times \mathbb{R}^d \big) $$
    with intensity $$  u \, \lambda \otimes\mathbb P_C $$
    where $\lambda$ is the Lebesgue measure. The corresponding point process is denoted by $\mathcal{X}$. We call any $\textbf{x}\in\mathcal{X}$ a \emph{vertex} and its first component~$x\in\mathbb R^d$
    its \emph{location}. By $\mathscr P$ we denote the Poisson point process of locations. The second component is denoted by $(\tilde C_x, m_x)$ and $\tilde C_x$ is called the \emph{grain at $x$}. We denote its diameters by $D_x^{_{(1)}},\dots,D_x^{_{(d)}}$
    and the corresponding directions by $p_x^{_{(1)}},\dots,p_x^{_{(d)}}$.
    Then we set
    $$C_{x}:=x+\tilde C_x$$
    and define
    $$\mathscr C := \bigcup_{x\in \mathscr P }  C_x,$$
    which is a representation of the \emph{Poisson Boolean model}. 
    When considering two vertices $\textbf{x},\textbf{y} \in \mathcal{X}$, we say they are connected by an edge, if and only if the sets 
    \smash{$C_{x}$ and $C_{y}$} intersect. We denote this by writing $\textbf{x}\sim \textbf{y}$. If $\textbf{x}$ and $\textbf{y}$ are connected through  $n$ edges we write $\textbf{x}\overset{n}{\sim}\textbf{y}$ and say $\textbf{x}$ and $\textbf{y}$ are connected by a path of length exactly $n$.
    \end{mydef}

    If the maximisers in Definition~\ref{def:diameters} are not unique, the orientations $p_x^{_{(1)}},\dots,p_x^{_{(d)}}$ can be any measurable choice, and our results hold for every such choice. Indeed, the proofs use the orientations only to construct boxes with edges parallel to $p_x^{_{(1)}},\dots,p_x^{_{(d)}}$ which contain $C_x$ or are contained in $C_x$, see Sections~\ref{Sect:density}, \ref{Sect:robustness} and~\ref{Sect:non-robust_ellipses}, and to define the sets $O_i$ in Section~\ref{Sect:robustness}. The properties of these objects that we use hold for every choice of the orientations. Moreover, since the distribution of $\tilde C_x$ is rotation invariant, $\tilde C_x$ has the same distribution as $\rho\tilde C_x$, where $\rho$ is a uniformly distributed rotation independent of $\tilde C_x$. The rotation $\rho$ maps the boxes constructed for $\tilde C_x$ to boxes for $\rho\tilde C_x$ with edges parallel to $\rho p_x^{_{(1)}},\dots,\rho p_x^{_{(d)}}$. Hence, for any choice of the orientations, we may assume that these directions are uniformly distributed and independent of the diameters.

    \subsection{Criteria for density}\label{Sect:density}
In all our proofs we use the notation
$B_\varepsilon(x):= \{y\in\mathbb{R}^d \colon |y-x|\leq\varepsilon\}$. We also write $\eta K:= \{ \eta x \colon x \in K\}$ for the \emph{blow-up} of the set $K\subset\mathbb{R}^d$ about the origin by a factor~$\eta>0$. Let \smash{$(e_i)_{i\in\{1,\dots,d\}}$} be the canonical basis of $\mathbb{R}^d$ and for $\vartheta\in\mathbb{S}^{d-1}$ let {\color{black}$\rot_{\vartheta}$ be an arbitrary but fixed rotation such that $\rot_{\vartheta}(e_1)=\vartheta$. Throughout the proofs we use $c\in(0,\infty)$ as a generic constant which may change its value at every inequality, but is always finite and may depend on $d$, $k$, $\alpha_k$ and on $\varepsilon>0$ appearing in the Potter bounds only. It depends also on $\epsilon>0$ where $\epsilon$ is the radius of the ball which is completely included in $C$.} \smallskip

    \begin{proof}[{Proof of Proposition \ref{Prop:one}}]
This proof does not require rotation invariance. By the mapping theorem, see e.g. Theorem~5.1 in~\cite{Last2017}, the point process given by the points $x+m_x, x\in\mathscr P$ is again a homogeneous Poisson point process. We may therefore assume, without loss of generality, that there is a fixed $\epsilon>0$ such that $B_\epsilon(x)\subset C_x$ for all $x\in\mathscr P$.%
\medskip

    $(a) \Leftrightarrow (b)$ is clear by ergodicity. 
	$\eqref{sparse} \Rightarrow (a) $ follows because, by Section 2 of \cite{Hall1985}, we have $
			\mathbb{P}(\mathrm{Vol}(\mathbb{R}^d \setminus \mathscr{C})=0) =1$ if and only if $\mathbb{E} [\mathrm{Vol}(C)]=\infty$.
   To also get $ \neg  \eqref{sparse} \Rightarrow \neg  (a) $ from this statement we apply it to 
   $\eta C$, for some fixed $0<\eta<1$, and get that 
   $\mathbb{E} [\mathrm{Vol}(C)]=\infty$ implies $\mathbb{E} [\mathrm{Vol}(\eta C)]
   = \eta^d \mathbb{E} [\mathrm{Vol}(C)]=\infty$ and hence 
   \begin{equation*}
			\mathbb{P}\bigg(\mathrm{Vol}\Big( \mathbb R^d \setminus \bigcup_{x\in\mathscr{P}} (\eta C)_x \Big) =0
   \bigg) =1.
   \end{equation*}
   This event implies $\bigcup_{x\in\mathscr{P}} C_x =\mathbb R^d$. Indeed, if $y\not\in \bigcup_{x\in\mathscr{P}} (\eta C)_x$ there exists a sequence of points $(x_n)$ in $\mathscr{P}$ such that
   the distance of $y$ and $(\eta C)_{x_n}$ goes to zero. If we replace $(\eta C)_{x_n}$ by $C_{x_n}$,
the distance to $y$ is reduced by at least the fixed amount $\epsilon(1-\eta)>0$. {Indeed, by convexity of $C$ and $B_\epsilon(0)\subset C$ we have $\eta C+B_{(1-\eta)\epsilon}(0)=\eta C+(1-\eta)B_\epsilon(0)\subset C$, and hence $\operatorname{dist}(z,C)\leq \max\{\operatorname{dist}(z,\eta C)-(1-\eta)\epsilon,0\}$ for all $z\in\mathbb{R}^d$.} Hence there exists~$n$ with $y\in C_{x_n}$. This completes the proof of $ (a) \Leftrightarrow \eqref{sparse} $.
    The implication $(c) \Rightarrow (a) $ is immediate from $\{\mathscr{C} = \mathbb{R}^d\}\subset \{0\in\mathscr{C}\} $.
    \smallskip\pagebreak[3]
    
  %
  To show that $\eqref{sparse} \Leftrightarrow (g)$ we  use Campbell's theorem to calculate
  \begin{align*}
        \mathbb{E}[M_0] &= u \int  \mathbb{P}(0\in x+C) \, \text{d}\lambda(x)
        = u \int  \mathbb{P}(x\in C) \, \text{d}\lambda(x) =u\mathbb{E}[\mathrm{Vol}(C)],
  \end{align*}
  which readily implies the equivalence. 
  \smallskip

  To show that $(d)\Rightarrow \eqref{sparse}$ observe that under the Palm distribution  $\mathscr{P}(C_0\setminus\{0\})=k$ implies that $N_0\geq k$.  Hence $\mathbb{E}_0[N_0] \geq \mathbb{E}_0[\mathscr{P}(C_0)-1] =u \mathbb{E}[\lambda(C)]$, from which the implication follows. 
  For the implication $\eqref{sparse}\Rightarrow (d)$ we enlarge the sets $C_x$, for each $x\in\mathscr P$, to become rectangles $R_x$ taken as the sum of $x$ and the cartesian product of the intervals $[-D_x^{_{(i)}},D_x^{_{(i)}}]$ with respect to the orthonormal basis given by the 
directions~$p^{_{(i)}}_{C_x}$. The enlargement increases the volume of the sets by no more than a constant factor. Indeed, if $R$ is the rectangle constructed from $C$ we have $\vol(R)= 2^d \prod D^{_{(i)}}$. As the convex hull of the points, which define the length  and  orientation of the diameter, is contained in $C$ we can lower bound its volume iteratively via the formula of the volume of hyperpyramids and get
    $$\vol(R) \leq (2^d\cdot d!) \vol(C).$$ 
  We write
  $\tilde{N}_A$ and $\tilde{N}_0$ for the degrees defined like ${N}_A$ and ${N}_0$ but with respect to the enlarged sets. 
  By assumption we have $x+[-\epsilon/\sqrt{d}, \epsilon/\sqrt{d}]^{d}\subset R_x$. As $C_x\subset R_x$ we have
  $N_0\leq \tilde{N}_0$ and we can bound the expected degree of the origin from above,
  \begin{equation*}
      \mathbb{E}_{0}[N_0]\leq \mathbb{E}_0[\tilde{N}_0 ]\leq \mathbb{E}_0 \bigl[\sum\limits_{x\in(\epsilon/\sqrt{d})\mathbb{Z}^d} \mathbbm{1}_{{R_0^+}}(x)\tilde{N}_{B_{\epsilon}(x)}\bigr],
  \end{equation*}
 {where $R_0^+:=R_0+B_{\epsilon}(0)$ and in the second inequality we use that every point of $R_0$ is within distance $\epsilon/2$ of a point of $(\epsilon/\sqrt{d})\mathbb{Z}^d$, which then lies in $R_0^+$, and hence the balls of radius $\epsilon$ centred in the points of $(\epsilon/\sqrt{d})\mathbb{Z}^d \cap R_0^+$ cover $R_0$.} 
Now \smash{$\mathbbm{1}_{{R_0^+}}(x)$} and $\tilde{N}_{B_{\epsilon}(x)}-\mathbbm{1}_{R_0\cap B_{\epsilon}(x) \not= \emptyset}$ are independent and we get an upper bound of
  \begin{equation*}
   \mathbb{E}_{0}[N_0]\leq  \mathbb{E}_0\bigl[\#((\epsilon/\sqrt{d})\mathbb{Z}^d\cap {R_0^+})\bigr]\mathbb{E}\bigl[\tilde{N}_{B_{\epsilon}(0)}+1\bigr].
  \end{equation*}
{As $B_{\epsilon}(0)\subset R_0$ we have $R_0^+\subset 2R_0$, and every cube of side length $\epsilon/\sqrt{d}$ centred in a lattice point of $R_0^+$ is contained in $R_0^+ + B_{\epsilon/2}(0)\subset 3R_0$. Counting these disjoint cubes yields that the number of lattice points in $R_0^+$ is bounded from above by $c(d,\epsilon)\vol(R_0)$, for a constant $c(d,\epsilon)$ which only depends on the dimension and $\epsilon$, see also \cite{Freyer2022} for sharper bounds of this type.} 
  Looking now at the last factor 
  we use
  \begin{equation*}
      \mathbb{E}\bigl[\tilde{N}_{B_{\epsilon}(0)}\bigr] =u\int\mathbb{P}_0(B_{\epsilon}(0) \cap (x+R) \neq \emptyset) \,\text{d}x = u\int\mathbb{P}(B_{\epsilon}(x) \cap R \neq \emptyset) \,\text{d}x.
  \end{equation*}
 As $\{B_{\epsilon}(x) \cap R \neq \emptyset\}\subset \{x\in 2R \}$ we can bound the previous term from above by 
 $$u \mathbb{E}[\vol(2R)]= u 2^{d}\,\mathbb{E}[\vol(R)] \leq u \,2^{2d}\cdot d! \mathbb{E}[\vol(C)].$$
 Hence we get that $\mathbb{E}_{0}[N_0]$ is finite if $\mathbb E[$Vol$(C)]$ is finite. 
\smallskip
  
  The equivalence of  $(g), (h), (i)$ and $\eqref{sparse}$ follows because
		$M_0$ is Poisson distributed with parameter $u\mathbb E[\vol(C)]$ and therefore finite almost surely if and only if the parameter is finite. The implications $(d)\Rightarrow (f)\Rightarrow (e)$ are trivial.  The implications $ (f) \Rightarrow (i)$ and $(e) \Rightarrow (h)$ follow as $N_0 \geq M_0-1$ under the Palm distribution.  Finally, to show the implication $\eqref{sparse}\Rightarrow (c)$ we note that $\mathbb{P}(0\in\mathscr{C})= \mathbb{P}(M_0 \geq 1)<1$ because $M_0$ is Poisson distributed with parameter $u\mathbb E[\vol(C)]<\infty$.
    \end{proof}

    \subsection{Universal criteria for  robustness 
    }\label{Sect:robustness}
    We now prove Theorem~\ref{Thm:one}. The idea of the proof is to construct,  for every $u>0$, an infinite self-avoiding path in $\mathscr G$. The path is constructed such that the convex bodies attached to the vertices on the path are growing faster than a given increasing threshold sequence. Note that without loss of generality we can assume that $\alpha_k \geq k$, for all $k\in\{1,\dots,d\}$, since otherwise $\vol(C)\not\in\mathcal{L}^1$ and by Proposition \ref{Prop:one} there is nothing left to show. 

    \begin{myrem} In this and the following section we bound the set of orientations which result in an intersection of two convex bodies using the following known geometric property. Consider two points $x,y\in\mathbb{R}^d$ at distance $a:=|x-y|$. Let $A$ be a $(d-1)$-dimensional convex set which contains $y$,lies in the hyperplane through $y$ perpendicular to $x-y$, and has $\operatorname{diam}(A)<a$. Denoting by $\lambda_{d-1}$ the $d-1$ dimensional Lebesgue measure, the volume of the set of orientations of (infinitely long) lines that go through $x$ and intersect $A$ can be bounded from above and below by constant multiples (depending only on $d$) of \smash{$\tfrac{\lambda_{d-1}(A)}{a^{{d-1}}}$}.
    \label{rem:winkel}
    \end{myrem}

As in the proof of Proposition~\ref{Prop:one} the point process given by the points $x+m_x, x\in\mathscr P$ is a homogeneous Poisson point process. As done there, we also replace the independent attachments $(C,m)$ by $(C-m,0)$ so that the new grains contain a ball of radius $\epsilon$ around the origin and, by the rotation invariance assumption, their distribution is invariant under rotations around the origin. Reverting to the original notation we may assume henceforth that $B_\epsilon(x) \subset C_x$ for all $x\in\mathscr P$.
\smallskip

For $k\in\{1,\dots,d-1\}$ with $\alpha_k\neq k$ define the increasing threshold sequence $(f_{_{n}})_{n\in\mathbb{N}}$~by
        \begin{equation*}
        f_{_{n}}:=\bigl(f_{_{n-1}}\bigr)^{\frac{\min\{d-k,k\}}{\alpha_k-k}-\epsilon},\, \mbox{ for } n\in\mathbb{N},
        \end{equation*} 
        where $f_{_{0}}>1$ can be chosen arbitrarily and will be set large later in the proof, and $0<\epsilon<\frac14$ is such that the exponent in this sequence is strictly bigger than~1 in order to guarantee that the sequence is increasing. This is possible by our assumption that $\alpha_k< \min\{2k,d\}$.  Note that we can without loss of generality take this $\epsilon$ to be the same as in the requirement that $B_{\epsilon}(0)\subset C$, by replacing the larger of the two with the smaller if necessary. We comment on the {cases $k=d$ and $\alpha_k=k$} at the end of the~proof. 
       \medskip
       
  As we are interested in an infinite self-avoiding path in $\mathscr{G}$ we define a sequence $(A_n)_{n\in\mathbb{N}}$ of events which describe the infinite path. We start with the event $A_0$ defined as 
       $$A_0:=\Big\{ \exists \textbf{x}_0\in\mathcal{X}\cap B_{f_{_{0}}}(0)\times \mathscr C\,:\, D_{x_0}^{_{(k)}}\geq 2^{2(d-1)+\epsilon}f_{_{0}} \Big\}.$$
    
\begin{lemma}\label{Lemma:one}
There exists $c(\varepsilon,d)>0 $ such that
\begin{equation*}
    \mathbb{P}(A_0)  \geq 1- \exp\bigl(-uc(\varepsilon,d)f_{_{0}}^{d-\alpha_k-\varepsilon}\bigr).
\end{equation*}
\end{lemma}
\begin{proof}
     As $D^{_{(k)}}_{\phantom{x}}$
        is regularly varying with index $-\alpha_k$, the probability of $A_0$ is
        \begin{align}
            \mathbb{P}(A_0) &= 1- \exp\Bigl(-u\int\limits_{B_{f_{_{0}}}(0)} \mathbb{P}(D^{_{(k)}}_{\phantom{x}}\geq  2^{2(d-1)+\epsilon}f_{_{0}}) \, \text{d}\lambda(x)\Bigr)  \notag \\
            & \geq 1- \exp\bigl(-uc(\varepsilon,d)f_{_{0}}^{d-\alpha_k-\varepsilon}\bigr), \label{a0}
        \end{align}
        where we note that $\varepsilon>0$ was chosen such that $d-\alpha_k-\varepsilon >0$. In the last inequality we have applied the Potter bounds (see \cite[Prop.\ 1.4.1]{Kulik2020}).
\end{proof}
   
We define in the following some quantities that are useful to {construct} a self-avoiding path in $\mathscr{G}$ and show that this {construction yields an infinite path with positive probability.}%
\medskip

For $\textbf{x}\in\mathcal{X}$ we define the set 
        \begin{equation*}
            O_{_{i}}(\textbf{x}):= \Biggl\{\begin{array}{c}
        y\in\mathbb{R}^d: |x-y| \in [\tfrac12  f_{_{i}}, f_{_{i}}], \measuredangle\big(x-y,  v\big)>  \varphi,\\
        \text{for } v\in\operatorname{span}\{p_{x}^{_{(j)}}: 1\leq j\leq k\} \end{array} \Biggr\},
        \end{equation*}
        where $\varphi=2^{-(d+1)}(d+1)^{-1}$ and $\measuredangle(x,y)$ denotes the angle between vectors $x,y\in\mathbb{R}^d$. 
        For $y\in\mathbb{R}^d$ we define 
        \begin{equation*}
            B^{*}_{f_{_{i-1}}}(\textbf{x},y) := P_{H_{x-y}}\Big(C_{x}\cap B_{f_{_{i-1}}}(x)\Big),
        \end{equation*}
        where $H_{x-y}$ is the hyperplane perpendicular to $x-y$ with \smash{$|H_{x-y} \cap \partial B_{f_{_{i-1}}}(x)| =1$} and \smash{$\dist(H_{x-y}, x)<\dist(H_{x-y},y)$}, i.e. the hyperplane that touches the ball of radius~$f_{_{i-1}}$ with $x$ as the centre, see Figure~\ref{Bstern} \vpageref{Bstern}. 
        \begin{figure}[ht]
           \begin{minipage}[b]{1\linewidth}
              \includegraphics[width=\linewidth]{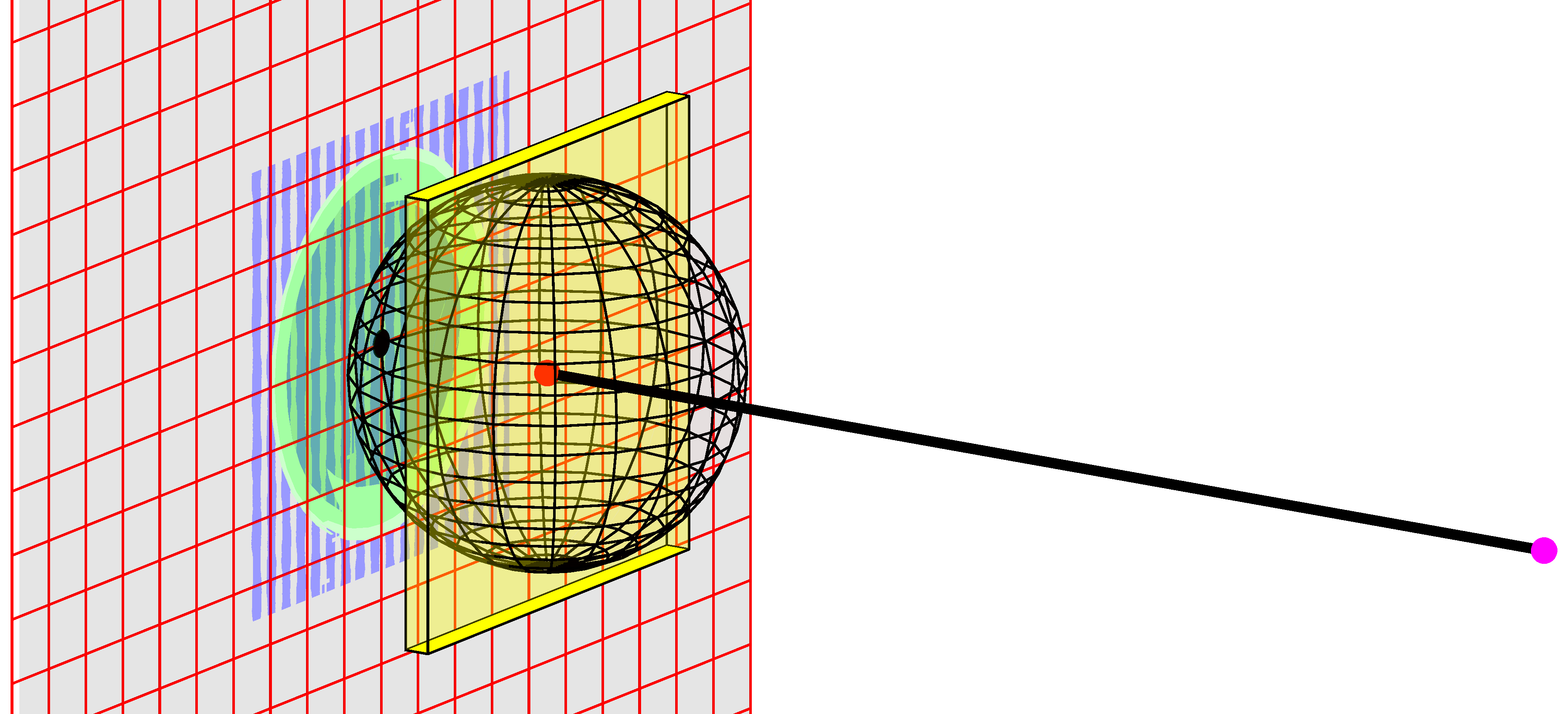}
              \caption{The various sets and  their relationships from the proof of Theorem~\ref{Thm:one}: The point $y$ is in pink, $x$ in red. The yellow area is $C_x$. The orthogonal projection of $C_x$ is dark blue and $H_{x-y}$ is the grey plane with the red grid. \smash{$B^{*}_{^{f_{_{i-1}}}}(\textbf{x},y)$} is lime green. Drawn in black are the line with orientation $x-y$, the point lying in \smash{$H_{x-y} \cap \partial B_{f_{_{i-1}}}(x)$} and the ball $B_{f_{_{i-1}}}(x)$ containing the location $x$ of the vertex $\textbf{x}$ corresponding to  $C_x$.}
              \label{Bstern}
        \end{minipage}
        \end{figure}
        \smallskip
        
        For $n\in\mathbb N$, define the events
        \begin{align*}
            A_n  := \Bigg\{ \begin{split}\exists &{\textbf{x}_0},\ldots, \textbf{x}_n\in\mathcal{X}\colon \textbf{x}_i\neq\textbf{x}_m \text{ for all } i\neq m, D_{x_i}^{_{(k)}}\geq 2^{2(d-1)+\epsilon}f_{_{i}},\\ & x_i\in O_i(\textbf{x}_{i-1})\text{ and }C_{x_i}\cap B^{*}_{f_{_{i-1}}}(\textbf{x}_{i-1},x_i)\neq \emptyset \text{ for all } {0\leq i}\leq n \end{split}\Bigg\},
        \end{align*}
        {where we formally define $O_0(\textbf{x}_{-1})=B_{f_0}(0)$ and
        $B^{*}_{f_{_{-1}}}(\textbf{x}_{-1},x_0)=\mathbb{R}^d$ to make this consistent with our earlier definition of $A_0$.}
        Roughly speaking, $A_n$ is the event that we find a path of length $n$ with the properties that, for every $i\in\{1,\ldots,n\}$, 
        \begin{itemize}
            \item the first $k$ diameters of the convex body $C_{x_i}$ do not fall below the threshold \smash{$2^{2(d-1)+\epsilon}f_{_{i}}$}. 
            It can easily be checked that a convex body with diameters \smash{$D^{{(1)}},\dots,D^{{(d)}}$} contains a rectangle with edges parallel to the orientations $p^{{(1)}},\dots,p^{{(d)}}$ of the diameters and edge lengths $2^{-2(d-1)}D^{{(1)}}, \dots ,2^{-2(d-1)}D^{{(d)}}$. 
            \item $x_i$ does not lie ``too close'' to the affine subspace through $x_{i-1}$ spanned by the orientations \smash{$p_{x_{i-1}}^{_{(1)}},\dots,p_{x_{i-1}}^{_{(k)}}$} of the big diameters of $\textbf{x}_{i-1}$.  More precisely, we require the angle between any spanning vector of this subspace through $x_{i-1}$ and the vector $x_{i}-x_{i-1}$ 
            to be larger than $\varphi$. This is 
            in order to keep the condition on the orientation of $\textbf{x}_i$ from becoming too restrictive when we formulate the requirement that $C_{x_{i}}$ intersects $C_{x_{i-1}}$. Finally,
            \item the convex body $C_{x_i}$ intersects a part of a hyperplane ``behind'' the convex body~$C_{x_{i-1}}$, see Figure \ref{Bstern}.
        \end{itemize}
       In addition to these restrictions that ensure that the points and their respective convex bodies are sufficiently close, suitably aligned  and large enough to keep the chain of intersections going, $A_n$ also gives that
        \begin{itemize}
            \item the distance between points $x_{i-1}$ and $x_i$ is at least  $\frac12 f_{_{i}}$ and $x_i\in B_{2f_{_{i}}}(0)$.
        \end{itemize}
        This last property allows us to search for each point $x_i$ in an annulus disjoint from those of the previous points $x_1, \ldots, x_{i-1}$.
        \medskip

  \begin{lemma}\label{Lemma:two}
       There exists $c>0$ such that, for $n\in\mathbb{N}$,
        \begin{equation*}
            \mathbb{P}[A_{n+1}^{c}|A_n] \leq \exp\bigl(  -u c f_{_{n}}^{\epsilon(\alpha_k -k)/2}\bigr).
        \end{equation*}
  \end{lemma}
  \begin{proof}
         Let $\mathfrak F_n$ be the $\sigma$-algebra
        generated by the restriction of the Poisson Boolean base model to the points with locations in $B_{2f_{_{n}}}(0)$. Then $A_n$ is 
         $\mathfrak F_n$-measurable and we find $\textbf{x}_{n+1}$, 
assuming that $\textbf{x}_1,\ldots, \textbf{x}_n 
        \in\mathcal{X}$ have been found satisfying $A_n$, such that
         $\textbf{x}_1,\ldots, \textbf{x}_{n+1} 
        \in\mathcal{X}$ satisfy $A_{n+1}$ with a conditional probability bounded from below. Figures \ref{d3m1} and \ref{d3m2} are sketches of the $(n+1)^{st}$ step in the construction.
        \begin{figure}
           \begin{subfigure}[t]{.495\linewidth} 
              \includegraphics[width=\linewidth]{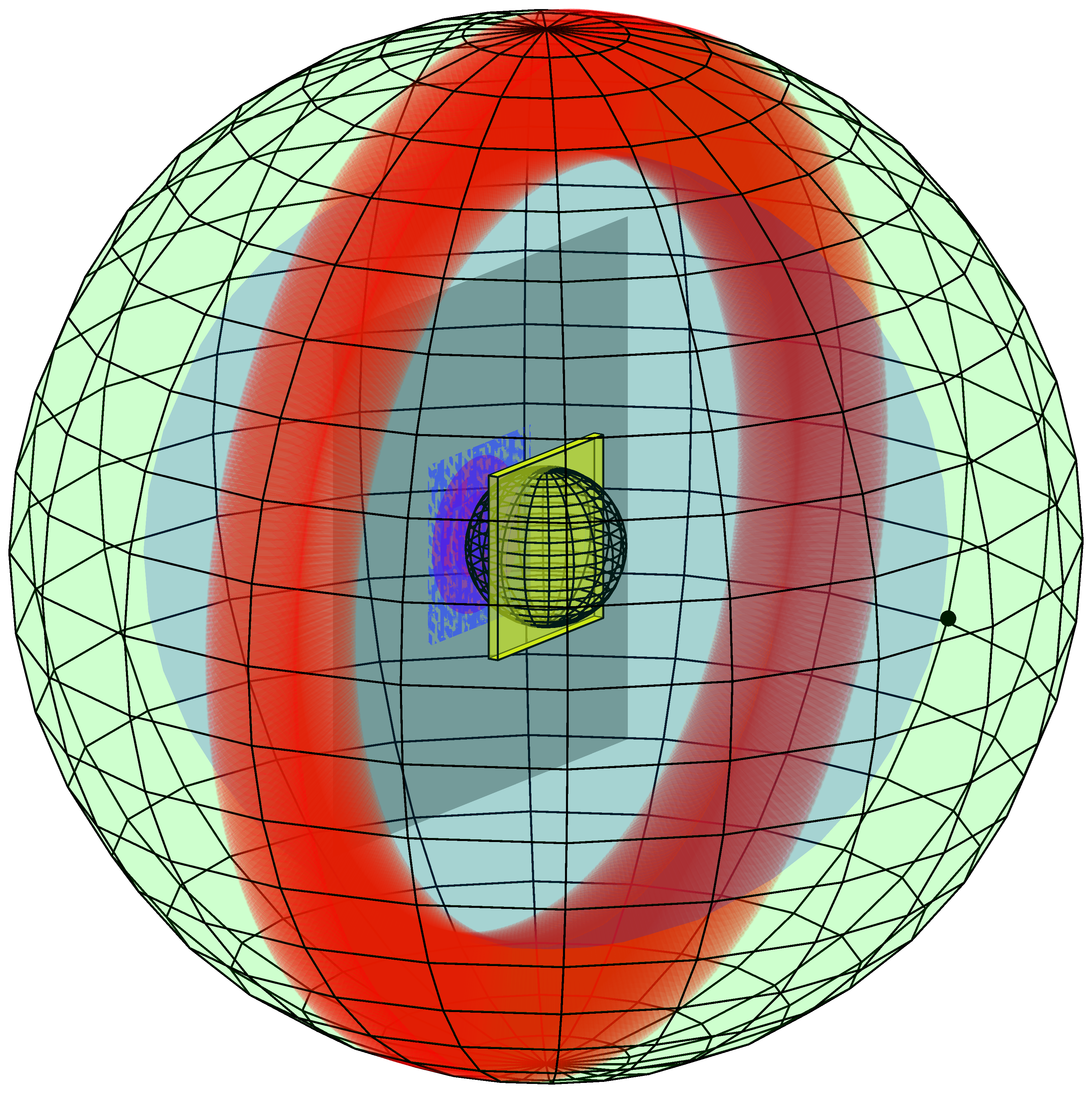}
              \subcaption{$i$-th step of the construction in $\mathbb{R}^3$ with \\$k=2$.}
            \label{d3m1}
           \end{subfigure}
           \hfill
           \begin{subfigure}[t]{.495\linewidth}
              \includegraphics[width=\linewidth]{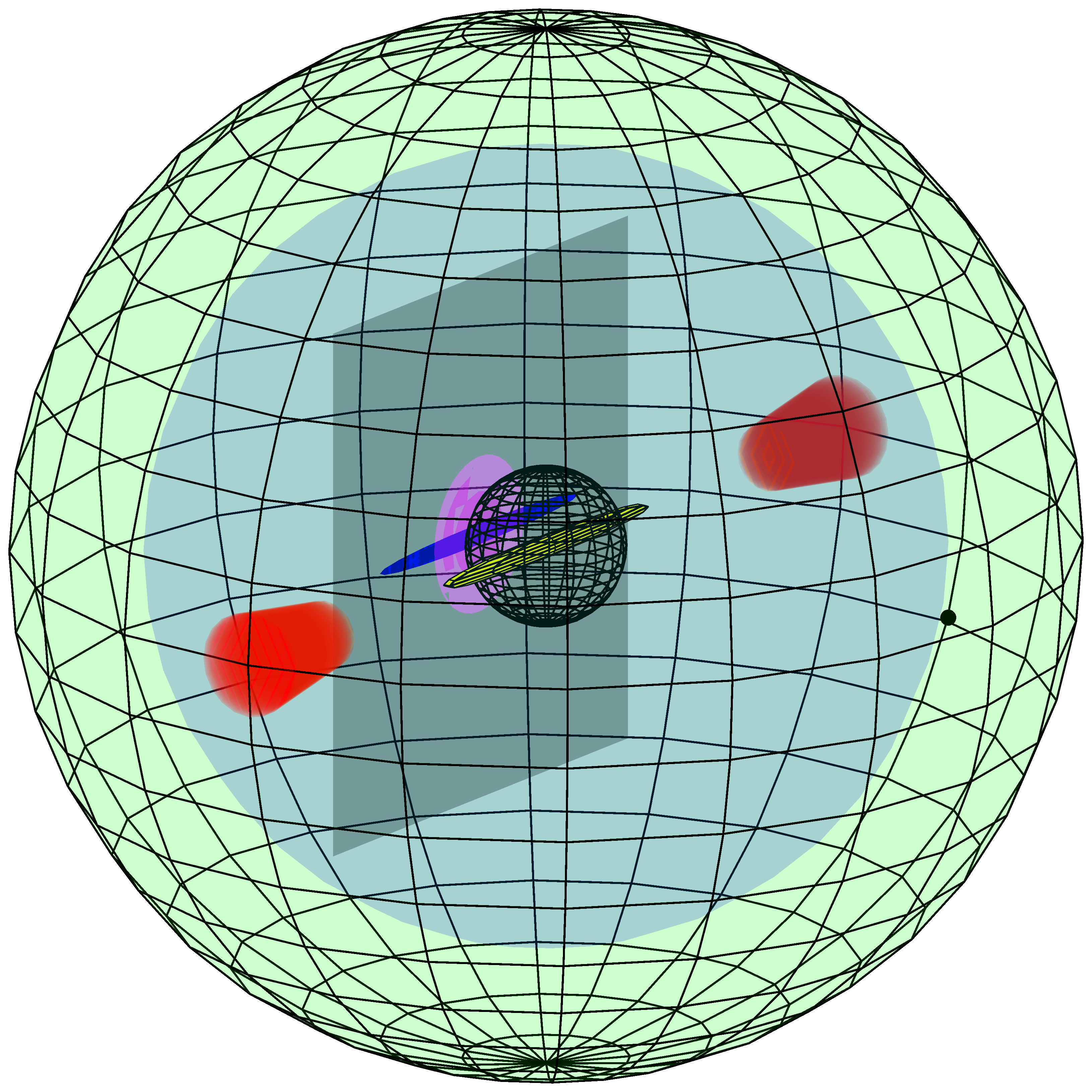}
              \subcaption{$i$-th step of the construction in $\mathbb{R}^3$ with \\$k=1$.}
              \label{d3m2}
           \end{subfigure}
           \caption{The recursive construction of the proof of Theorem~\ref{Thm:one}. The convex set $C_{x_{i-1}}$ is in yellow. Mint green is used for \smash{$B_{^{f_{_{i}}}}(x_{i-1})\setminus B_{^{f_{_{i}}/2}}(x_{i-1})$}; 
           i.e. the possible area for $x_i$.  The part in which $x_i$ is not permitted to be is in red. Grey is used for $H_{x_{i-1}-x_i}$, with dark blue for the orthogonal projection of $C_{x_{i-1}}$. In pink, the orthogonal projection of $B_{f_{_{i-1}}}(x_{i-1})$. Finally, the black point is a possible position for $x_i$.}     
           \label{construction}
        \end{figure}    

  On the event $A_n$, { we choose $\tilde{\textbf{x}}_n$ to be the 
lexicographically first vertex among all possible last vertices of paths satisfying 
the conditions in $A_n$. This is well-defined since, on $A_n$, there are 
almost surely finitely many such vertices. Moreover, $\tilde{\textbf{x}}_n$ is 
$\mathfrak{F}_n$-measurable. We therefore get}
		\begin{align*}
			\mathbb{P}[A^c_{n+1} \,\mid\, \mathfrak F_n] \leq \mathbb{P}\bigl[
    & \text{ there exists no }  \textbf{x} \in \mathcal{X} \mbox{ with } x\in B_{2f_{_{n+1}}}(0) \setminus
   B_{2f_{_{n}}}(0) \mbox{ such that }\\
   & 
   D_x^{_{(k)}}\geq 2^{2(d-1)+\epsilon}f_{_{n+1}}, \, x\in O_{n+1}(\tilde{\textbf{x}}_{n})\text{ and }C_x\cap B^{*}_{f_{_{n}}}(\tilde{\textbf{x}}_{n},x)\neq \emptyset
   \,\mid\, \mathfrak F_n \bigr].
   \end{align*}
 Recall that $\rot_{\vartheta}$ is a rotation such that $\rot_{\vartheta}(e_1)=\vartheta$. For our purpose it matters only that the shortest diameter is suitably oriented after this rotation.
    For $x,y\in\mathbb{R}^d$ we define $\rho_{x,y}:=\tfrac{x-y}{|x-y|}\in \mathbb{S}^{d-1}$ as the orientation of the vector $x-y$. With that we define   
    \begin{align*}
        Q_{f_{_{n}}}(x,y) := \rot_{\rho_{x,y}}\Bigl(\operatorname{conv}\bigl(&\bigl\{-\tfrac{\epsilon}{2}e_i,\tfrac{\epsilon}{2}e_i\,: \, 1\leq i \leq d-k\bigr\}\\ &\cup \bigl\{-\tfrac{f_{_{n}}}{2}e_m,\tfrac{f_{_{n}}}{2}e_m\, : \, d-k+1\leq m\leq d\bigr\}\bigr)\Bigr) +x
    \end{align*}
     with $\operatorname{conv}(A)$ being the convex hull of the set $A$ 
     and define
     \begin{equation*}
         Q^*_{_{f_{_{n}}}}(x,y):= P_{H_{x-y}}\bigl(Q_{f_{_{n}}}(x,y)\bigr).
     \end{equation*}
    Choosing $y=\tilde x_n$ the choice of \smash{$Q_{f_{_{n}}}(\tilde{x}_n,x)$} allows us to 
    obtain a lower bound for the probability that  $D^{_{(k)}}_x\geq 2^{2(d-1)+\epsilon}f_{_{n+1}}$ and \smash{$C_x\cap B^*_{{f_{_{n}}}}(\tilde{\textbf{x}}_n,x)\not=\emptyset$}, uniformly across all $C_x$ satisfying  $D^{_{(k)}}_x\geq 2^{2(d-1)+\epsilon}f_{_{n+1}}$, for all \smash{$x \in O_{n+1}(\tilde{\textbf{x}}_n)$}, by replacing  \smash{$B^*_{{f_{_{n}}}}(\tilde{\textbf{x}}_n,x)$} with the set $Q^*_{{f_{_{n}}}}(\tilde{x}_n,x)$. To see why, note that $Q^*_{_{f_{_{n}}}}(\tilde{x}_n,x)$ has smaller $d-1$ dimensional Lebesgue measure than $B^*_{_{f_{_{n}}}}(\tilde{\mathbf{x}}_n,x)$ for all $\tilde{\mathbf{x}}_n$ that satisfy $D^{_{(k)}}\geq f_{_{n}}$.
\pagebreak[3]

   Abbreviating $I_{n+1}:=  \bigl(B_{2f_{_{n+1}}(0)}\setminus B_{2f_{_{n}}(0)}\bigr)\cap O_{n+1}(\tilde{\textbf{x}}_n) $ we get
    \begin{align*}
     & \mathbb{P}[A_{n+1}^c | A_n] \\
     &\leq \mathbb{E}\Bigl[  \mathbb{P}\bigl[ \text{ there exists no }  \textbf{x} \in \mathcal{X} \mbox{ such that } x\in B_{2f_{_{n+1}}}\!\!(0)\setminus B_{2f_{_{n}}}\!\!(0)  \\ &\hspace{35pt} 
     D_x^{_{(k)}}\geq 2^{2(d-1)+\epsilon}f_{_{n+1}}, x\in O_{n+1}({\tilde{\textbf{x}}}_n)  \text{ and } C_x\cap B^*_{{f_{_{n}}}}({\tilde{\textbf{x}}}_n,x)\neq \emptyset  \,\mid\, \mathfrak F_n \bigr]\mathbbm{1}_{A_n}\Bigr]\big/{\mathbb P}(A_n)
\end{align*}
\begin{align*}     
     &= {\mathbb{E}\Bigl[\exp\Bigl(-u\!\!\!\!\int\limits_{I_{n+1} }\!\! \mathbb{P}_C( D_{C}^{_{(k)}} \!\!\geq 2^{2(d-1)+\epsilon}f_{_{n+1}},\, (x+C)\cap  B^*_{{f_{_{n}}}}({\tilde{\textbf{x}}}_n,x)\neq \emptyset\bigr)\,\text{d}\lambda (x)\Bigr)\mathbbm{1}_{A_n} \Bigr]\big/{\mathbb P}(A_n)}\\
     &\leq {\mathbb{E}\Bigl[\exp\Bigl(-u\!\!\!\!\int\limits_{I_{n+1} }\!\! \mathbb{P}_C( D_{C}^{_{(k)}} \!\!\geq 2^{2(d-1)+\epsilon}f_{_{n+1}},\, (x+C)\cap  Q^*_{f_{_{n}}}({\tilde{x}}_n,x)\neq \emptyset\bigr)\,\text{d}\lambda (x)\Bigr)\mathbbm{1}_{A_n} \Bigr]\big/{\mathbb P}(A_n)}.    
   \end{align*}
    We rewrite the last term as
    \begin{align*}
        \mathbb{E}\Bigl[\exp\Bigl(-u \smash{\int\limits_{I_{n+1} }}\!\! \mathbb{P}_C\bigl(  (x+C)\cap  Q^*_{f_{_{n}}}({\tilde{x}_n},x) & \, \neq \emptyset\big|\, D_{C}^{_{(k)}} \!\!\geq 2^{2(d-1)+\epsilon}f_{_{n+1}} \bigr) \\& \times \mathbb{P}_C(D_{C}^{_{(k)}} \!\!\geq 2^{2(d-1)+\epsilon}f_{_{n+1}}) \,\text{d}\lambda (x)\Bigr)\mathbbm{1}_{A_n} \Bigr]\big/{\mathbb P}(A_n)
    \end{align*}
    and focus on the integrand. \emph{First}, we bound the conditional probability from below. We use rotation invariance of the law of $C$ to bound the probability of the intersection from below by the $d-1$ dimensional Lebesgue measure of a subset of $S^{d-1}$ of rotations of $C$ for which  we can ensure a non-empty intersection with $Q^*_{f_{_{n}}}(\tilde{x}_n,x)$. Using that~$C$, having $k$ diameters of size at least \smash{$2^{2(d-1)+\epsilon}f_{n+1}$}, needs to hit a target with $k$ diameters with size of order $f_{_{n}}$, at distance at most $2f_{_{n+1}}$, we find using a simple inductive argument over~$d$ and Remark \ref{rem:winkel} such a set of rotations that has measure  \smash{$c f_{_{n}}^{\min\{d-k,k\}}f_{_{n+1}}^{-d+k}$}.\smallskip\pagebreak[3]
    
    \emph{Second}, we look at the factor \smash{$\mathbb{P}_C(D_{C}^{_{(k)}} \!\!\geq 2^{2(d-1)+\epsilon}f_{_{n+1}})$} and recall that the tail of
    $D_{C}^{_{(k)}}$
    is a regularly varying function with index $-\alpha_k$. Using the Potter bounds we get
    \begin{align*}
    \mathbb{P}_C\bigl(  (x+C)\cap  & Q^*_{f_{_{n}}}(\tilde{x}_n,x)\neq \emptyset\,|\, D_{C}^{_{(k)}} \geq 2^{2(d-1)+\epsilon}f_{_{n+1}} \bigr)  \mathbb{P}_C(D_{C}^{_{(k)}} \!\!\geq 2^{2(d-1)+\epsilon}f_{_{n+1}}) \\ & \geq c \frac{f_{_{n}}^{\min\{d-k,k\}}}{f_{_{n+1}}^{d-k}}f_{_{n+1}}^{-(\alpha_k+\varepsilon)},
    \end{align*}
    where $c>0$ depends only on $\varepsilon, \alpha_k, d,k,\epsilon$. As there exists  a further $c>0$  such that, for $f_{_{0}}$ big enough, points in \smash{$O_{n+1}(\tilde{\textbf{x}}_n)$} have distance bigger than $cf_{_{n+1}}$ from the origin, the volume of $I_{n+1}$ is of order \smash{$f_{_{n+1}}^d$}. This yields for $\mathbb{P}[A_{n+1}^{c}|A_n]$ the bound
        \begin{align*}
            \mathbb{E}&\Bigl[\exp\Bigl(-u\!\!\!\!\int\limits_{I_{n+1} }\!\!  c \frac{f_{_{n}}^{\min\{d-k,k\}}}{f_{_{n+1}}^{d-k}}f_{_{n+1}}^{-(\alpha_k+\varepsilon)}  \,\text{d}\lambda (x)\Bigr)\mathbbm{1}_{A_n} \Bigr]\big/{\mathbb P}(A_n) \\
            &\leq \mathbb{E}\Bigl[\exp\Bigl(-u  c f_{_{n}}^{\min\{d-k,k\}}f_{_{n+1}}^{-(\alpha_k+\varepsilon-k)}\Bigr) \mathbbm{1}_{A_n} \Bigr]\big/{\mathbb P}(A_n).
        \end{align*}
For fixed $0<\varepsilon<(\alpha_k-k)^2\epsilon\big/\big(2\min\{d-k,k\}\big)$ we get the inequality
\begin{equation*}
    -\varepsilon\frac{\min\{d-k,k\}}{\alpha_k-k}+\epsilon(\alpha_k +\varepsilon -k) > \frac{\epsilon}{2}(\alpha_k-k) >0,
\end{equation*}
and combining the above with the definition of our threshold sequence and the Potter bounds we get for $\mathbb{P}[A_{n+1}^c | A_n]$ the upper bound 
    \begin{equation*}
     \exp\bigl(  -u c f_{_{n+1}}^{-(\alpha_k + \varepsilon-k)}f_{_{n}}^{\min\{d-k,k\}} \bigr) \leq \exp\bigl(  -u c f_{_{n}}^{\epsilon(\alpha_k -k)/2}\bigr).
     \vspace{-20pt}
    \end{equation*}
    \end{proof}
    Using now Lemma~\ref{Lemma:one} and Lemma~\ref{Lemma:two} we can prove Theorem~\ref{Thm:one}.\pagebreak[3]
    
    \begin{proof}[Proof of Theorem~\ref{Thm:one}]
    As Lemma~\ref{Lemma:one} gives us a lower bound for $\mathbb{P}(A_0)$ and Lemma~\ref{Lemma:two} gives {an upper} bound for $\mathbb{P}[A_{n+1}^c|A_n]$ for ${n\in\mathbb{N}}$ we calculate
    \begin{align*}
        \mathbb{P}(A_n) &= {\mathbb{P}(A_0)\prod_{\ell=0}^{n-1} \mathbb{P}[A_{\ell+1}| A_{\ell}]}\\
        &\geq \Bigl(1-\exp\bigl(-uc f_{_{0}}^{d-\alpha_k-\varepsilon}\bigr)\Bigr) \prod\limits_{\ell=0}^{n-1} \Bigl(1- \exp\bigl(  -u c f_{_{\ell}}^{\epsilon(\alpha_k -k)/2}\bigr) \Bigr) \\
        &= \exp\Bigl\{\log\bigl(1-\exp(-uc f_{_{0}}^{d-\alpha_k-\varepsilon})\bigr) +\sum\limits_{\ell=0}^{n-1} \log \bigl(1-\exp (  -u c f_{_{\ell}}^{\epsilon(\alpha_k -k)/2} )\bigr)\Bigr\}.
    \end{align*}
    Using that for {small $x>0$ we have $\log(1-x) \geq -2x$ and $\exp(x)\geq 1+x$} we get
    \begin{align*}
        \mathbb{P}(A_n) &\geq  1- 2\exp(-uc f_{_{0}}^{d-\alpha_k-\varepsilon} ) - 2\sum\limits_{\ell=0}^{\infty} \exp(-u c f_{_{\ell}}^{\epsilon(\alpha_k -k)/2})\\
        &=  1- 2\exp(-uc f_{_{0}}^{d-\alpha_k-{\varepsilon}} ) - 2\sum\limits_{\ell=0}^{\infty}\exp\Bigl(-u c \bigl\{f_{_{0}}^{\epsilon(\alpha_k-k)/2}\bigr\}^{\bigl(\frac{\min\{d-k,k\}}{\alpha_k-k} -\epsilon \bigr)^\ell} \Bigr).  
    \end{align*}
    Due to our assumption that $k<\alpha_k < \min\{2k,d\}$, the last sum is finite and by choosing $f_{_{0}}$ large enough 
    this lower bound can be made arbitrarily close to one.   
    \medskip
    
    If \smash{$\alpha_d <d$} we can look at the classic 
    Poisson Boolean model with balls as convex grains with radius $D^{{(d)}}$. This grain distribution is dense and therefore also robust. 
    If there exists $k$ such that $\alpha_k=k$, we can choose \smash{$f_{_{n}}=f_{_{0}}^n$} for $n\in\mathbb{N}$ with $f_{_{ 0}}>0$ large enough. Using $(A_n)_{n\in\mathbb{N}}$ as before we get by the same arguments
    \begin{equation*}
        \mathbb{P}(A_0) \geq 1-\exp\bigr( -uf_{_{0}}^{d-k-\varepsilon}\bigr)
    \end{equation*}
    and
    \begin{align*}
    \mathbb{P}\bigl(A_{n+1}^c \,|\, A_n\bigr) &\leq \exp\bigl(-uc f_{_{n+1}}^{-\varepsilon} f_{_{n}}^{\min\{d-k,k\}}\bigr)\\
    &= \exp\bigl(-uc f_{_{0}}^{-\varepsilon(n+1) +\min\{d-k,k\}n}\bigr).
    \end{align*}
    The rest of the calculation can be done analogously to the calculation for $\alpha_k\neq k$ by choosing $\varepsilon>0$ small enough.
    \end{proof}
   
\subsection{Specific criteria for  robustness 
}\label{Sect:robustness_strong_dependence}

In this section we are dealing with diameters with a correlation condition. We assume there exists some $c>0$ and a positive integer $k\leq d/2$ and $\gamma_k, \ldots, \gamma_{d-k} \ge 0$ with
        \[
            \mathbb{P}_C\big(D^{_{(j)}}_C
            \ge c (D^{_{(k)}}_C)^{\gamma_j}
            \mbox{ for all } d-k\ge j >k, D^{_{(k)}}_C \ge x\big)
        \]
        regularly varying in $x$ with index $-\gamma_k$.

\begin{proof}[Proof of Theorem~\ref{Thm:one'}]
    The proof is similar to the one in the previous section. We therefore only point out the differences here and omit the details that remain unchanged. {Note that the event in the assumption of Theorem~\ref{Thm:one'} is contained in $\{D^{_{(k)}}_C\geq x\}$, so that $\gamma_k\geq\alpha_k$. If $\gamma_k<k$ we therefore have $\alpha_k<k\leq\min\{2k,d\}$ and robustness follows from Theorem~\ref{Thm:one}, while the case $\gamma_k=k$ can be treated as the case $\alpha_k=k$ in the previous section. We may hence assume that $\gamma_k>k$ and consider a threshold sequence.}
    Let $(\tilde f_n)_{n\in\mathbb{N}}$ be this sequence defined as
    \begin{equation*}
        \tilde f_n:=\tilde f_{n-1}^{\frac{k+\sum\limits_{j=k+1}^{d-k}\gamma_j}{\gamma_k-k} -\epsilon},
    \end{equation*}
    where $\tilde f_0>1$ 
    will be set large later in the proof. Let also $0<\epsilon<\tfrac{1}{4}$ be small enough to obtain that the exponent is strictly bigger than one in order to guarantee that the sequence is increasing. Note that by our assumption this is possible. 
\smallskip
\pagebreak[3]

    We argue in the following that an infinite path exists.
    For this we need again a sequence of events $(\tilde A_n)_{n\in\mathbb{N}}$ that describe the properties of this infinite path. 
    For $\mathbf{x}\in\mathcal{X}$ { we set $\tilde O_{0}(\textbf{x}):= B_{f_{_{0}}}(0)$ and for $i\in\mathbb N$} we modify the set $O_i(\mathbf{x})$ to become
    \begin{equation*}
        \tilde O_{_{i}}(\textbf{x}):=  \Biggl\{\begin{array}{c}y\in\mathbb{R}^d: |x-y| \in [\tfrac12 \tilde f_{_{i}},\tilde f_{_{i}}], \measuredangle\big(x-y,  v\big)>  \varphi,\\ \text{for } v\in\operatorname{span}\{p_{x}^{_{(j)}}: 1\leq j\leq d-1\}\end{array} \Biggr\},
    \end{equation*}    
    with $ \varphi =2^{-(d+1)}(d+1)^{-1}$. The relevant change here is that compared to $O_i(\mathbf{x})$ we impose requirements on the orientations $p_x^{_{(k+1)}},\dots,p_x^{_{(d-1)}}$. For {$n\in\mathbb{N} \cup \{0\}$} we define
    \begin{align*}
        \tilde A_n  := \Bigg\{ \begin{split}&\exists {\textbf{x}_0},\ldots, \textbf{x}_n\in\mathcal{X}\colon \textbf{x}_i\neq\textbf{x}_m \text{ for all } i\neq m, D^{_{(j)}}_{x_i}
            \ge c (D^{_{(k)}}_{x_i})^{\gamma_j}
            \mbox{ for all } d-k\ge j >k,\\ & 
          D_{x_i}^{_{(k)}}\geq 2^{2(d-1)+\epsilon}\tilde f_{_{i}},\,
        x_i\in \tilde O_i(\textbf{x}_{i-1})\text{ and }C_{x_i}\cap B^{*}_{\tilde f_{_{i-1}}}(\textbf{x}_{i-1},x_i)\neq \emptyset \text{ for all } {0\leq} i\leq n \end{split}\Bigg\},
    \end{align*}
{where we formally define 
        \smash{$B^{*}_{\tilde f_{{-1}}}(\textbf{x}_{-1},x_0)=\mathbb{R}^d$} to make this consistent with our earlier definition of $A_0$.}
  Note that in this definition of
  $(A_n)_{n\in\mathbb{N}}$ we additionally do not allow  $x_i$ to lie `too close' to the affine subspace through $x_{i-1}$ spanned by the orientations $p_{x_{i-1}}^{_{(1)}},\dots, p_{x_{i-1}}^{_{(d-1)}}$ of first $d-1$ diameters of $\mathbf{x}_{i-1}$ for all $i\in\{1,...,n\}$ and also that the smaller diameters have some lower bound depending on the $k$th diameter. All other properties of $(A_n)_{n\in\mathbb{N}}$ still hold for $(\tilde A_n)_{n\in\mathbb{N}}$ by replacing $(f_n)_{n\in\mathbb{N}}$ by the new threshold sequence $(\tilde f_n)_{n\in\mathbb{N}}$. \smallskip\pagebreak[3]

    To get a suitable lower bound for $\mathbb{P}(\tilde A_n)$ we find, similarly as before, an upper bound for $\mathbb{P}[\tilde A_{n+1}^c|\tilde A_n]$. We use the same considerations as in the previous section but replace the definition for $x,y\in\mathbb{R}^d$ of $Q_{f_n}(x,y)$ by
    \begin{align*}
        \tilde Q_{\tilde f_{_{n}}}& (x,y) := \rot_{\rho_{x,y}}\Bigl(\operatorname{conv}\bigl(\bigl\{-\tfrac{\tilde f_n^{\,\gamma_{k+i}}}{2}e_i,\tfrac{\tilde f_n^{\,\gamma_{k+i}}}{2}e_i\,: \, 1\leq i \leq {d-2k}\bigr\}\\
        & {\cup\bigl\{-\tfrac{\epsilon}{2}e_i,\tfrac{\epsilon}{2}e_i\,: \, d-2k< i \leq d-k\bigr\}}\cup \bigl\{-\tfrac{\tilde f_{_{n}}}{2}e_m,\tfrac{\tilde f_{_{n}}}{2}e_m\, : \, d-k+1\leq m\leq d\bigr\}\bigr)\Bigr) +x.
    \end{align*}
    and define 
    \begin{equation*}
         \tilde Q^*_{_{\tilde f_{_{n}}}}(x,y):= P_{H_{x-y}}\bigl(\tilde Q_{\tilde f_{_{n}}}(x,y)\bigr).
    \end{equation*}
    Recall that the key step of the calculation in the previous section was to find a lower bound for $\mathbb{P}_C\bigl(  (x+C)\cap  Q^*_{f_{_{n}}}(\tilde{x}_n,x)\neq \emptyset\,|\, D_{C}^{_{(k)}} \geq 2^{2(d-1)+\epsilon}f_{_{n+1}} \bigr)  \mathbb{P}_C(D_{C}^{_{(k)}} \!\!\geq 2^{2(d-1)+\epsilon}f_{_{n+1}}) $ as it is done in the proof of Lemma~\ref{Lemma:two}. Translating this to the new events we want to find a lower bound for 
    \begin{align*}
        \mathbb{P}_C\bigl(   (x+C)\cap  \tilde Q^*_{\tilde f_{_{n}}}&(\tilde{x}_n,x)\neq \emptyset\,|\,D^{_{(j)}}_{C}
            \ge c (D^{_{(k)}}_{C})^{\gamma_j}
            \mbox{ for } d-k\ge j >k, D_{C}^{_{(k)}} \geq 2^{2(d-1)+\epsilon}\tilde f_{_{n+1}}\bigl) \\
            &\times\mathbb{P}_C(D^{_{(j)}}_{C}
            \ge c (D^{_{(k)}}_{C})^{\gamma_j}
            \mbox{ for } d-k\ge j >k,D_{C}^{_{(k)}} \!\!\geq 2^{2(d-1)+\epsilon}\tilde f_{_{n+1}}). 
    \end{align*}
    Focusing on the first term we want to calculate the probability of a suitable orientation of the convex body such that an intersection occurs with $\tilde Q^*_{\tilde f_{_{n}}}(\tilde{x}_n,x)$ under the condition that the diameters have some lower bound. Note that in the previous section we get for this kind of term the lower bound ${ f_{n}^{\min\{d-k,k\}}}/{ f_{n+1}^{d-k}}$. In this section we have that $\min\{d-k,k\}=k$ due to the choice of $k$. In addition to that we want in this section that for the $i$th vertex not only  $ D_{x_i}^{_{(k)}} \geq 2^{2(d-1)+\epsilon}f_{_{i}}$, but also $D^{_{(j)}}_{x_i} \ge c (D^{_{(k)}}_{C})^{\gamma_j} \mbox{ for all } d-k\ge j >k$. With this we have the additional information that the  smaller diameters of the $i$th vertex satisfy $D^{_{(j)}}_{x_i} \ge c (\tilde f_i)^{\gamma_j} \mbox{ for all } d-k\ge j >k$. 
    With this we get
    \begin{align*}
        \mathbb{P}_C\bigl(   (x+C)\cap  \tilde Q^*_{\tilde f_{_{n}}}(\tilde{x}_n,x)\neq \emptyset\,|&\,D^{_{(j)}}_{C}
            \ge c (D^{_{(k)}}_{C})^{\gamma_j}
            \mbox{ for } d-k\ge j >k, D_{C}^{_{(k)}} \geq 2^{2(d-1)+\epsilon}{\tilde f_{_{n+1}}}\bigl)\\
            &\geq c\frac{ \tilde f_{n}^{k }}{ \tilde f_{n+1}^{d-k}}\cdot \tilde f_n^{\sum_{j=k+1}^{d-k} \gamma_j}.
    \end{align*}
    The extra factor $ \tilde f_n^{\sum_{j=k+1}^{d-k} \gamma_j}$ is given by the new minimal size of the smaller diameters. For the second term we use the potter bound for regularly varying functions to get
    \begin{equation*}
        \mathbb{P}_C(D^{_{(j)}}_{C}
            \ge c (D^{_{(k)}}_{C})^{\gamma_j}
            \mbox{ for all } d-k\ge j >k,D_{C}^{_{(k)}} \!\!\geq 2^{2(d-1)+\epsilon}\tilde f_{_{n+1}}) \geq {c\,\tilde f_{n+1}^{-(\gamma_k +\varepsilon)}}.
    \end{equation*}
    Altogether we get 
    \begin{align*}
        \mathbb{P}_C\bigl(   (x+C)\cap  \tilde Q^*_{\tilde f_{_{n}}}&(\tilde{x}_n,x)\neq \emptyset\,|\,D^{_{(j)}}_{C}
            \ge c (D^{_{(k)}}_{C})^{\gamma_j}
            \forall  d-k\ge j >k, D_{C}^{_{(k)}} \geq 2^{2(d-1)+\epsilon}\tilde f_{_{n+1}}\bigl) \\
            \times & \, \mathbb{P}_C(D^{_{(j)}}_{C}
            \ge c (D^{_{(k)}}_{C})^{\gamma_j}
            \mbox{ for all } d-k\ge j >k,D_{C}^{_{(k)}} \!\!\geq 2^{2(d-1)+\epsilon}\tilde f_{_{n+1}}) \\
            &\geq {c\,}\frac{ \tilde f_{n}^{k + \sum_{j=k+1}^{d-k} \gamma_j}}{ \tilde f_{n+1}^{d-k}} {\tilde f_{n+1}^{-(\gamma_k +\varepsilon)}}.
    \end{align*}
    
    Completing the calculation as 
    in the previous section with the new lower bound we~get \\[-23pt]
    \begin{align*}
        \mathbb{P}(\tilde A_n) \geq 1- 2\exp(-uc {\tilde f_{_{0}}^{\,d-\gamma_k-\varepsilon}} ) - 2\sum\limits_{\ell=0}^{\infty}\exp\Bigl(-u c \bigl\{{\tilde f_{_{0}}}^{\,\epsilon(\gamma_k-k)/2}\bigr\}^{\bigl(\frac{k+\sum\limits_{j=k+1}^{d-k}\gamma_j}{\gamma_k-k} \, -\epsilon \bigr)^\ell} \Bigr). 
    \end{align*}
    The sum is finite by our assumption \smash{$\gamma_k< 2k+\sum_{j=k+1}^{d-k} \gamma_j$}. {Note that $\gamma_k<d$: as $D^{_{(j)}}_C\leq D^{_{(k)}}_C$ for $j>k$, the event in the assumption of Theorem~\ref{Thm:one'} has probability zero for all large $x$ unless $\gamma_j\leq 1$ for all $k<j\leq d-k$, so that $(b)$ gives $\gamma_k<2k+(d-2k)=d$ and the first exponential term is small for $\tilde f_0$ large.} This lower bound is strictly positive by choosing ${\tilde f_0}$ large enough.     
\end{proof}

\begin{myrem}
    The model in this section is a generalisation of the 
    ellipsoids with strongly dependent axes from Section \ref{Sect:introduction}, obtained by setting $d=3$, $k=1$, $\gamma_1=\alpha$ and $j=2$ and $\gamma_j=\gamma$.
\end{myrem}

\begin{myrem}
    As mentioned above, Proposition~\ref{Prop:Gouere} follows directly from \cite{Gouere2008}. For that we replace all convex bodies by balls with radius given by {their} first diameters and the result follows by the non-robustness of the Poisson Boolean model with balls as grains.
\end{myrem}

    \subsection{Specific criteria for non-robustness}
    \label{Sect:non-robust_ellipses}

    In this section we prove non-robustness, {i.e.\ the existence of a subcritical phase,} for the Poisson Boolean model with correlations given in
    Theorem \ref{Thm:two}. Note that by choosing $\beta=0$ we have a generalisation of the $2$-dimensional ellipses model described in \cite{Teixeira2017}. \smallskip
    
    The proof of Theorem~\ref{Thm:two} is based on the idea of Section 4 in \cite{Hall1985}. We show that in the Palm version of the point process the expected number of vertices that are connected by a path to the origin is finite for small enough intensity $u$. The result follows from this. \smallskip

    We prove our claim by induction and use throughout the proof the notation 
    \begin{equation*}
        \sum\limits_{_{\textbf{x}_1,\dots,\textbf{x}_{n}\in\mathcal{X}}}^{\neq}  \prod\limits_{i=1}^{n} \mathbbm{1}_{\textbf{x}_{i-1}\sim \textbf{x}_i} := \big|\big\{\textbf{x}_n\in\mathcal{X}\,:\,\sum\limits_{_{\textbf{x}_1,\dots,\textbf{x}_{n-1}\in\mathcal{X}}} \prod\limits_{i=1}^{n} \mathbbm{1}_{\textbf{x}_{i-1}\sim \textbf{x}_i}>0\big\}\big|,
    \end{equation*}
    that is the number of vertices $\textbf{x}_n$ that are the final vertex of a path of length $n$ starting in the vertex $\textbf{x}_0$. 
    Assuming without loss of generality that the origin is the location of a vertex of $\mathcal{X}$ we are first interested in the expected number of vertices that are connected to the origin via a path of length two. This is given by
    \begin{equation*}
        \mathbb{E}_0 \Bigl[\sum\limits_{_{\textbf{y}\in\mathcal{X}}} \mathbbm{1}_{\textbf{0} \overset{2}{\sim} \textbf{y}} \Bigr] = \mathbb{E}_0 \Bigl[\sum\limits_{_{\textbf{x},\textbf{y}\in\mathcal{X}}}^{\neq} \mathbbm{1}_{\textbf{0}\sim \textbf{x}} \mathbbm{1}_{\textbf{x}\sim \textbf{y}} \Bigr].
    \end{equation*}
    After finding an upper bound for this expectation, we will bound
    \begin{equation*}
        \mathbb{E}_0 \Bigl[\sum\limits_{_{\textbf{x}\in\mathcal{X}}} \mathbbm{1}_{\textbf{0} \overset{n}{\sim}\textbf{x}} \Bigr] = \mathbb{E}_0 \Bigl[\sum\limits_{_{\textbf{x}_1,\dots,\textbf{x}_{n}\in\mathcal{X}}}^{\neq}  \prod\limits_{i=1}^{n} \mathbbm{1}_{\textbf{x}_{i-1}\sim \textbf{x}_i} \Bigr]
    \end{equation*}
    from above for $\textbf{x}_0=\textbf{0}$, first for $n\in\mathbb{N}$ even and after that for $n$ odd, and show that  choosing $u$ small enough ensures that
    \begin{equation*}
        \sum\limits_{n\in\mathbb{N}} \mathbb{E}_0\Bigl[\sum\limits_{\textbf{x}\in\mathcal{X}} \mathbbm{1}_{\textbf{0}\overset{n}{\sim} \textbf{x}}\Bigr] < \infty.
    \end{equation*}
    Note that the upper bound for the case $n=1$ is already done in the proof of Proposition~\ref{Prop:one}. To bound the expectations from above we dominate our model by replacing the convex body $C$ with diameters $D^{{(1)}},\dots,D^{{(d)}}$ with rectangles $\bar{R}$ similar to Section 2.1. Instead of $C_x$ we look at rectangles $\bar{R}_x$ taken as the sum of $x$ and the cartesian product of the intervals $[-\bar{D}_x^{_{(i)}}, \bar{D}_x^{_{(i)}}]$ with respect to the orthonormal basis given by directions~$p^{_{(i)}}_{C_x}$,   where we use the notation 
    \begin{equation*}
        \bar{L}:=\min\{n\in\mathbb{N}\,:n>L\}, \text{ for } L\geq 0.
    \end{equation*}
    For $\textbf{x}\in\mathcal{X}$ we write $\bar{D}_x:=\bigl(\bar{D}_x^{_{(1)}},\dots,\bar{D}_x^{_{(d)}}\bigr)\in\mathbb{N}^d$ for the vector of rounded diameters of $\textbf{x}$.
    {For $\eta>0$ we write $\eta\bar{R}_x:=x+\eta(\bar{R}_x-x)$ for the blow-up of $\bar{R}_x$ about its centre $x$.}
    Note that the tails of $\bar{D}^{(1)},\dots,\bar{D}^{(d)}$ can also be bounded via suitable Potter bounds since $D^{_{(1)}},\dots, D^{_{(d)}}$ are regularly varying. For $\textbf{x},\textbf{y}\in\mathcal{X}$ we get
        \begin{equation*}
            \mathbbm{1}_{ C_x\cap C_y \neq \emptyset}\leq  \mathbbm{1}_{ \bar{R}_x\cap \bar{R}_y \neq \emptyset}.
        \end{equation*}{We define 
    \begin{equation*}
       V:=\prod\limits_{j=1}^d \bar{D}^{_{(j)}}\geq {1}.
    \end{equation*}
 We will use induction to prove existence of $c>1$ such that
    \begin{equation*}
        \mathbb{E}_0 \Bigl[\sum\limits_{_{\textbf{x}\in\mathcal{X}} } \mathbbm{1}_{\textbf{0} \overset{2n}{\sim}\textbf{x}} \Bigr] \leq u^{n}c^{n}10^{n} \mathbb{E}\bigl[V^2\bigr]^{n+1}
    \end{equation*}
    and
    \begin{equation*}
        \mathbb{E}_0 \Bigl[\sum\limits_{_{\textbf{x}\in\mathcal{X}} } \mathbbm{1}_{\textbf{0} \overset{2n+1}{\sim}\textbf{x}} \Bigr] \leq u^{n+1}c^{n+1}10^{n+1} \mathbb{E}\bigl[V^2\bigr]^{n+2}.
    \end{equation*}
   
}
    
    \paragraph{Base clause.} 
        Replacing the convex bodies by the rectangles, we get 
        \begin{equation}\begin{split}
              \mathbb{E}_0 \Bigl[\sum\limits_{\textbf{y}\in\mathcal{X}} \mathbbm{1}_{\textbf{0}\overset{2}{\sim} \textbf{y}}\Bigr] 
             &\leq  \mathbb{E}_0 \Bigl[\sum\limits_{\textbf{y}\in\mathcal{X}} \mathbbm{1}_{ y\in 2\bar{R}_{0}}\Bigr] + \mathbb{E}_0 \Bigl[\sum\limits_{\textbf{y}\in\mathcal{X}} \mathbbm{1}_{ y \in (2\bar{R}_{{y}}-y)}\Bigr]\\
             &\hspace{-5pt}+ \, 
             \mathbb{E}_0 \Bigl[ \sum\limits_{{j_{_{0}},j_{_{2}} \in\mathbb{N}^d} }\sum\limits_{{\textbf{x},\textbf{y}\in\mathcal{X}}} \mathbbm{1}_{\textbf{0}\sim \textbf{x}} \mathbbm{1}_{\textbf{x}\sim \textbf{y}}  \mathbbm{1}_{\bar{D}_{0}=j_{_{0}}}\mathbbm{1}_{\bar{D}_{y}=j_{_{2}}} \mathbbm{1}_{y \not\in 2\bar{R}_0} \mathbbm{1}_{0 \not\in 2\bar{R}_y}\mathbbm{1}_{\bar{R}_0\cap\bar{R}_y =\emptyset} \Bigr],
    \end{split}\label{eq:two_step}\end{equation}
        bounding the expectation from above by counting every vertex in $2\bar{R}_0$, every vertex $\textbf{y}$ such that $0\in2\bar{R}_y$, and for every other vertex $\textbf{y}$ all the vertices $\textbf{x}$ connected by an edge to $\textbf{0}$ and $\textbf{y}$. In addition, we have rewritten the second term using distributional symmetry.%
        \smallskip%

    We now focus on the last term in the expression. Recall that $\rho_{v,w}$ is the orientation of the vector $v-w$.
    Set 
    \begin{align*}
        \mathcal{W}^{-}_{w,v} := \,\rot_{\rho_{v,w}}\bigl((-\infty,0] \times \mathbb{R}^{d-1} \bigr)+w, \qquad
        \mathcal{W}^{+}_{w,v} := \,\rot_{\rho_{v,w}}\bigl([0,\infty) \times \mathbb{R}^{d-1} \bigr) +v, 
          \end{align*}
        and
        $\mathcal{W}_{w,v} := \,\mathbb{R}^d \setminus \bigl(\mathcal{W}^{-}_{w,v} \cup \mathcal{W}^{+}_{w,v} \bigr).$     
Define for $\textbf{w},\textbf{v}\in\mathcal{X}$ the sets $ \mathcal{T}_{\textbf{w},\textbf{v}}:= \mathcal{T}(\textbf{w},\textbf{v})\cup \mathcal{T}(\textbf{v},\textbf{w})$ where $\mathcal{T}$ is given as 
    \begin{equation*}
        \mathcal{T}(\textbf{w},\textbf{v}) := \Bigl\{\rot_{\rho_{v,w}}\Bigl( \{\bar{D}_{w}^{_{(d)}}+\!\lambda\}\!\times\!  \bigtimes\limits_{i=1}^{d-1} [-\bar{D}_{w}^{_{(i)}}\!\!-\!\! f_i(\textbf{w},\textbf{v},\lambda) , \bar{D}_{w}^{_{(i)}}\!+\!f_i(\textbf{w},\textbf{v},\lambda)] \Bigr)\!+w:\lambda\in(0,\infty)\Bigr\}
    \end{equation*}
    with 
    \begin{equation*}
    f_i(\textbf{w},\textbf{v},\lambda):= \lambda\frac{ \bar{D}_{v}^{_{(i)}} \!- \!\bar{D}_{w} ^{_{(i)}}}{|v-w| -\bar{D}_{v}^{_{(d)}}-\bar{D}_{w}^{_{(d)}}}\mathbbm{1}_{\bar{D}_{v}^{_{(i)}}\geq \bar{D}_{w}^{_{(i)}}} +(\bar{D}_{v}^{_{(i)}} -\bar{D}_{w}^{_{(i)}} )\mathbbm{1}_{\bar{D}_{v}^{_{(i)}}< \bar{D}_{w}^{_{(i)}}}, \text{ for } 1\leq i <d. 
    \end{equation*}
We define a partition \smash{$\bigl(\mathcal{A}_m(\textbf{w},\textbf{v})\bigr)_{m=1,\dots,8}$} of~$\mathbb{R}^d$ which only depends on the locations $w$ and~$v$ and the corresponding diameters of $\textbf{w}$ and $\textbf{v}$, as  
    \begin{equation}\begin{split}
\mathcal{A}_1(\textbf{w},\textbf{v}) &:= \rot_{\rho_{v,w}}\big( [-\bar{D}_{w}^{_{(d)}} , \bar{D}_{w}^{_{(d)}}]\times\bigtimes\limits_{i=1}^{d-1} [-\bar{D}_{w}^{_{(i)}} , \bar{D}_{w}^{_{(i)}}] \big)+w, \\[-9pt]
    \mathcal{A}_2(\textbf{w},\textbf{v}) & :=\mathcal{A}_1(\textbf{v},\textbf{w}),\\       
     \mathcal{A}_3(\textbf{w},\textbf{v}) &:= \mathcal{T}_{\textbf{w},\textbf{v}} \cap  \mathcal{W}^{-}_{w,v}\setminus\bigl(\mathcal{A}_1(\textbf{w},\textbf{v})\cup \mathcal{A}_2(\textbf{w},\textbf{v})\bigr),\\
   \mathcal{A}_4(\textbf{w},\textbf{v}) &:= \mathcal{T}_{\textbf{w},\textbf{v}} \cap \mathcal{W}^{+}_{w,v}\setminus\bigl(\mathcal{A}_1(\textbf{w},\textbf{v})\cup \mathcal{A}_2(\textbf{w},\textbf{v})\bigr),\\ 
     \mathcal{A}_5(\textbf{w},\textbf{v}) &:= \mathcal{T}_{\textbf{w},\textbf{v}} \cap \mathcal{W}_{w,v}\setminus\bigl(\mathcal{A}_1(\textbf{w},\textbf{v})\cup \mathcal{A}_2(\textbf{w},\textbf{v})\bigr),\\ \mathcal{A}_6(\textbf{w},\textbf{v}) &:= \mathcal{W}^{-}_{w,v}  \setminus \bigl( \mathcal{T}_{\textbf{w},\textbf{v}}\cup \mathcal{A}_1(\textbf{w},\textbf{v})\cup \mathcal{A}_2(\textbf{w},\textbf{v})\bigr),\\
   \mathcal{A}_7(\textbf{w},\textbf{v}) &:=
   \mathcal{W}_{w,v} \setminus \bigl( \mathcal{T}_{\textbf{w},\textbf{v}} \cup \mathcal{A}_1(\textbf{w},\textbf{v})\cup \mathcal{A}_2(\textbf{w},\textbf{v})\bigr),\\
        \mathcal{A}_8(\textbf{w},\textbf{v}) &:=  \mathcal{W}^{+}_{w,v} \setminus \bigl(  \mathcal{T}_{\textbf{w},\textbf{v}} \cup \mathcal{A}_1(\textbf{w},\textbf{v})\cup \mathcal{A}_2(\textbf{w},\textbf{v})\bigr).
    \end{split}\label{eq:partition}
\end{equation}

    \begin{figure}[!ht]
    \begin{annotate}{
    \includegraphics[width=0.975\linewidth]{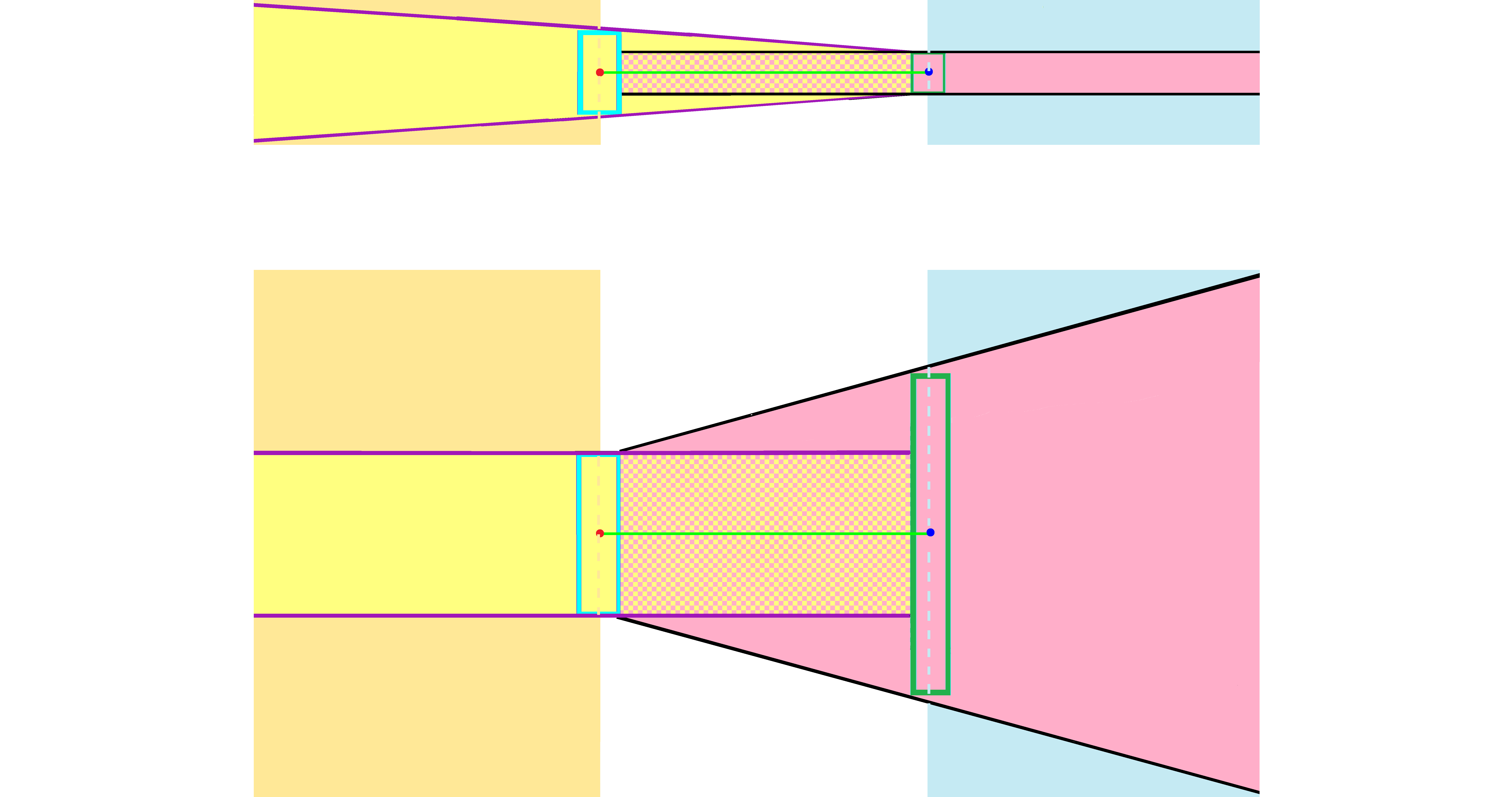}}{0.975}
        \filldraw[draw=candypink,fill=candypink] (7.5,-3.5) rectangle (5.5,-3);
        \node at (6.5,-3.25) {$\mathcal{T}(\textbf{v},\textbf{w})$};
        \filldraw[draw=unmellowyellow,fill=unmellowyellow] (7.5,-2) rectangle (5.5,-2.5);
        \node at (6.5,-2.25) {$\mathcal{T}(\textbf{w},\textbf{v})$};
        \draw[->] (0.2,1.8) -- (-1.5,1.8);
        \draw[->] (2.3,1.8) -- (4,1.8);
        \draw[line width=1.5] (-1.5,2.0) -- (-1.5,1.6);
        \draw[line width=1.5] (4,2.0) -- (4,1.6);
        \draw[dashed,line width=0.5pt] (-1.5,2.75) --  (-1.5,-0.5);
        \draw[dashed,line width=0.5pt] (4,2.95) --  (4,0.95);
        \node at (1.25,1.8) {$\bar{D}_{w}^{_{(3)}}+\lambda$};
        \draw[<->] (5.4,3.15) -- (5.4,3.37);
        \draw[line width=1.5] (5.2,3.15) -- (5.6,3.15);
        \draw[line width=1.5] (5.2,3.37) -- (5.6,3.37);
        \draw[dashed,line width=0.5pt] (4.85,3.15) -- (5.2,3.15);
        \draw[dashed,line width=0.5pt] (4.85,3.37) -- (5.6,3.37);
        \draw[->] (6.5,2.7) .. controls (6.5,3.25) .. (5.6,3.25);
        \node at (5.8,2.4) {$\bar{D}_{w}^{_{(3)}}+\bar{D}_{v}^{_{(1)}} -\bar{D}_{w}^{_{(1)}} $};
        \draw[->] (6,-0.6) -- (6,-1.3);
        \draw[->] (6,0.25) -- (6,0.95);
        \draw[line width=1.5] (5.8,-1.3) -- (6.2,-1.3);
        \draw[line width=1.5] (5.8,0.95) -- (6.2,0.95);
        \draw[dashed,line width=0.5pt] (4,0.95) --  (6,0.95);
        \draw[dashed,line width=0.5pt] (4,-1.3) --  (6,-1.3);
        \node at (5.5,-0.175) {$\bar{D}_{w}^{_{(3)}}+\lambda\frac{ \bar{D}_{v}^{_{(2)}} \!- \!\bar{D}_{w} ^{_{(2)}}}{|v-w| -\bar{D}_{v}^{_{(3)}}-\bar{D}_{w}^{_{(3)}}}$};
        %
        \filldraw[draw=black,fill=white] (-7.5,3.5) rectangle (-5,3);
        \node at (-6.25,3.25) {$\bar{D}_{w}^{_{(1)}} > \bar{D}_{v}^{_{(1)}} $:};
        \filldraw[draw=black,fill=white] (-7.5,1) rectangle (-5,0.5);
        \node at (-6.25,0.75) {$\bar{D}_{w}^{_{(2)}} <\bar{D}_{v}^{_{(2)}} $:};

    \end{annotate}
    \caption{Visualisation of $\mathcal{T}_{\textbf{w},\textbf{v}}$ in 3 dimensions by colouring of the different parts of $\mathcal{T}(\textbf{w},\textbf{v})$ and $\mathcal{T}(\textbf{v},\textbf{w})$. The perspective in the top figure is along $\rot_{\rho_{v,w}}(e_2)$ and along $\rot_{\rho_{v,w}}(e_3)$ in the bottom figure. The turquoise rectangle represents the set \smash{$\rot_{\rho_{v,w}}([-\bar{D}_{w}^{_{(d)}}, \bar{D}_{w}^{_{(d)}}]\times\bigtimes_{i=1}^{d-1} [-\bar{D}_{w}^{_{(i)}} , \bar{D}_{w}^{_{(i)}}])$} and the dark green one represents \smash{$\rot_{\rho_{v,w}}([-\bar{D}_{v}^{_{(d)}}, \bar{D}_{v}^{_{(d)}}]\times\bigtimes_{i=1}^{d-1} [-\bar{D}_{v}^{_{(i)}} , \bar{D}_{v}^{_{(i)}}])$}. 
    The light orange area and the area on the left side of the orange dashed line is $\mathcal{W}^{-}_{w,v}$ and the light blue area and the area on the right side of the blue dashed line $\mathcal{W}^{+}_{w,v}$. The green line is the connection line of $w$ and $v$ given as the red and blue point. 
    The purple lines are $f_{i}(\textbf{v},\textbf{w},\lambda)$ and the black lines $f_{i}(\textbf{w},\textbf{v},\lambda)$ for $i\in\{2,3\}$ and $\lambda>0$. }
    \label{Twv}
    \end{figure}
    Figure~\ref{Twv} \vpageref{Twv} is a visualisation of $\mathcal{T}_{\textbf{w},\textbf{v}}$ as a union of $\mathcal{T}(\textbf{w},\textbf{v})$  and $\mathcal{T}(\textbf{v},\textbf{w})$. To help with understanding how $\mathcal{T}_{\textbf{w},\textbf{v}}$ relates to the partition $(\mathcal{A}_{m}(\textbf{w},\textbf{v}))_{m=1,..,8}$, see also Figures~\ref{side} to \ref{area} \vpageref{side}.
    In all sub-figures of Figure~\ref{partitionAll3}, $\mathcal{T}_{\textbf{w},\textbf{v}}$ is the union of the black (middle), pink (rightmost) and red (leftmost) shaped volumes. In Figure~\ref{area} the red (left) plane is the boundary of $\mathcal{W}^{-}_{w,v}$ and the pink (right) plane the boundary of $\mathcal{W}^{+}_{w,v}$. It is now possible to recognise the partition $\bigl(\mathcal{A}_m\colon {m=1,\ldots,8}\bigr)$: The set $\mathcal{A}_1(\textbf{w},\textbf{v})$ is the turquoise part and $\mathcal{A}_2(\textbf{w},\textbf{v})$ the dark blue part. Furthermore $\mathcal{A}_3(\textbf{w},\textbf{v}), \mathcal{A}_4(\textbf{w},\textbf{v})$ and $\mathcal{A}_5(\textbf{w},\textbf{v})$ are given respectively by the red, pink and black shapes in the figures. In Figure \ref{side} and \ref{above} we see that $\mathcal{A}_6(\textbf{w},\textbf{v})$ is the white part on the left side of the first orthogonal black line including the line itself, $\mathcal{A}_7(\textbf{w},\textbf{v})$ the white part between both orthogonal black lines and $\mathcal{A}_8(\textbf{w},\textbf{v})$ the white part on the right side of the second orthogonal black line, again including the line itself. Looking at Figure \ref{area}, the first black line is the red plane and the second black line the pink plane, i.e. $\mathcal{A}_6(\textbf{w},\textbf{v})$ is the white part on the left side of the red plane joined with the red plane, $\mathcal{A}_7(\textbf{w},\textbf{v})$ the white part between the red and pink plane and $\mathcal{A}_8(\textbf{w},\textbf{v})$ the white part on the right side of the pink plane joined with the pink plane. 
    \smallskip
    
\begin{figure}[ht]
    \begin{subfigure}[t]{.495\linewidth}
        \includegraphics[width=\linewidth]{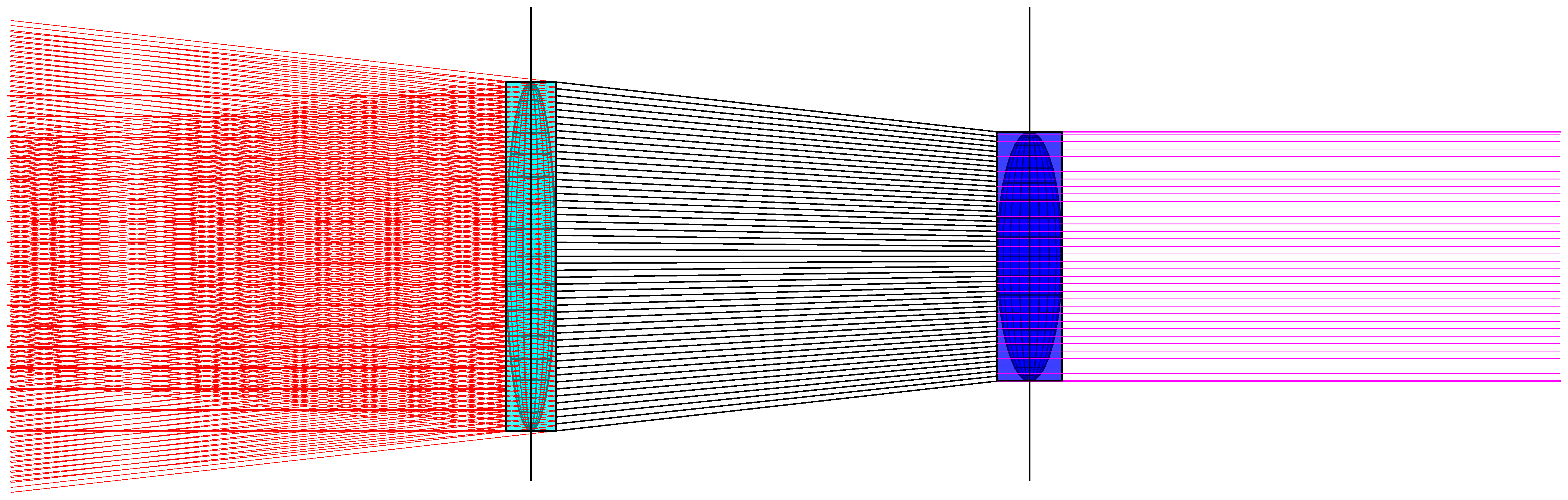}
        \subcaption{The perspective along $\rot_{\rho_{v,w}}(e_2)$.}\label{side}
    \end{subfigure}
    \begin{subfigure}[t]{.495\linewidth}
        \includegraphics[width=\linewidth]{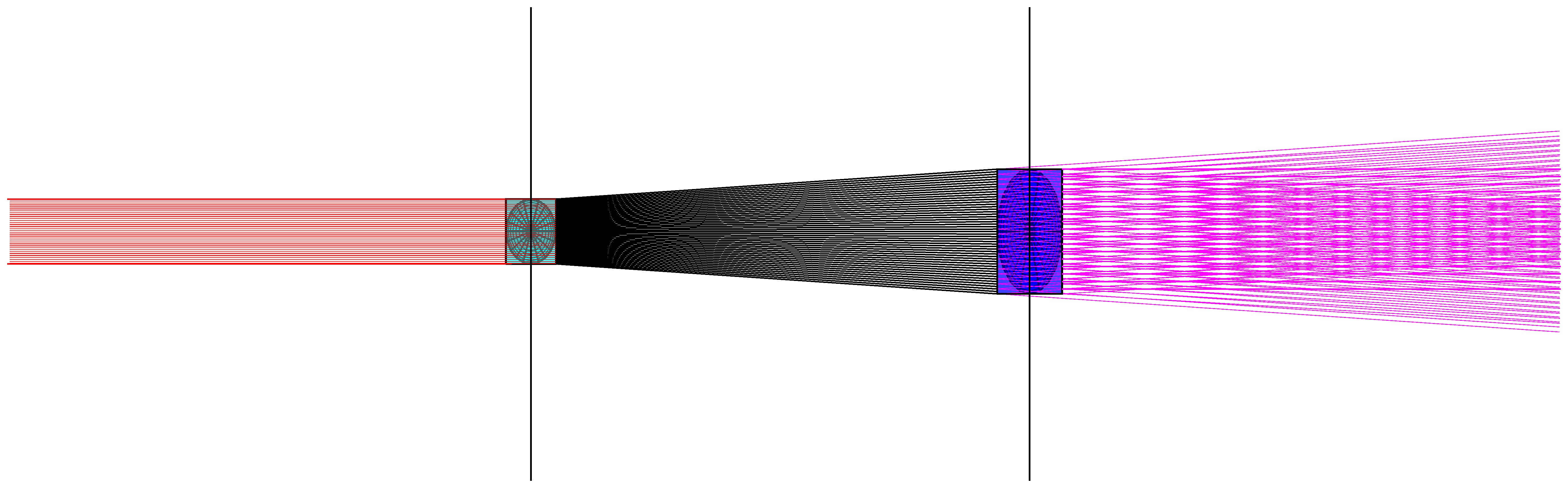}
        \subcaption{The perspective along $\rot_{\rho_{v,w}}(e_3)$.}\label{above}
    \end{subfigure}
    \begin{subfigure}[t]{1\linewidth}
        \begin{annotationimage}{width=12cm}{nonrobust_3d} 
\draw[annotation above = {$\mathcal{A}_1(\textbf{w},\textbf{v})$ at 0.25}] to (0.35 ,0.55 );
\draw[annotation above = {$\mathcal{A}_2(\textbf{w},\textbf{v})$ at 0.75}] to ( 0.65 ,0.45);
\draw[annotation left = {$\mathcal{A}_3(\textbf{w},\textbf{v})$ at 0.9}] to ( 0.1, 0.7);
\draw[annotation above = {$\mathcal{A}_4(\textbf{w},\textbf{v})$ at 0.9}] to ( 0.9,0.4 );
\draw[annotation above = {$\mathcal{A}_5(\textbf{w},\textbf{v})$ at 0.5}] to ( 0.45, 0.55 );
\draw[annotation left = {$\mathcal{A}_6(\textbf{w},\textbf{v})$ at 0.3}] to ( 0.2, 1);
\draw[annotation left = {$\mathcal{A}_6(\textbf{w},\textbf{v})$ at 0.3}] to ( 0.15, 0.35);
\draw[annotation below = {$\mathcal{A}_7(\textbf{w},\textbf{v})$ at 0.4}] to ( 0.55, 0.8);
\draw[annotation below = {$\mathcal{A}_7(\textbf{w},\textbf{v})$ at 0.4}] to ( 0.53, 0.21);
\draw[annotation below = {$\mathcal{A}_8(\textbf{w},\textbf{v})$ at 0.8}] to ( 0.75, 0.7);
\draw[annotation below = {$\mathcal{A}_8(\textbf{w},\textbf{v})$ at 0.8}] to ( 0.73, 0.15);
\end{annotationimage}
\caption{$\mathcal{T}_{\textbf{w},\textbf{v}}$ is the union of $\mathcal{A}_3(\textbf{w},\textbf{v})$, $\mathcal{A}_4(\textbf{w},\textbf{v})$ and $\mathcal{A}_5(\textbf{w},\textbf{v})$.} 
\label{area}
\end{subfigure}
\caption{The partition $\bigl(\mathcal{A}_m(\textbf{w},\textbf{v})\bigr)_{m=1,\dots,8}$ of $\mathbb{R}^3$ from various perspectives.}\label{partitionAll3}
\end{figure}
 
    We now return to considering the full expression in \eqref{eq:two_step}.
    Using the above partition and setting $p_q$ to be the probability that $(\bar{D}^{(i)})_{i\in\{1,\dots,d\}}$ is equal to 
    $q\in\mathbb{N}^d$, we can upper bound \eqref{eq:two_step} by
    \begin{equation}\begin{split}
         & u   \sum\limits_{j_{_{0}}\in\mathbb{N}^d} p_{j_{_{0}}}
        \prod\limits_{i=1}^{d} 4j_{_{0}}^{_{(i)}}+
         u   \sum\limits_{j_{_{2}}\in\mathbb{N}^d} p_{j_{_{2}}}\prod\limits_{i=1}^{d} 4j_{_{2}}^{_{(i)}}  \\
         & \hspace{0pt}+ \sum\limits_{m=1}^8  \mathbb{E}_0 \Bigl[ \sum\limits_{_{j_{_{0}},j_{_{2}} \in\mathbb{N}^d} }\sum\limits_{_{\textbf{x},\textbf{y}\in\mathcal{X}}} \mathbbm{1}_{\textbf{0}\sim \textbf{x}} \mathbbm{1}_{\textbf{x}\sim \textbf{y}}  \mathbbm{1}_{\bar{D}_{0}=j_{_{0}}}\mathbbm{1}_{\bar{D}_{y}=j_{_{2}}} \mathbbm{1}_{x\in\mathcal{A}_m(\textbf{0},\textbf{y})}  \mathbbm{1}_{y \not\in 2\bar{R}_0 }\mathbbm{1}_{0 \not\in 2\bar{R}_y} \mathbbm{1}_{\bar{R}_0\cap\bar{R}_y=\emptyset}\Bigr]\\
\end{split}\notag
\end{equation}
\begin{equation}
\begin{split}         
         &\hspace{0pt}=:u  4^{d}  \sum\limits_{j_{_{0}}j_{_{2}}\in\mathbb{N}^d} p_{j_{_{0}}}p_{j_{_{2}}}\Bigl(\prod\limits_{i=1}^{d} j_{_{0}}^{_{(i)}}+\prod\limits_{i=1}^{d} j_{_{2}}^{_{(i)}} \Bigr)  + \sum\limits_{m=1}^8 S_m,
    \end{split}\label{eq:sumS}\end{equation}
    where each individual $S_m$ corresponds to the expectation in \eqref{eq:sumS} within the area $\mathcal{A}_m$. 
    \smallskip

    We abbreviate $\gamma=2j_0^{(d)}$ and get upper bounds for $(S_m)_{m=1,\dots,8}$ as 
    \begin{equation*}        
        S_m \leq \, u^2 \, \sum\limits_{_{j_{_{0}},j_{_{2}}\in\mathbb{N}^d} }p_{j_{_{0}}}p_{j_{_{2}}}\int\limits_{\mathbb{R}^d\setminus B_{\gamma}(0)} \, \int\limits_{\mathcal{A}_m(\textbf{0},\textbf{y})} \mathbb{P}_{0,x,y}(\textbf{0}\sim \textbf{x}, \textbf{x}\sim \textbf{y},\textbf{0}\not\sim\textbf{y},0\not\in2\bar{R}_y) \,\text{d}\lambda (x)\text{d}\lambda(y).
    \end{equation*}
    Looking at $\mathbb{P}_{0,x,y}(\textbf{0}\sim \textbf{x}, \textbf{x}\sim \textbf{y}, \textbf{0}\not\sim\textbf{y},0\not\in2\bar{R}_y) $ we use that the event $\{\textbf{0}\sim \textbf{x}, \textbf{x}\sim \textbf{y}\}$ roughly corresponds to the event that the diameters of $\textbf{x}$ are ``big enough'' and their orientations are ``good enough'' so that intersections of the rectangles $\bar R_0$ and $\bar R_x$, and of $\bar R_x$ and $\bar R_y$ are possible. In addition to that we bound the permissible area for the orientation of the diameters of $\bar{R}_x$ from above by assuming that the largest face of $\bar R_0$, resp.~$\bar R_y$, is perpendicular to the vector $x\in\mathbb{R}^d$, resp.~$y-x\in\mathbb{R}^d$. It can easily be checked that given $j_0$, resp.~$j_2$, the set of rotations that result in an intersection under this assumption is larger than for any other rotation of $\bar R_0$, resp.~$\bar R_y$.  With this we have 
    \begin{equation*}
        \mathbb{P}_{0,x,y}\biggl(\begin{array}{c}\textbf{0}\sim \textbf{x}, \textbf{x}\sim \textbf{y},\\\textbf{0}\not\sim\textbf{y},0\not\in2\bar{R}_y\end{array}\biggr) \leq \sum_{k\in \mathcal K}\,\mathbb{P}_{0,x,y}\biggl(\begin{array}{c}\!\!\bar{D}_{x}^{_{(k)}}\geq \max\bigl(\dist(\bar{R}_{{0}} , x), \dist(\bar{R}_{{y}} , x)\bigr)\,, \\\bar{R}_{{0}}\cap \bar{R}_{{x}}\neq \emptyset, 
        \bar{R}_{{x}}\cap \bar{R}_{{y}}\neq \emptyset,\,\textbf{0}\not\sim\textbf{y},0\not\in2\bar{R}_y\end{array}\biggr),
    \end{equation*}
    where $\mathcal K$ denotes the set of $k\in\{1,\dots,d-1\}$ for which $(c)$ holds, and we have decomposed the set according to the value of $k$ in our assumption.
In particular the inequality involving $\beta$ and relationships of diameters are required to hold for the given value of $k$.
    From now on we fix $k$ throughout the computations and carry out the argument for this fixed value. 
    \pagebreak[3]
    
    To finish the base clause we show the following lemma.
     \begin{lemma}\label{Lemma:S1toS8}
    For $0\mathrel{{\le}} \beta < \min\big\{\frac{\alpha_k-2k}{2(d-k)},\frac{\alpha_k}{2d}\big\}$ there exists $c\in(0,\infty)$ such that
    \begin{equation*}
        S_m \leq \, cu^2\, \sum\limits_{_{j_{_{0}},j_{_{2}}\in\mathbb{N}^d} }p_{j_{_{0}}}p_{j_{_{2}}}   \prod\limits_{s=1}^{{d}} j_{_{2}}^{_{(s)}}  \, \prod\limits_{l=1}^{{d}} j_{_{0}}^{_{(l)}}, 
    \end{equation*}
    for $m\in\{1,\dots,8\}$.
    \end{lemma}
    
    \begin{proof}
    In this proof we use symmetry of the partition and therefore look at different $m$ at the same time.
    But before looking at the different cases for $m$ note that we have deal with three parts in every case for $m$. The first one is that the $k$th diameter is of order $|x|+|y|$, so that the connector is of $2\bar{R}_0$ and $2\bar{R}_y$ is big enough so that $\textbf{0}\sim\textbf{x}$ and $\textbf{x}\sim\textbf{y}$ can occur. In that part we have that the $j$th diameter for $j>k$ is bounded from above by $(|x|+|y|)^{\beta}$. This is the part we are focusing in the induction part.
    \smallskip
    
    The next part is that the $k$th diameter is of order $(|x|+|y|)^{1/\delta}$ for $\delta\in(\beta,1)$. In this case one can repeat the whole calculation that follows for all $m$ by replacing $\alpha_k-\varepsilon$ by $(\alpha_k-\varepsilon)/\delta$ and all $\beta$ by $\beta/\delta$ and get bigger bounds for $\beta$ compared in the first part. {More precisely, one splits the range of the $k$th diameter into the dyadic scales $[2^{i}(|x|+|y|),2^{i+1}(|x|+|y|))$, $i\in\mathbb{N}$, and observes that the bound obtained for scale $i$ is the bound from the first part multiplied by at most $2^{-i(\alpha_k-\varepsilon-2(d-k)\beta)}$, as the tail of the $k$th diameter contributes a factor $2^{-i(\alpha_k-\varepsilon)}$ and each of the at most $2(d-k)$ orientation factors a factor $2^{i\beta}$. This is summable in $i$ since $\beta<\frac{\alpha_k}{2d}<\frac{\alpha_k}{2(d-k)}$.} As all calculations are similar to the one in the first parts following here, we do not go into detail for this~part.\smallskip\pagebreak[3]

    The third part for all $m$ is the case where the $k$th diameter is bigger than $(|x|+|y|)^{1/\beta}${; for $\beta=0$ this case is empty and the second part covers all $\delta\in(0,1)$}. We bound here the $j$th diameters for all $j>k$ from above by $|x|+|y|$ by the assumption on the diameters. This leads to the fact that the rotation part in the connection probability of a vertex between the vertices with location $x$ and $y$ is not relevant anymore. To get a restriction on $\beta$ in this part we just have to look at
    \begin{equation*}
        \int\limits_{B_1^c(0)}\int\limits_{B_1^c(0)} (|x|+|y|)^{-(\alpha_k-\varepsilon)/\beta}\text{d}\lambda(x)\text{d}\lambda(y)<\infty
    \end{equation*}
    which is finite if $\beta<\tfrac{\alpha_k}{2d}$. 
\smallskip 

    From here on we consider the first part and in order to keep notation a bit easier we do not write $D^{_{(k)}}\asymp |x|+|y|$ in the calculations except from the case $m\in\{1,2\}$. 
    \paragraph{$m\in\{1,2\}$:}
    We first start with the proof for $m=1$ and get the same upper bound for $m=2$. We consider a given pair $j_{_{0}},j_{_{2}}\in\mathbb{N}^d$ and therefore  bound the integral of $S_1$ from above by
    \begin{align*}
        %
        & \int\limits_{{\mathbb{R}^d\setminus B_{\gamma}(0)}}  \int\limits_{{\mathcal{A}_1(\textbf{0},\textbf{y})}} \mathbb{P}_{0,x,y}\biggl(\begin{array}{c}\bar{D}_{x}^{_{(k)}} {  \asymp } \max\bigl(\dist(\bar{R}_{{0}} , x), \dist(\bar{R}_{{y}} , x)\bigr)\,, \\\bar{R}_{{0}}\cap \bar{R}_{{x}}\neq \emptyset, 
        \bar{R}_{{x}}\cap \bar{R}_{{y}}\neq \emptyset,\,\textbf{0}\not\sim\textbf{y}\end{array}\biggr)
        \,\text{d}\lambda (x)\text{d}\lambda(y)\\
        &\leq    \,\int\limits_{\mathbb{R}^d\setminus B_{\gamma}(0)} \, \int\limits_{\mathcal{A}_1(\textbf{0},\textbf{y})} c(|x|+|y|-(|x|+|y|)^{\beta})^{-(\alpha_k- \varepsilon)} \prod\limits_{s=1}^{d-k} \frac{j_{_{2}}^{_{(s)}}+(|x|+|y|)^{\beta}}{|x|+|y|} \,\text{d}\lambda (x)\text{d}\lambda(y). 
        \end{align*}
        We used that $\max\{\dist(\bar{R}_{{0}} , x), \dist(\bar{R}_{{y}} , x)\} \asymp |x|+|y| - (|x|+|y|)^{\beta}$ and the above assumption on the orientation of $\bar{R}_x$ to the  inequality. {Here and below we are using the Potter bounds for regularly varying functions, resulting in the $\varepsilon$ corrections in the exponent.} 
        Using that $(|x|+|y|)^{\beta}>1$ we can replace the summand $(|x|+|y|)^{\beta}$ in the product over $s$ and bound the integral from above by
        \begin{align*}
        %
        %
        c\prod\limits_{s=1}^{d-k} j_{_{2}}^{_{(s)}} \int\limits_{\mathbb{R}^d\setminus B_{\gamma}(0)} \int\limits_{\mathcal{A}_1(\textbf{0},\textbf{y})}&\, c(|x|+|y|)^{-(\alpha_k- \varepsilon +(d -k))+\beta(d-k)}  \,\text{d}\lambda(x)\text{d}\lambda(y) \\
        &\leq  c\prod\limits_{s=1}^{d-k} j_{_{2}}^{_{(s)}} \int\limits_{\mathbb{R}^d\setminus B_{\gamma}(0)} \, |y|^{-\alpha_k+ \varepsilon +(d-k)(-1+\beta) }  \lambda( \mathcal{A}_1(\textbf{0},\textbf{y}) )\,\text{d}\lambda(y)\\
        &\leq  c \prod\limits_{s=1}^{{d}} j_{_{2}}^{_{(s)}}  \, \prod\limits_{l=1}^{{d}} j_{_{0}}^{_{(l)}}   \,\int\limits_{\mathbb{R}^d\setminus B_1(0)} |y|^{-\alpha_k+ \varepsilon +(d-k)(-1+\beta) }   \,\text{d}\lambda(y) \\
        &\leq  c \prod\limits_{s=1}^{{d}} j_{_{2}}^{_{(s)}}  \, \prod\limits_{l=1}^{ {d}} j_{_{0}}^{_{(l)}} \int\limits_{1}^{\infty} r^{d-1-\alpha_k+ \varepsilon +(d-k)(-1+\beta)} \,\text{d} r.
    \end{align*}
    In the first inequality we use that the exponent $-(\alpha_k- \varepsilon +(d -k))+\beta(d-k)$ is negative and $\beta\in {[0,1)}$ so that we bound $|x|$ from below by $0$. Note that $c\in(0,\infty)$ depends here and also in the rest of the section on the model parameters and may change in every inequality. 
    In the third inequality we  {use $\lambda(\mathcal{A}_1(\textbf{0},\textbf{y}))=c\prod_{l=1}^{d} j_{_{0}}^{_{(l)}}$ and $j_{_{2}}^{_{(s)}}\geq 1$, and} also that the diameters of $\bar{R}_0$ are all at least $1$, which has as a consequence that $|y|\geq1$.
    Remember that the Potter bounds are given for every $\varepsilon>0$ and so taking $\varepsilon$ small enough, the last integral is finite if $\beta<\tfrac{\alpha_k-k}{d-k}$, 
     which is the case. We consequently have
    \begin{equation}
        S_1 \leq \, cu^2\, \sum\limits_{_{j_{_{0}},j_{_{2}}\in\mathbb{N}^d} }p_{j_{_{0}}}p_{j_{_{2}}}   \prod\limits_{s=1}^{ {d}} j_{_{2}}^{_{(s)}}  \, \prod\limits_{l=1}^{ {d}} j_{_{0}}^{_{(l)}},\label{eq:S1}
    \end{equation}
    and can obtain an analogous upper bound for $S_2$, namely
    \begin{equation}
        S_2  \leq \, cu^2\, \sum\limits_{_{j_{_{0}},j_{_{2}}\in\mathbb{N}^d} }p_{j_{_{0}}}p_{j_{_{2}}}  \prod\limits_{s=1}^{ {d}} j_{_{0}}^{_{(s)}}  \, \prod\limits_{l=1}^{ {d}} j_{_{2}}^{_{(l)}}.\label{eq:S2}
    \end{equation}
    By using similar observations, i.e.\ $\max\{\dist(\bar{R}_{{0}} , x), \dist(\bar{R}_{{y}} , x)\} \asymp |x|+|y|$ in $\mathcal{A}_m(\textbf{0},\textbf{y})$ $m\in\{3,4,6,7,8\}$ and $\max\{\dist(\bar{R}_{{0}} , x), \dist(\bar{R}_{{y}} , x)\} \asymp |y|$ in $\mathcal{A}_5(\textbf{0},\textbf{y})$, the calculation of the upper bounds of $S_m$ for $m\geq 3$ can be related to the bounds for $S_1$ and $S_2$. 
    
    \paragraph{$m\in\{3,4\}$:} For $S_3$ we, like before, focus first on the integral with fixed $j_{_{0}},j_{_{2}}\in\mathbb{N}^d$ and get, defining $j:=2(j_{_{0}}^{_{(d)}}+j_{_{2}}^{_{(d)}})$, that
    \begin{align*}
          \int\limits_{\mathbb{R}^d\setminus B_{j}(0)} &\, \int\limits_{\mathcal{A}_3(\textbf{0},\textbf{y})} \mathbb{P}_{0,x,y}(\textbf{0}\sim \textbf{x}, \textbf{x}\sim \textbf{y},\textbf{0}\not\sim\textbf{y},0\not\in2\bar{R}_y)\ \,\text{d}\lambda (x)\text{d}\lambda(y) 
        \\
        &\leq  c \, \smash{\int\limits_{{\mathbb{R}^d\setminus B_{j}(0)} }\int\limits_{\mathcal{A}_3(\textbf{0},\textbf{y})} \prod\limits_{s=1}^{d-k} \frac{\min\{j_{_{0}}^{_{(s)}},j_{_{2}}^{_{(s)}}\}+(|x|+|y|)^{\beta}}{|x|+|y|}} \\
        &\hspace{130pt}\times(|x|+|y|-(|x|+|y|)^{\beta})^{-(\alpha_k- \varepsilon)} \,\text{d}\lambda (x)\text{d}\lambda(y) 
        \end{align*}
        \begin{align*}
        &\leq  c \prod\limits_{s=1}^{d-k} \min\{j_{_{0}}^{_{(s)}},j_{_{2}}^{_{(s)}}\}\, \int\limits_{{\mathbb{R}^d\setminus B_{j}(0)} }\int\limits_{\mathcal{A}_3(\textbf{0},\textbf{y})} (|x|+|y|)^{-\alpha_k+ \varepsilon+(d-k)(-1+\beta) } \,\text{d}\lambda (x)\text{d} (y).      \end{align*}
        In the first inequality we use the Potter bound to get the ${-(\alpha_k-\varepsilon)}$ exponent. The product in the integral appears since we consider orientations of $\bar R_0$ and $\bar R_y$ such that all of the diameters have pairwise the same orientation (i.e.\ the first diameter of $\bar R_0$ has the same orientation as the first diameter of $\bar R_y$, and similarly for the remaining $d-1$ diameters), which gives an upper bound for the largest set of orientations of $\bar R_x$ that result in an intersection with both. The $\min\{j_0^{(s)},j_{2}^{(s)}\}$ part appears as $\bar R_x$ has to intersect $\bar R_0$ and $\bar R_y$ so that the minimum determines the size of this largest set. Together with $|x|+|y|$ and using Remark \ref{rem:winkel} this gives the probability of an appropriate orientation of $\bar R_x$ existing. In addition to that we get twice the $(|x|+|y|)^{\beta}$ term for the upper bound for the smaller $d-k$ diameters effecting the orientation term and the distance term i.e.\ the term with the $-(\alpha-\varepsilon)$ exponent. %
\medskip

        \begin{figure}[!ht]
    \begin{annotate}{\includegraphics[width=0.9\textwidth]{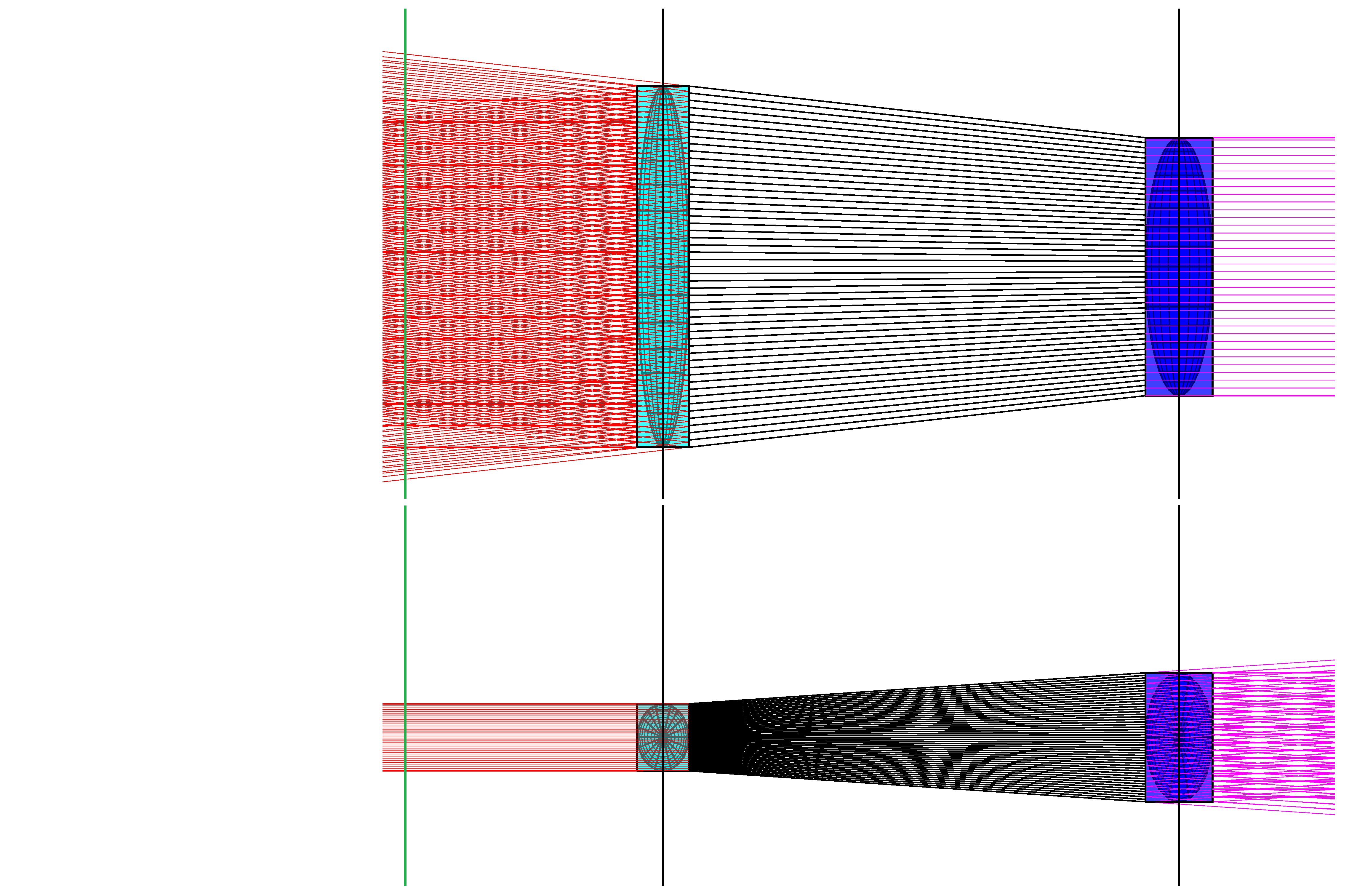}}{1}
        \draw[->] (-1.8*0.9,4.7*0.9) -- (-2.95*0.9,4.7*0.9);
        \draw[line width=1.5pt] (-2.95*0.9,4.3*0.9) -- (-2.95*0.9,5.1*0.9);
        \draw[->] (-1*0.9,4.7*0.9) -- (0.1*0.9,4.7*0.9);
        \node at (-1.4*0.9,4.7*0.9) {$\ell$};
        \draw[->] (4.0*0.9,4.7*0.9) -- (5*0.9,4.7*0.9);
        \draw[line width=1.5pt] (5*0.9,4.3*0.9) -- (5*0.9,5.1*0.9);
        \draw[->] (1.2*0.9,4.7*0.9) -- (0.1*0.9,4.7*0.9);
        \draw[line width=1.5pt] (0.1*0.9,4.3*0.9) -- (0.1*0.9,5.1*0.9);
        \node at (2.625*0.9,4.7*0.9) {$|y|-j/2$};
        \draw[line width=1.5pt] (0.1*0.9,0) -- (0.1*0.9,-0.72);
        \draw[line width=1.5pt] (-0.4,0) -- (-0.4,-0.72);
        \draw[line width=1.5pt] (0.1*0.9,-2.48) -- (0.1*0.9,-1.76);
        \draw[line width=1.5pt] (-0.4,-2.48) -- (-0.4,-1.76);
        \draw[<->] (0.1*0.9,-0.36) -- (-0.4,-0.36);
        \draw[<->] (0.1*0.9,-2.12) -- (-0.4,-2.12);
        \node at (1,-1.3) {$2j_0^{_{(3)}}$};
        \draw[line width=1.1pt] [->] (0.6,-1.25) .. controls (-0.155,-1.25) .. (-0.155,-0.38);
        \draw[line width=1.1pt] [->] (0.6,-1.35) .. controls (-0.155,-1.35) .. (-0.155,-2.10);

        \draw (-1.5*0.9,2.3*0.9) -- (-4.8*0.9, 4.2*0.9);
        \draw (-1.5*0.9,-3.2*0.9) -- (-1.4*0.9, -4.5*0.9 );
        \fill (-1.5*0.9,2.3*0.9) circle (3pt);
        \fill (-1.5*0.9,-3.2*0.9) circle (3pt);
        \node at (-5.5*0.9,4.5*0.9) {$\mathcal{A}_3(\textbf{0},\textbf{y})$};
        \node at (-1.4*0.9,-4.8*0.9) {$\mathcal{A}_3(\textbf{0},\textbf{y})$};
        %
        \draw[dashed,line width=0.5pt] (-2.95*0.9,-2.75*0.9) -- (-4.3*0.9,-2.75*0.9);
        \draw[dashed,line width=0.5pt] (-2.95*0.9,-3.5*0.9) -- (-4.3*0.9,-3.5*0.9);
        \draw[bend left=30] (-4.4*0.9, -3.05*0.9) to (-4.3*0.9,-2.775*0.9);
        \draw[bend left=30] (-4.3*0.9,-3.525*0.9) to (-4.4*0.9, -3.225*0.9);
        \draw[bend left=30] (-4.45*0.9, -3.15 *0.9)to (-4.4*0.9,-3.225*0.9);
        \draw[bend right=30] (-4.45*0.9, -3.15*0.9 )to (-4.4*0.9,-3.025*0.9);
        \node at (-6.2*0.9,-3.15*0.9) {$2\min\Bigl\{j_{_0}^{_{(2)}},j_{_2}^{_{(2)}}\Bigr\}=$};  
        \draw[dashed,line width=0.5pt] (-2.95*0.9, 4.25*0.9) -- (-4.3*0.9,4.25*0.9);
        \draw[dashed,line width=0.5pt] (-2.95*0.9, -0.35*0.9) -- (-4.3*0.9,-0.35*0.9);
        \draw[bend left=15] (-4.4*0.9, 2.07*0.9) to (-4.3*0.9,4.22*0.9);
        \draw[bend left=15] (-4.3*0.9,-0.28*0.9) to (-4.4*0.9, 1.87*0.9);
        \draw[bend left=30] (-4.5*0.9, 1.97*0.9 )to (-4.4*0.9,1.87*0.9);
        \draw[bend right=30] (-4.5*0.9, 1.97 *0.9)to (-4.4*0.9,2.07*0.9);
        \node at (-6.9*0.9,2.47*0.9) {$2\Bigl(\max\Bigl\{j_{_{0}}^{_{(1)}},j_{_{2}}^{_{(1)}}\Bigr\} $};
        \node at (-4.8*0.9,1.97*0.9) {$=$};
        \node at (-6.7*0.9,1.63*0.9) {$+\,\,\frac{\ell|j_{_{0}}^{_{(1)}}-j_{_{2}}^{_{(1)}}|}{|y|-j/2}\Bigr)$};
    \end{annotate}
    \caption{Visualisation of the role of $\ell$ in the calculation for $S_3$, with the perspective along $\rot_{\rho_{v,w}}(e_2)$ in the top and $\rot_{\rho_{v,w}}(e_3)$ in the bottom figure. Note also that $\mathcal{A}_3'(\mathbf{0},\mathbf{y})$ represents the everything ``above'' and ``below'' of the turquoise box including the box itself. Consequently, $\mathcal{A}_3(\mathbf{0},\mathbf{y})\cap \mathcal{A}_3'(\mathbf{0},\mathbf{y})$ represents the two small red areas ``above'' and ``below'' the turquoise box.} 
    \label{fig:ell}
    \end{figure}
    
        The set $\mathcal{A}_3(\textbf{0},\textbf{y})$ can be split in two parts. For that we define
        \begin{equation*}
            \mathcal{A}_{3}^{\prime}(\textbf{0},\textbf{y}):= \rot_{\rho_{y,0}}\bigl([-\bar D ^{(d)}_0,\bar D ^{(d)}_0]\times \mathbb{R}^{d-1} \bigr) + 0
        \end{equation*}
        and look at $\mathcal{A}_{3}(\textbf{0},\textbf{y}) \cap \mathcal{A}_{3}^{\prime}(\textbf{0},\textbf{y}) $ and $ \mathcal{A}_{3}(\textbf{0},\textbf{y}) \setminus \mathcal{A}_{3}^{\prime}(\textbf{0},\textbf{y}) $. 
        As $-(\alpha_k- \varepsilon+d-k)$ is negative we bound the first part of $\mathcal{A}_3(\textbf{0},\textbf{y})$ 
        from above via
        \begin{align*}
            c& \prod\limits_{s=1}^{d-k} \min\{j_{_{0}}^{_{(s)}},j_{_{2}}^{_{(s)}}\}\, \int\limits_{{\mathbb{R}^d\setminus B_{j}(0)} }\,\int\limits_{{\mathcal{A}_3(\textbf{0},\textbf{y})\cap\mathcal{A}_{3}^{\prime}(\textbf{0},\textbf{y}) }}  (|x|+|y|)^{-\alpha_k+ \varepsilon+(d-k)(-1+\beta)  } \,\text{d}\lambda (x)\text{d}\lambda(y)
            \end{align*}
            \begin{align*}
            &\leq c \prod\limits_{s=1}^{ {d}} \min\{j_{_{0}}^{_{(s)}},j_{_{2}}^{_{(s)}}\}\prod\limits_{t=1}^{ {d}}\max\{j_{_{0}}^{_{(t)}},j_{_{2}}^{_{(t)}}\} \int\limits_{{\mathbb{R}^d\setminus B_{j}(0)}} |y|^{-\alpha_k+ \varepsilon+(d-k)(-1+\beta) } \text{d}\lambda(y)\\
            &\leq c \prod\limits_{s=1}^{ {d}} \min\{j_{_{0}}^{_{(s)}},j_{_{2}}^{_{(s)}}\}\prod\limits_{t=1}^{ {d}}\max\{j_{_{0}}^{_{(t)}},j_{_{2}}^{_{(t)}}\} \int\limits_{j}^{\infty} r^{d-1-\alpha_k+ \varepsilon +(d-k)(-1+\beta) } \text{d}r,
        \end{align*}
        which is finite for $\beta <\tfrac{\alpha_k-k}{d-k}$. In the first inequality we use that $|x|\geq0$ and bound the volume of the first part of $\mathcal{A}_3(\textbf{0},\textbf{y})$ from above by $$ {c\, j_{_{0}}^{_{(d)}}\prod\limits_{t=1}^{d-1} \max\{j_{_{0}}^{_{(t)}},j_{_{2}}^{_{(t)}}\}\leq c\prod\limits_{t=1}^{ {d}} \max\{j_{_{0}}^{_{(t)}},j_{_{2}}^{_{(t)}}\}}.$$ 

        To see how, we use that the minimal distance of  $\bar R_0$ and $\bar R_y$ is bigger then $j_0^{_{(d)}}+j_2^{_{(d)}}$ so that the boundary of $\mathcal{A}_3(\textbf{0},\textbf{y})$ has in the direction of the $i$-th diameter of $\mathcal{A}_1(\textbf{0},\textbf{y})$ distance from $\mathcal{A}_1(\textbf{0},\textbf{y})$ given bounded from above by $\max\{j_0^{_{(i)}},j_2^{_{(i)}}\}$ (see Figure \ref{fig:ell}).
\smallskip

        Looking now at $\mathcal{A}_{3}(\textbf{0},\textbf{y}) \setminus \mathcal{A}_{3}^{\prime}(\textbf{0},\textbf{y})$ we are interested in 
        \begin{equation*}
            \ell := \dist\bigl(x,\rot_{\rho_{y,0}}\bigl(\{0\}\times \mathbb{R}^{d-1} + \bar{D}_0^{(d)}e_1\bigr) +0\bigr),
        \end{equation*}
        that is, the distance from $x$ to the hyperplane orthogonal to $y-0$ and intersecting with $\mathcal{A}_1(\mathbf{0},\mathbf{y})$ that is furthest away from $x$ (see again Figure \ref{fig:ell}). This gives us the following upper bound for the second part of $\mathcal{A}_{3}(\textbf{0},\textbf{y})$ 
    \begin{align*}
        & c \prod\limits_{s=1}^{d-k} \min\{j_{_{0}}^{_{(s)}},j_{_{2}}^{_{(s)}}\}\, \!\!\!\!\!\!\!\!\!\int\limits_{{\mathbb{R}^d\setminus B_{j}(0)} } \int\limits_{j_{_{0}}^{_{(d)}}}^{\infty}  (\ell+|y|)^{-\alpha_k+ \varepsilon+(d-k)(-1+\beta)  } \prod \limits_{t=1}^{d-1} 2\bigg\{ \min\{j_{_{0}}^{_{(t)}},j_{_{2}}^{_{(t)}}\}\mathbbm{1}_{\min\{j_{_{0}}^{_{(t)}},j_{_{2}}^{_{(t)}}\}=j_{_{0}}^{_{(t)}}} \\
        & \hspace{60pt}+ \Bigl[\max\{j_{_{0}}^{_{(t)}},j_{_{2}}^{_{(t)}}\} + \frac{\ell|j_{_{0}}^{_{(t)}}-j_{_{2}}^{_{(t)}}|}{|y|-j/2} \Bigr]\mathbbm{1}_{ \max\{j_{_{0}}^{_{(t)}},j_{_{2}}^{_{(t)}}\} = j_{_{0}}^{_{(t)}}}\bigg\}\,\text{d}\ell\text{d}\lambda(y) \\
        &\leq c\prod\limits_{s=1}^{d-k} \min\{j_{_{0}}^{_{(s)}},j_{_{2}}^{_{(s)}}\} \prod \limits_{t=1}^{d-1}  \max\{j_{_{0}}^{_{(t)}},j_{_{2}}^{_{(t)}}\} \int\limits_{{\mathbb{R}^d\setminus B_{j}(0)} }\int\limits_{j_{_{0}}^{_{(d)}}}^{\infty}  (\ell+|y|)^{-\alpha_k+ \varepsilon+(d-k)(-1+\beta)  } \\
        &\hspace{195pt} \times\smash{\Bigl(1+\frac{\ell}{|y|-j/2}\Bigr)^{d-1} \,\text{d}\ell \text{d}\lambda(y)}
          \end{align*}
    \begin{align*}
        &\leq c\prod\limits_{s=1}^{ {d}} \min\{j_{_{0}}^{_{(s)}},j_{_{2}}^{_{(s)}}\} \prod \limits_{t=1}^{ {d}}  \max\{j_{_{0}}^{_{(t)}},j_{_{2}}^{_{(t)}}\}  \int\limits_{j}^{\infty} r^{d-1} (r/2)^{1-d} \int\limits_{j_{_{0}}^{_{(d)}}}^{\infty} (\ell+r)^{-\alpha_k+ \varepsilon-1+k+\beta(d-k)} \,\text{d}\ell\text{d}r ,
        %
    \end{align*}
    where we used in the last inequality that $|y|\geq j$ to obtain the term $(r/2)^{1-d}$.
    This integral is finite, since $\beta<\tfrac{\alpha_k-k-1}{d-k}$, which is the case. In summary we get
    \begin{equation}
        S_3\leq \,  c u^2 \, \sum\limits_{_{j_{_{0}},j_{_{2}}\in\mathbb{N}^d} }p_{j_{_{0}}}p_{j_{_{2}}}  \prod \limits_{s=1}^{ {d}} j_{_{0}}^{_{(s)}}j_{_{2}}^{_{(s)}}.\label{eq:S3}
    \end{equation}  
    Using the symmetry of the partition we get the same upper bound for $S_4$. 
    
    \paragraph{$m=5$:} For $S_5$ we~get
    \begin{align*}
        \int\limits_{{\mathbb{R}^d\setminus B_{\gamma}(0)}}& \, \int\limits_{{\mathcal{A}_5(\textbf{0},\textbf{y})} }\mathbb{P}_{0,x,y}(\textbf{0}\sim \textbf{x}, \textbf{x}\sim \textbf{y},\textbf{0}\not\sim \textbf{y},0\not\in2\bar{R}_y)  \,\text{d}\lambda (x)\text{d}\lambda(y) \\
        &\leq  c\, \int\limits_{{\mathbb{R}^d\setminus B_{\gamma}(0) }}\, \int\limits_{{\mathcal{A}_5(\textbf{0},\textbf{y})}} \prod\limits_{s=1}^{d-k} \frac{\min\{j_{_{0}}^{_{(s)}},j_{_{2}}^{_{(s)}}\} +|y|^{\beta}}{|y|-|y|^{\beta}} |y|^{-(\alpha_k- \varepsilon)} \,\text{d}\lambda (x)\text{d}\lambda(y). 
        \end{align*}
    As before we used the Potter bound to get $|y|^{-(\alpha_k-\varepsilon)}$, and consider again that the orientations of $\bar{R}_0$ and $\bar{R}_y$ are such that all of the diameters have pairwise the same orientation. Additionally the smallest diameters are taken to have the same orientation as $|y|$. This gives us the largest possible area relative to the location of the connector~$\mathbf{x}$, 
    which we use to obtain an upper bound for the size of the sets of orientations that result in an intersection of $\bar{R}_x$ with both $\bar{R}_0$ and $\bar{R}_y$, giving the stated inequality. For the orientation of $\textbf{x}$ we use Remark \ref{rem:winkel} as before.\smallskip
    
    Using how $\mathcal{A}_5(\textbf{0},\textbf{y})$ is defined and $\beta<1$, we bound the last expression from above as
    \begin{align*}
        c\, \int\limits_{{\mathbb{R}^d\setminus B_{\gamma}(0) }}&\,\prod\limits_{s=1}^{d-k} \frac{\min\{j_{_{0}}^{_{(s)}},j_{_{2}}^{_{(s)}}\}}{|y|} \lambda(\mathcal{A}_5(\textbf{0},\textbf{y}))|y|^{-\alpha_k+ \varepsilon+\beta(d-k)} \,\text{d}\lambda(y)\\
        &\leq
         c \prod\limits_{s=1}^{d-1} \min\{j_{_{0}}^{_{(s)}},j_{_{2}}^{_{(s)}}\}  \int\limits_{\mathbb{R}^d\setminus B_1(0)}|y|  \prod\limits_{t=1}^{d-1} (j_{_{0}}^{_{(t)}}+j_{_{2}}^{_{(t)}}) |y|^{-\alpha_k+ \varepsilon+(d-k)(-1+\beta)}\,\text{d} \lambda(y)\\
         & \leq c \prod\limits_{s=1}^{ {d}} \min\{j_{_{0}}^{_{(s)}},j_{_{2}}^{_{(s)}}\}  \prod\limits_{t=1}^{ {d}} (j_{_{0}}^{_{(t)}}+j_{_{2}}^{_{(t)}}) \int\limits_{1}^{\infty} r^{-\alpha_k+ \varepsilon+k+\beta(d-k)}\,\text{d} r.
    \end{align*}
    The last expression is finite, since $\beta <\tfrac{\alpha_k-k-1}{d-k}$ . For $t\in\{1,\dots,d-1\}$ we have $j_{_{0}}^{(t)}+j_2^{(t)}\leq 2\max\{j_{_0}^{(t)},j_{_2}^{(t)}\}$, so we can bound $S_5$ from above by
    \begin{equation}
        S_5 \leq \, c u^2\, \sum\limits_{_{j_{_{0}},j_{_{2}}\in\mathbb{N}^d} }p_{j_{_{0}}}p_{j_{_{2}}} \prod\limits_{s=1}^{ {d}} j_{_{0}}^{_{(s)}}j_{_{2}}^{_{(s)}}.\label{eq:S5}
    \end{equation}
    It remains to find bounds for $S_6, S_7$ and $S_8$. As before it is enough to find upper bounds for $S_6$ and $S_7$, as the upper bound for $S_6$ is also the upper bound for $S_8$ due to symmetry.
   
    \paragraph{$m\in\{6,8\}$:}: 
    For the integral of $S_6$ we get 
    \begin{align*}
        &\int\limits_{{\mathbb{R}^d\setminus B_{\gamma}(0)}} \, \int\limits_{{\mathcal{A}_6(\textbf{0},\textbf{y})}} \, \mathbb{P}_{0,x,y}(\textbf{0}\sim \textbf{x}, \textbf{x}\sim \textbf{y},\textbf{0}\not\sim\textbf{y},0\not\in2\bar{R}_y)\,\text{d}\lambda (x)\text{d}\lambda(y)
        \end{align*}
   \begin{align} \label{eq:S6_calc} 
        %
        %
        &\hspace{15pt}\leq c\, \int\limits_{{\mathbb{R}^d\setminus B_{\gamma}(0)} }\, \int\limits_{{\mathcal{A}_6(\textbf{0},\textbf{y})}}  \prod\limits_{s=1}^{d-k} \frac{(j_{_{0}}^{_{(s)}}+(|x|+|y|)^{\beta})(j_{_{2}}^{_{(s)}}+(|x|+|y|)^{\beta})}{(|x|+|y|)^2 } \\\notag
        &\hspace{140pt}\times(|x|+|y|-(|x|+|y|)^{\beta})^{-(\alpha_k- \varepsilon)} \,\text{d}\lambda (x)\text{d}\lambda(y) \\ 
        &\hspace{15pt}\leq c \,  \prod\limits_{s=1}^{d-k} j_{_{0}}^{_{(s)}}j_{_{2}}^{_{(s)}} \int\limits_{{\mathbb{R}^d\setminus B_{\gamma}(0)} }\, \int\limits_{{\mathcal{A}_6(\textbf{0},\textbf{y})}} (|x|+|y|)^{-\alpha_k+ \varepsilon+2(d-k)(-1+\beta)}\, \text{d}\lambda (x)\text{d}\lambda(y).\notag
    \end{align}
    In the first inequality, the product is an upper bound for the probability that $\bar{R}_0$ and $\bar{R}_y$ intersect with $\bar{R}_x$ under the condition that the diameters are large enough to result in an intersection. As in the previous calculations, we used Remark \ref{rem:winkel}. The second term, i.e.\ the term with the exponent $-(\alpha_k-\varepsilon)$ comes from the Potter bounds. Using $\beta<1$ we can bound the term to get the second inequality. 
    
    \begin{myrem}
        A careful reader might note that for $S_6$ and $S_7, S_8$ below it does not suffice that only the largest diameter of the connector is big. To ensure that $\bar{R}_x$ intersects with $\bar{R}_0$ and $\bar{R}_y$, at least two of the diameters of $\textbf{x}$ have to be large enough. This however has no effect on the calculations, beyond changing the value of~$c$ by a factor of $\frac{d-1}{d-2}$.
    \end{myrem}
    
    Focusing now on the construction of $\mathcal{A}_6(\textbf{0},\textbf{y})$, we begin by reformulating it. Using $j:=2(j_0^{_{(d)}}+j_2^{_{(d)}})$ as before, we define $a_i(\ell)$ for $i\in\{2,\dots,d\}$ and $\ell>0$, as
    \begin{align*}
        a_i(\ell) \!&:= e_i\min\{j_{_{0}}^{_{(i-1)}},j_{_{2}}^{_{(i-1)}}\}\mathbbm{1}_{\min\{j_{_{0}}^{_{(i-1)}},j_{_{2}}^{_{(i-1)}}\}= j_{_{0}}^{_{(i-1)}}}\\
        &\,\,+\!e_i\Bigl(\max\{j_{_{0}}^{_{(i-1)}},j_{_{2}}^{_{(i-1)}}\} \!+\! \frac{(\ell+j_0^{(d)})|j_{_{0}}^{_{(i-1)}}-j_{_{2}}^{_{(i-1)}}|}{|y|-j/2} \Bigr)\mathbbm{1}_{\max\{j_{_{0}}^{_{(i-1)}},j_{_{2}}^{_{(i-1)}}\}= j_{_{0}}^{_{(i-1)}}}
    \end{align*}
    (see Figure \ref{fig:a-i}; $a_i(\ell)$ is one of the two vertical values, depending on which of the two $i$-th diameters considered is the larger one). With this, we define
    \begin{align*}
        \mathcal{A}_6^{\prime}(j_0,j_2) :=\Big\{x\in\mathbb{R}^d:\, x^{(1)}\in (-\infty,0],\,x^{(i)}\in  \bigl(-\infty,-|a_i(x^{(1)})|\bigr]\cup \bigl[&|a_i(x^{(1)})|,\infty\bigr),\\
        &i\in\{2,\dots,d\}\Big\}.
    \end{align*}
    
      \begin{figure}[!ht]
    \begin{annotate}{\includegraphics[width=0.9\textwidth]{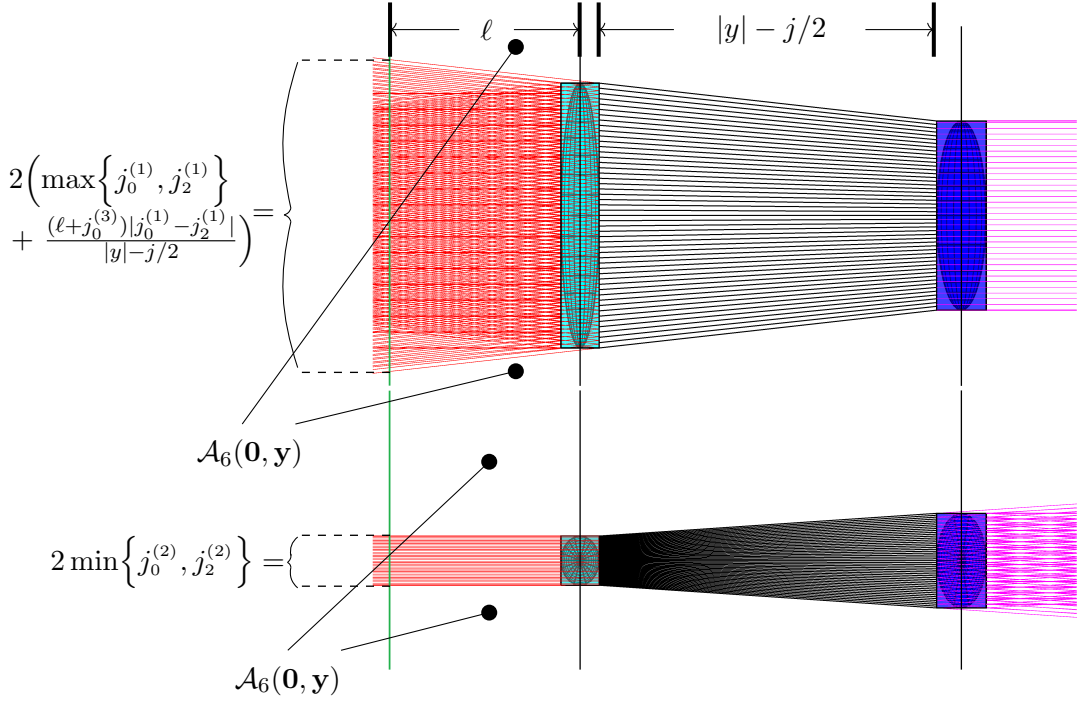}}{1}
        \draw[->] (-1.8*0.9-0.125,4.7*0.9) -- (-2.95*0.9,4.7*0.9);
        \draw[line width=1.5pt] (-2.95*0.9,4.3*0.9) -- (-2.95*0.9,5.1*0.9);
        \draw[->] (-1*0.9-0.125,4.7*0.9) -- (0.1*0.9-0.25,4.7*0.9);
        \draw[line width=1.5pt] (0.1*0.9-0.25,4.3*0.9) -- (0.1*0.9-0.25,5.1*0.9);
        \node at (-1.4*0.9-0.125,4.7*0.9) {$\ell$};
        \draw[->] (4.0*0.9,4.7*0.9) -- (5*0.9,4.7*0.9);
        \draw[line width=1.5pt] (5*0.9,4.3*0.9) -- (5*0.9,5.1*0.9);
        \draw[->] (1.2*0.9,4.7*0.9) -- (0.1*0.9,4.7*0.9);
        \draw[line width=1.5pt] (0.1*0.9,4.3*0.9) -- (0.1*0.9,5.1*0.9);
        \node at (2.625*0.9,4.7*0.9) {$|y|-j/2$};
        \draw (-1,4) -- (-4.6, -1);
        \fill (-1,4) circle (3pt);
        \draw (-1,-0.3) -- (-3.9, -1);
        \fill (-1,-0.3) circle (3pt);
        \node at (-4.5,-1.4) {$\mathcal{A}_6(\textbf{0},\textbf{y})$};
        
        \draw (-1.5*0.9,-1.5) -- (-4, -4);
        \fill (-1.5*0.9,-1.5) circle (3pt);
        \draw (-1.5*0.9,-3.5) -- (-3.3, -4);
        \fill (-1.5*0.9,-3.5) circle (3pt);
        \node at (-4,-4.4) {$\mathcal{A}_6(\textbf{0},\textbf{y})$};
        %

        \draw[dashed,line width=0.5pt] (-2.95*0.9,-2.75*0.9) -- (-4.3*0.9,-2.75*0.9);
        \draw[dashed,line width=0.5pt] (-2.95*0.9,-3.5*0.9) -- (-4.3*0.9,-3.5*0.9);
        \draw[bend left=30] (-4.4*0.9, -3.05*0.9) to (-4.3*0.9,-2.775*0.9);
        \draw[bend left=30] (-4.3*0.9,-3.525*0.9) to (-4.4*0.9, -3.225*0.9);
        \draw[bend left=30] (-4.45*0.9, -3.15 *0.9)to (-4.4*0.9,-3.225*0.9);
        \draw[bend right=30] (-4.45*0.9, -3.15*0.9 )to (-4.4*0.9,-3.025*0.9);
        \node at (-6.2*0.9,-3.15*0.9) {$2\min\Bigl\{j_{_0}^{_{(2)}},j_{_2}^{_{(2)}}\Bigr\}=$};  
        \draw[dashed,line width=0.5pt] (-2.95*0.9, 4.25*0.9) -- (-4.3*0.9,4.25*0.9);
        \draw[dashed,line width=0.5pt] (-2.95*0.9, -0.35*0.9) -- (-4.3*0.9,-0.35*0.9);
        \draw[bend left=15] (-4.4*0.9, 2.07*0.9) to (-4.3*0.9,4.22*0.9);
        \draw[bend left=15] (-4.3*0.9,-0.28*0.9) to (-4.4*0.9, 1.87*0.9);
        \draw[bend left=30] (-4.5*0.9, 1.97*0.9 )to (-4.4*0.9,1.87*0.9);
        \draw[bend right=30] (-4.5*0.9, 1.97 *0.9)to (-4.4*0.9,2.07*0.9);
        \node at (-6.9*0.9,2.47*0.9) {$2\Bigl(\max\Bigl\{j_{_{0}}^{_{(1)}},j_{_{2}}^{_{(1)}}\Bigr\} $};
        \node at (-4.8*0.9,1.97*0.9) {$=$};
        \node at (-6.7*0.9,1.63*0.9) {$+\,\,\frac{(\ell+j_0^{(3)})|j_{_{0}}^{_{(1)}}-j_{_{2}}^{_{(1)}}|}{|y|-j/2}\Bigr)$};
        
    \end{annotate}
    \caption{Visualisation of the role of $\ell$ in the calculation for $S_6$, with the perspective along $\rot_{\rho_{v,w}}(e_2)$ in the top and $\rot_{\rho_{v,w}}(e_3)$ in the bottom figure. Note that $\ell$ is different from $\ell$ in the calculation of $S_3$.} 
    \label{fig:a-i}
    \end{figure}
    \noindent
    Using this we can rewrite $\mathcal{A}_6(\textbf{0},\textbf{y})$ and bound \eqref{eq:S6_calc} from above by 
    \begin{align*}
        c\, \prod\limits_{s=1}^{d-k} & j_{_{0}}^{_{(s)}}j_{_{2}}^{_{(s)}} \int\limits_{\mathbb{R}^d\setminus B_{\gamma}(0)} \int\limits_{\rot_{\rho_{y,0}}\bigl(\mathcal{A}_6^{\prime}(j_0,j_2)\bigr)+0}\Bigl(|x|_1+|y| \Bigr)^{-\alpha_k+ \varepsilon+2(d-k)(-1+\beta)}  \,\text{d}\lambda(x)\text{d}\lambda(y) \\
        &\leq c\, \prod\limits_{s=1}^{d-k}  j_{_{0}}^{_{(s)}}j_{_{2}}^{_{(s)}} \int\limits_{\mathbb{R}^d\setminus B_1(0)} \Bigl(\sum\limits_{i=2}^{d} \min\{j_{_{0}}^{_{(i)}},j_{_{2}}^{_{(i)}}\} +|y| \Bigr)^{-\alpha_k+ \varepsilon+2(d-k)\beta -d+2k}  \,\text{d}\lambda(y) \\
        &\leq c\, \prod\limits_{s=1}^{d-k} j_{_{0}}^{_{(s)}}j_{_{2}}^{_{(s)}}  \int\limits_{1}^{\infty} r^{d-1}  r^{-\alpha_k+ \varepsilon+2(d-k)\beta -d+2k} dr,
    \end{align*}
    with $|\cdot|_1$ denoting the $1$-norm. In the first inequality we bound the integral from above by integrating over $x$ coordinate by coordinate, using the additivity of the $1$-norm and by increasing the integration area of $y$. Using polar coordinates yields the second inequality.
    The last expression is finite since $\beta<\tfrac{\alpha_k-2k}{2(d-k)}$. Note that $\tfrac{\alpha_k-2k}{2(d-k)}< \tfrac{\alpha_k-k-1}{d-k}$ so that this is the more  restrictive upper bound. 
    Therefore the upper bound for $S_6$ and so also for $S_8$ due to symmetry is given by
    \begin{equation}
        S_6,S_8\leq \, cu^2\, \sum\limits_{_{j_{_{0}},j_{_{2}}\in\mathbb{N}^d} }p_{j_{_{0}}}p_{j_{_{2}}}  \prod\limits_{s=1}^{ {d}} j_{_{0}}^{_{(s)}}j_{_{2}}^{_{(s)}}.\label{eq:S6}
    \end{equation}
    
   \paragraph{$m=7$:} For $S_7$ we similarly define 
   \begin{equation*}
       \mathcal{A}_7^{\prime}(j_{_{0}},j_{_{2}}):= (0,|y|)\times \Bigl(\bigtimes\limits_{i=2}^{d} \bigl(-\infty,-\min\{j_{_{0}}^{_{(i-1)}},j_{_{2}}^{_{(i-1)}}\}\bigr]\cup\bigl[\min\{j_{_{0}}^{_{(i-1)}},j_{_{2}}^{_{(i-1)}}\},\infty\bigr)\Bigr).
   \end{equation*}
    Using the definition of $\mathcal{A}_{7}(\textbf{0},\textbf{y})$  we bound the integral $S_7$ for fixed $j_{_{0}},j_{_{2}}$ from above by replacing $\mathcal{A}_{7}(\textbf{0},\textbf{y})$ in the bound of the integral with the larger $\mathcal{A}_{7}'(\textbf{0},\textbf{y})$. Moreover we use the $1$-norm as before and get the following upper bound, by using the same considerations about the orientation of $\bar{R}_x$ that result in an intersection with $\bar{R}_0$ and $\bar{R}_y$,  as well as the Potter bounds to obtain 
    \begin{align*}
        c\, \prod\limits_{s=1}^{d-k} j_{_{0}}^{_{(s)}}j_{_{2}}^{_{(s)}} &\int\limits_{\mathbb{R}^d\setminus B_{\gamma}(0)} \, \int\limits_{\rot_{\rho_{y,0}}\Bigl( \mathcal{A}_7^{\prime}(j_0,j_2)\Bigr) +0}(|x|_1+|y|)^{-\alpha_k+ \varepsilon+2(d-k)(-1+\beta)} \,\text{d}\lambda (x)\text{d}\lambda(y) 
                \end{align*}
        \begin{align*}
        &\leq c \, \prod\limits_{s=1}^{d-k} j_{_{0}}^{_{(s)}}j_{_{2}}^{_{(s)}} \int\limits_{\mathbb{R}^d\setminus B_1(0)} \Bigl(\sum\limits_{i=2}^{d} \min\{j_{_{0}}^{_{(i)}},j_{_{2}}^{_{(i)}}\}+|y|\Bigr)^{-\alpha_k+ \varepsilon+2(d-k)\beta-d+2k}  \text{d}\lambda(y)\\
        & \leq   c\, \prod\limits_{s=1}^{d-k} j_{_{0}}^{_{(s)}}j_{_{2}}^{_{(s)}}  \int\limits_{1}^{\infty} r^{d-1}  r^{-\alpha_k+ \varepsilon+2(d-k)\beta-d+2k} \text{d}r.
    \end{align*}
    In the first inequality we integrated over $x$, bounded the integrand from above and increased the integration area of $y$. We see that the last integral is finite since $\beta<\tfrac{\alpha_k-2k}{2(d-k)}$. 
    This yields 
    \begin{equation*}
        S_7\leq \, cu^2\, \sum\limits_{_{j_{_{0}},j_{_{2}}\in\mathbb{N}^d} }p_{j_{_{0}}}p_{j_{_{2}}}  \prod\limits_{s=1}^{ {d}} j_{_{0}}^{_{(s)}}j_{_{2}}^{_{(s)}}. 
         \vspace{-1.2cm}
    \end{equation*}
    \end{proof}
\bigskip

    As we are done with the proof of  Lemma~\ref{Lemma:S1toS8} the base clause is given, for $u\in(0,1)$ 
    by
    \begin{align}
        \mathbb{E}_0 \Bigl[\sum\limits_{\textbf{y}\in\mathcal{X}} \mathbbm{1}_{\textbf{0}\overset{2}{\sim} \textbf{y}}\Bigr] &\leq c\,u
        \sum\limits_{_{j_{_{0}},j_{_{2}}\in\mathbb{N}^d} }p_{j_{_{0}}}p_{j_{_{2}}} \biggl( \,\prod\limits_{i=1}^{ {d}} j_{_{0}}^{_{(i)}} +\,\prod\limits_{i=1}^{ {d}} j_{_{2}}^{_{(i)}} +  8 \prod\limits_{i=1}^{ {d}}j_{_{0}}^{_{(i)}}j_{_{2}}^{_{(i)}}\biggr)
        \notag\\ 
        & \leq 10\,c\,u\, \mathbb{E}\bigl[V^2\bigr]^{2},
    \label{eq:Induction1}\end{align}
    {where the second inequality holds since the sum equals $2\mathbb{E}[V]+8\mathbb{E}[V]^2$, and $\mathbb{E}[V]^2\leq\mathbb{E}[V^2]$ and $\mathbb{E}[V]\leq\mathbb{E}[V^2]\leq\mathbb{E}[V^2]^2$ as $V\geq 1$.}

    \paragraph{Induction step} 
     {As in the base clause, we carry out the computations in the induction step for a fixed $k\in \mathcal K$. We now explain why this suffices. The constant $c$ in Lemma~\ref{Lemma:S1toS8} depends on $k$ through $\alpha_k$ and the constants in the Potter bounds, and we write $c_k$ for it. In the base clause, $k$ is the index from the assumption of Theorem~\ref{Thm:two} for the connector~$\textbf{x}$, and summing over $k\in \mathcal K$ gives~\eqref{eq:Induction1} with $c=\sum_{k\in \mathcal K}c_k$. In the induction step we extend a path either by two edges with a new connector, see Proposition~\ref{Prop:IH}, or by one edge with a new last vertex, see Proposition~\ref{Prop:n_odd}. As in the base clause, we decompose according to the value of $k$ for this new vertex. For each $k\in \mathcal K$ the computations below give a bound with constant $c_k$, and summing over $k\in \mathcal K$ gives the same bound with constant $\sum_{k\in \mathcal K}c_k$. Hence Propositions~\ref{Prop:IH} and~\ref{Prop:n_odd} hold with $c=\sum_{k\in \mathcal K}c_k$ for all paths, regardless of which $k\in \mathcal K$ applies at each vertex. This constant does not depend on $n$, so choosing $u$ small enough makes the bounds summable over $n$. We use this in the proof of Theorem~\ref{Thm:two} at the end of this section.}\smallskip
    
    In the following proposition we deal with the case of paths of even length. {Recall that $V=\prod_{j=1}^d \bar{D}^{_{(j)}}$ satisfies $V^2\geq V\geq 1$, since $\bar{D}^{_{(j)}}\geq 1$ for all $j\in\{1,\dots,d\}$.}
    
    \begin{proposition}\label{Prop:IH}
        For $n\in\mathbb{N}$ we get the following inequality for the expected number of vertices that are connected via a path of length $2n$
    \begin{align*}
        \mathbb{E}_0 \Bigl[\sum\limits_{_{\textbf{x}\in\mathcal{X}} } \mathbbm{1}_{\textbf{0} \overset{2n}{\sim}\textbf{x}} \Bigr] &\overset{(IH)}{\leq}u^nc^n   \sum\limits_{_{j_{_{0}},j_{_{2}},\dots,j_{_{2n}}\in\mathbb{N}^d} }\!\!\!p_{j_{_{2n}}}\prod\limits_{k=0}^{n-1} p_{j_{_{2k}}}\biggl( \prod\limits_{i=1}^{ {d}} j_{_{2k}}^{_{(i)}} +  \prod\limits_{i=1}^{ {d}} j_{_{2k+2}}^{_{(i)}} +8 \prod\limits_{i=1}^{ {d}}j_{_{2k}}^{_{(i)}}j_{_{2k+2}}^{_{(i)}} \!\!\biggr) \\
        &\leq  u^n c^n 10^n \mathbb{E}\bigl[V^2\bigr]^{n+1}. 
    \end{align*} 
    \end{proposition}

    \begin{proof}
     Note that the second inequality can easily be checked by using that the diameters of different vertices are iid. For $n=1$, we have just proved the claim in \eqref{eq:Induction1}. We now proceed to show it for $n\geq 2$.\smallskip

    We split the expected number of vertices connected to $\textbf{0}$ via the partition \eqref{eq:partition} as we have done for $n=1$ and use induction over $n$. For that let $n\geq 2$. Recall that in the path of length $2n$ we are interested in vertices such that the last vertex satisfies $2\bar R_{x_{ {2n-2}}}\cap2\bar R_{x_{2n}}=\emptyset$ and count vertices that does not satisfy this as vertices that are the last vertex in a path of length $2n-1$. We obtain
    \begin{align*}
       \mathbb{E}_0 \Bigl[\sum\limits_{_{\textbf{x}\in\mathcal{X}}} \mathbbm{1}_{\textbf{0} \overset{2n}{\sim}\textbf{x}} \Bigr] \!&=\!  \, \mathbb{E}_0 \Bigl[\sum\limits_{_{\textbf{x}_2,\textbf{x}_4,\dots,\textbf{x}_{2n}\in\mathcal{X}}}^{\neq}  \prod\limits_{i=1}^n \mathbbm{1}_{\textbf{x}_{2i-2} \overset{2}{\sim}\textbf{x}_{2i}} \Bigr]\\[3mm]  &\hspace{-25pt}\leq\!\smash{\sum\limits_{m=1}^{8}\mathbb{E}_0 \Bigl[\sum\limits_{_{\substack{\textbf{x}_2,\textbf{x}_4,\dots,\textbf{x}_{2n-2}\in\mathcal{X}\\ j_{_{0}},j_{_{2}},\dots,j_{_{2n}}\in\mathbb{N}^d}} }\prod\limits_{i=0}^{n} \mathbbm{1}_{\bar{D}_{{x}_{2i}}=j_{_{2i}}}\mathbbm{1}_{\textbf{x}_{2i-2} \overset{2}{\sim}\textbf{x}_{2i}} 
       \mathbbm{1}_{x_{2n}\not\in2\bar{R}_{{x}_{2n-2}}} \mathbbm{1}_{x_{2n-2}\not\in2\bar{R}_{{x}_{2n}}}}
       \\
       &\hspace{155pt}  \times\mathbbm{1}_{x_{2n-1}\in\mathcal{A}_m(\textbf{x}_{2n-2}, \textbf{x}_{2n})}
        \mathbbm{1}_{\bar{R}_{{x}_{2n-2}}\cap \bar{R}_{{x}_{2n}} = \emptyset}  \Bigr]  
\end{align*}
\begin{align*}        
       &\quad+  
       \!\mathbb{E}_0 \Bigl[\sum\limits_{_{\substack{\textbf{x}_2,\textbf{x}_4,\dots,\textbf{x}_{2n-2}\in\mathcal{X}\\ j_{_{0}},j_{_{2}},\dots,j_{_{2n}}\in\mathbb{N}^d}}} \prod\limits_{i=0}^{n} \mathbbm{1}_{\bar{D}_{{x}_{2i}}=j_{_{2i}}}\mathbbm{1}_{\textbf{x}_{2i-2} \overset{2}{\sim}\textbf{x}_{2i}} \mathbbm{1}_{x_{2n}\in2\bar{R}_{{x}_{2n-2}}} \Bigr]  \\
       &\quad+\!\mathbb{E}_0 \Bigl[\sum\limits_{_{\substack{\textbf{x}_2,\textbf{x}_4,\dots,\textbf{x}_{2n-2}\in\mathcal{X}\\ j_{_{0}},j_{_{2}},\dots,j_{_{2n}}\in\mathbb{N}^d}} }\prod\limits_{i=0}^{n} \mathbbm{1}_{\bar{D}_{{x}_{2i}}=j_{_{2i}}}\mathbbm{1}_{\textbf{x}_{2i-2} \overset{2}{\sim}\textbf{x}_{2i}} \mathbbm{1}_{x_{2n-2}\in2\bar{R}_{{x}_{2n}}} \Bigr]  \!=:\! \sum\limits_{m=1}^{10} S_m^{_{(n)}}
    \end{align*}
   Using $\mathbb{P}_{r}$ as shorthand for $\mathbb{P}_{\textbf{0},\textbf{x}_1,\dots,\textbf{x}_r}$, we can write $S_m^{_{(n)}}$, for $m\in\{1,\dots,8\}$, as

   \begin{align*}
        S_m^{_{(n)}}\!=\! \smash{u^{2n}\!\!\!\!\!\!\!\!\!\sum\limits_{_{j_{_{0}},j_{_{2}},\dots,j_{_{2n}}\in\mathbb{N}^d} }\prod\limits_{i=0}^n p_{j_{_{2i}}} \!\!\int\limits_{\mathbb{R}^d}\!\!\dots\!\int\limits_{\mathbb{R}^d}}\! \mathbb{P}_{2n}&\Biggl(\begin{array}{c}\bigcap_{i=1}^{2n} \{\textbf{x}_{i-1} \sim\textbf{x}_{i}\}\cap\{\textbf{x}_{2n-2}\not\sim \textbf{x}_{2n}\}\\ \{x_{2n-2}\not\in2\bar{R}_{x_{2n}}\}\cap\{x_{2n}\not\in2\bar{R}_{x_{2n-2}}\}\end{array}\Biggr) \\
        &\hspace{10pt}\times\hspace{-3pt}\mathbbm{1}_{x_{2n-1}\in\mathcal{A}_m(\textbf{x}_{2n-2}, \textbf{x}_{2n})}\,\text{d}\lambda(x_1) \dots \text{d}\lambda(x_{2n}),
    \end{align*}
    while for $m\in\{9,10\}$ we have 
    
    \begin{align*}
        S_9^{_{(n)}}\!&=\! u^{2n}\hspace{10pt}\smash{\sum\limits_{_{j_{_{0}},j_{_{2}},\dots,j_{_{2n}}\in\mathbb{N}^d} }\prod\limits_{i=0}^n p_{j_{_{2i}}} \!\!\int\limits_{\mathbb{R}^d}\!\!\dots\!\int\limits_{\mathbb{R}^d}}\! \mathbb{P}_{2n}(\bigcap_{i=1}^{2n-2} \{\textbf{x}_{i-1} \sim\textbf{x}_{i}\}\cap\{x_{2n}\in2\bar{R}_{{x}_{2n-2}}\}) \\
        &\hspace{270pt}\text{d}\lambda(x_1) \dots \text{d}\lambda(x_{2n})\\
        \intertext{and}
        S_{10}^{_{(n)}}\!&=\! u^{2n}\hspace{10pt}\smash{\sum\limits_{_{j_{_{0}},j_{_{2}},\dots,j_{_{2n}}\in\mathbb{N}^d} }\prod\limits_{i=0}^n p_{j_{_{2i}}} \!\!\int\limits_{\mathbb{R}^d}\!\!\dots\!\int\limits_{\mathbb{R}^d}\!} \mathbb{P}_{2n}(\bigcap_{i=1}^{2n-2} \{\textbf{x}_{i-1} \sim\textbf{x}_{i}\}\cap\{x_{2n-2}\in2\bar{R}_{{x}_{2n}}\}) \\
        &\hspace{270pt}\text{d}\lambda(x_1) \dots \text{d}\lambda(x_{2n}).
    \end{align*}
    Defining now $(F_{_m}^{_{(n)}})_{m\in\{1,\dots,10\}}$ for $n\in\mathbb{N}$ as 
    \begin{equation*}
        F_m^{_{(n)}}= \prod\limits_{s=1}^{ {d}}j_{_{2n-2}}^{_{(s)}}\prod\limits_{s=1}^{ {d}}j_{_{2n}}^{_{(s)}},\, \text{ with } 1\leq m\leq8,\,\,\,\,\,
        F_{9}^{_{(n)}}\!\!=\!\!\prod\limits_{s=1}^{d}j_{_{2n-2}}^{_{(s)}},\,\,\,\,\, F_{10}^{_{(n)}}\!\!=\!\!\prod\limits_{s=1}^{d}j_{_{2n}}^{_{(s)}},
    \end{equation*}
    we claim that $S_m^{_{(n)}}$ can be bounded from above for $m\in\{1,\dots,10\}$ and $n\in\mathbb{N}$ by
    \begin{equation*}
        S_m^{_{(n)}} \leq u^n c^{n} \hspace{-10pt}\sum\limits_{_{j_{_{0}},j_{_{2}},\dots,j_{_{2n}}\in\mathbb{N}^d}}\hspace{-10pt} p_{j_{_{2n}}} p_{j_{_{2n-2}}} \prod\limits_{k=0}^{n-2} p_{j_{_{2k}}}\biggl( \prod\limits_{i=1}^{ {d}} j_{_{2k}}^{_{(i)}} +  \prod\limits_{i=1}^{ {d}} j_{_{2k+2}}^{_{(i)}}  + 8 \prod\limits_{i=1}^{ {d}}j_{_{2k}}^{_{(i)}}j_{_{2k+2}}^{_{(i)}} \biggr) F_m^{_{(n)}}.
    \end{equation*}
    Observe that the case $n=1$ was shown in \eqref{eq:Induction1}, so it remains to consider $n>1$ and focus on the induction step. We show the claim for $m\in\{1,9,10\}$ and note that the inequalities for $m \in\{2,3,\dots,8\}$ can be proved similarly to $m=1$. We first focus on the integral terms of $S_m^{_{(n)}}$ before looking at the rest. Moreover, we roughly do the following. First, we look in every integral at the last three vertices in the path of length $2n$. We count how many vertices can connect $\bar{R}_{x_{2n-2}}$ with $\bar{R}_{x_{2n}}$ and ignore the rest of the path. After that we integrate over $x_{2n}$ before we move on to the rest by using the induction hypothesis. Note that in the following calculations we use the same consideration for $\textbf{x}_{2n}$, $\textbf{x}_{2n-1}$ and $\textbf{x}_{2n-2}$ as we do for $\textbf{y}$, $\textbf{x}$, and $\textbf{0}$ in the case $n=1$. In addition to that we define similarly to the calculation for $n=1$ the value $\tilde{\gamma}:=2j_{2n-2}^{(d)}$. \smallskip
    
    \paragraph{$m=1$:} We rewrite the integral of $S_1^{(n)}$ and bound it from above so we get 
    \begin{align*}
&\smash{\int\limits_{\mathbb{R}^d}\dots\hspace{-10pt}\int\limits_{\mathbb{R}^d\setminus B_{\tilde{\gamma}}(x_{2n-2})} 
    \int\limits_{{\mathcal{A}_{1}(\textbf{x}_{2n-2},\textbf{x}_{2n})}}}\hspace{-20pt} \mathbb{P}_{2n}\big(\bigcap_{i=1}^{_{2n}} \{\textbf{x}_{i-1} \sim \textbf{x}_{i}\}\cap \{\textbf{x}_{2n-2}\not\sim \textbf{x}_{2n}\}\big)  \text{d}\lambda(x_1) \dots \text{d}\lambda( x_{2n-2}) \\
    &\hspace{170pt}\text{d}\lambda(x_{2n})\text{d}\lambda(x_{2n-1}) \\[5mm]
    &\leq c \int\limits_{\mathbb{R}^d}\dots\int\limits_{\mathbb{R}^d\setminus B_{\tilde{\gamma}}(x_{2n-2})} \int\limits_{{\mathcal{A}_{1}(\textbf{x}_{2n-2},\textbf{x}_{2n})}} \hspace{-18pt} \mathbb{P}_{2n}\big(\bigcap_{i=1}^{_{2n-2}} \{\textbf{x}_{i-1} \sim\textbf{x}_{i}\} \cap\{\textbf{x}_{2n-2}\not\sim \textbf{x}_{2n}\}\big) \\ 
    &\hspace{80pt} \times (|x_{2n-1}-x_{2n-2}|+|x_{2n}-x_{2n-1}|) ^{-\alpha_k+ \varepsilon +(d-k)(-1+\beta)}\prod\limits_{s=1}^{d-k} j_{_{2n}}^{_{(s)}} \\ & \hspace{100pt}  \text{d}\lambda(x_1)\dots\text{d}\lambda( x_{2n-2})\text{d}   \lambda(x_{2n})\text{d}\lambda(x_{2n-1}).
    \end{align*}
    The inequality is derived as in the case $n=1$ by using that the connections between $\textbf{x}_{2n}$, $\textbf{x}_{2n-1}$ and $\textbf{x}_{2n-2}$ are conditionally independent of the preceding $2n-3$ vertices. We are now covering the cases where the $k$-th diameter of $\bar{R}_{x_{2n-1}}$ is big enough, i.e. $$\max\{\dist(\bar{R}_{x_{2n}},x_{2n-1}),\dist(\bar{R}_{x_{2n-2}},x_{2n-1})\} \asymp |x_{2n}-x_{2n-1}|+|x_{2n-2}-x_{2n-1}|.$$ Using this and the Potter bounds yields the sum and $\alpha_k-\varepsilon$ in the exponent of $|x_{2n-1}-x_{2n-2}|+|x_{2n}-x_{2n-1}|$. We assume also that the orientation of $\bar{R}_{x_{2n}}$ is such that the largest face of this rectangle is perpendicular to the vector $x_{2n}-x_{2n-1}$ to get an upper bound for the set of orientations that result in an intersection of $\bar{R}_{x_{2n-1}}$ with $\bar{R}_{x_{2n}}$. Note that here we again use Remark \ref{rem:winkel}. In the next step we use substitution and  \smash{$\lambda\bigl(\mathcal{A}_1(\textbf{x}_{2n-2},\textbf{x}_{2n})\bigr)=c\prod_{s=1}^{d} j_{_{2n-2}}^{_{(s)}} $} which leads to the following upper bound
    \begin{align*}    
    c&\smash{\int\limits_{\mathbb{R}^d}\dots \hspace{-10pt}\int\limits_{\mathbb{R}^d\setminus B_{\tilde{\gamma}}(x_{2n-2})} \int\limits_{\bigl(\mathcal{A}_{1}(\textbf{x}_{2n-2},\textbf{x}_{2n})-x_{2n-2}\bigr)} } \hspace{-6pt}\mathbb{P}_{2n-2}\Big(\bigcap_{i=1}^{_{2n-2}} \{\textbf{x}_{i-1} \sim\textbf{x}_{i}\}\Big) \\[5pt]
    &\hspace{20pt} \times (|\tilde{x}_{2n-1} | + |\tilde{x}_{2n}|) ^{-\alpha_k+ \varepsilon +(d-k)(-1+\beta)} \prod\limits_{s=1}^{d-k}j_{_{2n}}^{_{(s)}} \, \text{d}\lambda(x_1) \dots \text{d}\lambda( x_{2n-2}) \text{d}\lambda(\tilde{x}_{2n})\text{d}\lambda(\tilde{x}_{2n-1}) \\
    &\leq   c \prod\limits_{s=1}^{ {d}} j_{_{2n}}^{_{(s)}}\prod\limits_{s=1}^{d} j_{_{2n-2}}^{_{(s)}}  \int\limits_{\mathbb{R}^d}\dots\int\limits_{\mathbb{R}^d} \ \mathbb{P}_{2n}\Big(\bigcap_{i=1}^{_{2n-2}} \{\textbf{x}_{i-1} \sim\textbf{x}_{i}\}\Big)   \,\text{d}\lambda(x_1) \dots \text{d}\lambda( x_{2n-2}).  
    \end{align*}
    Together with the induction hypothesis this leads to the claimed upper bound for $S_1^{(n)}$.\smallskip

    \paragraph{$m=9$:} To bound the integral of $S_9^{(n)}$ from above we count every vertex inside $2\bar{R}_{x_{2n-2}}$, i.e. we assume that every vertex in this set is connected to $\textbf{x}_{2n-2}$ via a path of length~$2$. Therefore the integral of $S_9^{(n)}$ can be bounded from above via
    \begin{align*}
 \int\limits_{\mathbb{R}^d}\dots\int\limits_{\mathbb{R}^d}&\mathbb{P}_{2n-2}\Big(\bigcap_{i=1}^{_{2n-2}} \{\textbf{x}_{i-1} \sim\textbf{x}_{i}\}\cap\{x_{2n}\in2\bar{R}_{x_{2n-2}}\}\Big)   \text{d}\lambda(x_1) \dots \text{d}\lambda(x_{2n-2}) \text{d}\lambda(x_{2n}) 
 \end{align*}
  \begin{align*}      
        &\leq c\prod\limits_{s=1}^{ {d}} j_{_{2n-2}}^{_{(s)}}
        \int\limits_{\mathbb{R}^d}\dots\int\limits_{\mathbb{R}^d} \mathbb{P}_{2n-2}\Big(\bigcap_{i=1}^{_{2n-2}} \{\textbf{x}_{i-1} \sim\textbf{x}_{i}\}\Big)  \text{d}\lambda(x_1) \dots \text{d}\lambda(x_{2n-2}).
        \end{align*}
        Using now the induction hypothesis leads the required upper bound
        \begin{equation*}
        S_{9}^{_{(n)}} \leq  u^nc^{n} \sum\limits_{_{j_{_{0}},j_{_{2}},\dots,j_{_{2n}}\in\mathbb{N}^d}} p_{j_{_{2n}}}p_{j_{_{2n-2}}} \prod\limits_{k=0}^{n-2} p_{j_{_{2k}}}\biggl( \prod\limits_{i= {1}}^{d} j_{_{2k}}^{_{(i)}} +  \prod\limits_{i=1}^{ {d}} j_{_{2k+2}}^{_{(i)}} + 8 \prod\limits_{i=1}^{ {d}}j_{_{2k}}^{_{(i)}}j_{_{2k+2}}^{_{(i)}} \biggr)  \prod\limits_{s=1}^{ {d}} j_{_{2n-2}}^{_{(s)}} .
    \end{equation*}
    
    \paragraph{$m=10$:} Except for a single step the calculation is similar to the case $m=9$. Looking at the integral of $S_{10}^{_{(n)}}$ we count every vertex $\textbf{x}_{2n}$ such that $x_{2n-2}\in2\bar{R}_{x_{2n}}$. We therefore assume that each such vertex is connected to $\textbf{x}_{2n-2}$ via a path of length~$2$. In the first step we rewrite the integral by using the distributional symmetry of $\bar{R}$. So we get for the integral of $S_{10}^{(n)}$,
        \begin{align*}
        \int\limits_{\mathbb{R}^d}\dots\int\limits_{\mathbb{R}^d} &\mathbb{P}_{2n}\Big(\bigcap_{i=1}^{2n-2} \{\textbf{x}_{i-1}\sim\textbf{x}_{i}\}\cap\{x_{2n-2}\in 2\bar{R}_{{x}_{2n}}\}\Big)  \, \text{d}\lambda(x_1) \dots \text{d}\lambda(x_{2n})\\
        &\leq c\prod\limits_{s=1}^{ {d}} j_{_{2n-2}}^{_{(s)}}\int\limits_{\mathbb{R}^d}\dots\int\limits_{\mathbb{R}^d} \mathbb{P}_{2n-2}\Big(\bigcap_{i=1}^{_{2n-2}} \{\textbf{x}_{i-1} \sim\textbf{x}_{i}\}\Big)   \,\text{d}\lambda(x_1) \dots \text{d}\lambda(x_{2n-2}).
        \end{align*}
        The inequality follows from the distributional symmetry of $\bar{R}$. 
        Using now the induction hypothesis leads to the claimed upper bound
    \begin{equation*}
        S_{10}^{(n)}\leq   u^n c^{n} \hspace{-15pt}\sum\limits_{_{j_{_{0}},j_{_{2}},\dots,j_{_{2n}}\in\mathbb{N}^d}} p_{j_{_{2n}}}p_{j_{_{2n-2}}} \prod\limits_{k=0}^{n-2} p_{j_{_{2k}}} \biggl(\, \prod\limits_{i=1}^{ {d}} j_{_{2k}}^{_{(i)}} +  \prod\limits_{i=1}^{ {d}} j_{_{2k+2}}^{_{(i)}} +  8 \prod\limits_{i=1}^{ {d}}j_{_{2k}}^{_{(i)}}j_{_{2k+2}}^{_{(i)}} \biggr) \prod\limits_{s=1}^{ {d}} j_{_{2n-2}}^{_{(s)}}  .
    \end{equation*}
    Putting everything together, this gives us an upper bound for the expected number of vertices that are connected via a path of length~$2n$ to the origin, for $n\in\mathbb{N}$, that is
    \begin{equation*}
        \mathbb{E}_0 \Bigl[\sum\limits_{\textbf{x}\in\mathcal{X}} \mathbbm{1}_{\textbf{0} \overset{2n}{\sim}\textbf{x}} \Bigr] \leq   u^n c^{n}\sum\limits_{_{j_{_{0}},j_{_{2}},\dots,j_{_{2n}}\in\mathbb{N}^d}} p_{j_{_{2n}}} \prod\limits_{k=0}^{n-1} p_{j_{_{2k}}}\biggl( \,\prod\limits_{i=1}^{ {d}} j_{_{2k}}^{_{(i)}} +  \prod\limits_{i=1}^{ {d}} j_{_{2k+2}}^{_{(i)}} +8 \prod\limits_{i=1}^{ {d}}j_{_{2k}}^{_{(i)}}j_{_{2k+2}}^{_{(i)}}\!\!\biggr).
    \end{equation*}
    \end{proof}

    To finish the proof of Theorem~\ref{Thm:two} it remains to show that there exists a suitable upper bound also for the expected number of vertices with distance $2n+1$ to the origin, for $n\in\mathbb{N}$. 

    \begin{proposition}\label{Prop:n_odd}
        For $n\in\mathbb{N}$ the following inequality holds
         \begin{equation*}
                \mathbb{E}_0 \Bigl[\sum\limits_{\textbf{x}\in\mathcal{X}} \mathbbm{1}_{\textbf{0} \overset{2n+1}{\sim}\textbf{x}} \Bigr]
                \leq u^{n+1}c^{n+1}10^{n+1} \mathbb{E}\bigl[V^2\bigr]^{n+2}
    \end{equation*}
    \end{proposition}
    \begin{proof}
    We start by rewriting the expected number of vertices that are connected to the origin by a path of length $2n+1$ by 
    \begin{align*}
        \mathbb{E}_0 \Bigl[\sum\limits_{\textbf{x}\in\mathcal{X}} \mathbbm{1}_{\textbf{0} \overset{2n+1}{\sim}\textbf{x}} \Bigr]
        &=  \mathbb{E}_0 \Bigl[\sum\limits_{_{\substack{\textbf{x}_2,\textbf{x}_4,\dots,\\\textbf{x}_{2n},\textbf{x}_{2n+1}\in\mathcal{X}}}}^{\neq} \prod\limits_{i=1}^n \mathbbm{1}_{\textbf{x}_{2i-2} \overset{2}{\sim}\textbf{x}_{2i}}  \mathbbm{1}_{x_{2n+1}\in  2\bar{R}_{{x}_{2n}}}\Bigr] \\ 
        & \quad+\mathbb{E}_0 \Bigl[\sum\limits_{_{\substack{\textbf{x}_2,\textbf{x}_4,\dots,\\\textbf{x}_{2n},\textbf{x}_{2n+1}\in\mathcal{X}}}}^{\neq} \prod\limits_{i=1}^n \mathbbm{1}_{\textbf{x}_{2i-2} \overset{2}{\sim}\textbf{x}_{2i}}  \mathbbm{1}_{x_{2n}\in  2\bar{R}_{{x}_{2n+1}}}\Bigr] 
      \end{align*}
      \begin{align*}  
        & \quad+ \mathbb{E}_0 \Bigl[\sum\limits_{_{\substack{\textbf{x}_2,\textbf{x}_4,\dots,\\\textbf{x}_{2n},\textbf{x}_{2n+1}\in\mathcal{X}}}}^{\neq} \prod\limits_{i=1}^n \mathbbm{1}_{\textbf{x}_{2i-2} \overset{2}{\sim}\textbf{x}_{2i}}  \mathbbm{1}_{x_{2n}\not\in  2\bar{R}_{{x}_{2n+1}}} \mathbbm{1}_{x_{2n+1}\not\in  2\bar{R}_{{x}_{2n}}}\Bigr]. 
    \end{align*}
    To find an upper bound for this expression we are again looking at the individual summands. The first summand can be rewritten as
    \vspace{-2pt}
    \begin{align*} 
        u^{2n+1} \!\!\!\!\!\!\!\!\!&\smash{\sum\limits_{_{\substack{j_{_{0}},j_{_{2}},\dots,\\j_{_{2n}},j_{_{2n+1}}\in\mathbb{N}^d}}}  {p_{j_{_{2n+1}}}}\prod\limits_{i=0}^{ {n}} p_{j_{_{2i}}} \int\limits_{\mathbb{R}^d}\dots\int\limits_{\mathbb{R}^d}}\int\limits_{ {2}\bar{R}_{{x}_{2n}}} \hspace{-7pt}\mathbb{P}_{2n}\Big(\bigcap_{k=1}^n \{\textbf{x}_{2(k-1)} \overset{2}{\sim}\textbf{x}_{2k}\}\Big) \text{d}\lambda(x_1) \dots \text{d}\lambda (x_{2n+1}) \\
        &\leq  \, c^{n+1} u^{n+1} \!\!\!\!\!\!\!\!\!\sum\limits_{_{\substack{j_{_{0}},j_{_{2}},\dots,\\j_{_{2n}},j_{_{2n+1}}\in\mathbb{N}^d}} }
        \prod\limits_{s=1}^{ {d}} j_{_{2n}}^{_{(s)}} p_{j_{_{2n+1}}} p_{j_{_{2n}}}\prod\limits_{k=0}^{n-1} p_{j_{_{2k}}}\biggl( \prod\limits_{i=1}^{ {d}} j_{_{2k}}^{_{(i)}} +  \prod\limits_{i=1}^{ {d}} j_{_{2k+2}}^{_{(i)}} +  8 \prod\limits_{i=1}^{ {d}}j_{_{2k}}^{_{(i)}}j_{_{2k+2}}^{_{(i)}} \biggr).
    \end{align*}
    We obtain the inequality by first using that \smash{$\lambda( {2}\bar{R}_{x_{2n}}) = c\prod_{s=1}^{ {d}} j_{2n}^{(s)}$} and then the previously shown result for the expected number of vertices that are connected via a path of length~$2n$ from the origin. Similar to the previous calculations we get by using the distributional symmetry of $\bar{R}$ for the second summand
     \begin{align*}
        \mathbb{E}_0 &\Bigl[\sum\limits_{_{\substack{\textbf{x}_2,\textbf{x}_4,\dots,\\\textbf{x}_{2n},\textbf{x}_{2n+1}\in\mathcal{X}}}}^{\neq} \prod\limits_{i=1}^n \mathbbm{1}_{\textbf{x}_{2i-2} \overset{2}{\sim}\textbf{x}_{2i}}  \mathbbm{1}_{x_{2n}\in  2\bar{R}_{{x}_{2n+1}}}\Bigr] \\
        &= \mathbb{E}_0 \Bigl[\sum\limits_{_{\substack{\textbf{x}_2,\textbf{x}_4,\dots,\\\textbf{x}_{2n},\textbf{x}_{2n+1}\in\mathcal{X}}}}^{\neq} \prod\limits_{i=1}^n \mathbbm{1}_{\textbf{x}_{2i-2} \overset{2}{\sim}\textbf{x}_{2i}}  \mathbbm{1}_{x_{2n+1}\in  (2\bar{R}_{{x}_{2n+1}} - x_{2n+1} + x_{2n}) }\Bigr]\\
        &\leq   c^{n+1} u^{n+1} \!\!\!\!\!\!\!\!\!\smash{\sum\limits_{_{\substack{j_{_{0}},j_{_{2}},\dots,\\j_{_{2n}},j_{_{2n+1}}\in\mathbb{N}^d}} }}
        \prod\limits_{s=1}^{ {d}} j_{_{2n+1}}^{_{(s)}} p_{j_{_{2n+1}}} p_{j_{_{2n}}} \prod\limits_{k=0}^{n-1} p_{j_{_{2k}}}\biggl( \prod\limits_{i=1}^{ {d}} j_{_{2k}}^{_{(i)}} +  \prod\limits_{i=1}^{ {d}} j_{_{2k+2}}^{_{(i)}} + 8 \prod\limits_{i=1}^{ {d}}j_{_{2k}}^{_{(i)}}j_{_{2k+2}}^{_{(i)}} \biggr).\\
    \end{align*}
    Looking now at the last summand we get
    \begin{align*}
         \mathbb{E}_0 &\Bigl[\sum\limits_{_{\textbf{x}_2,\textbf{x}_4,\dots,\textbf{x}_{2n}\in\mathcal{X}}}^{\neq} \prod\limits_{i=1}^n \mathbbm{1}_{\textbf{x}_{2i-2} \overset{2}{\sim}\textbf{x}_{2i}}  \mathbbm{1}_{x_{2n}\not\in  2\bar{R}_{{x}_{2n+1}}} \mathbbm{1}_{x_{2n+1}\not\in  2\bar{R}_{{x}_{2n}}}\Bigr]\\
        &= u^{2n+1} \smash{\sum\limits_{_{j_{_{0}},j_{_{2}},\dots,j_{_{2n}}\in\mathbb{N}^d} }
        p_{j_{_{2n+1}}}\prod\limits_{m=0}^{n}p_{j_{_{2m}}}  \int\limits_{\mathbb{R}^d} \dots  \int\limits_{\mathbb{R}^d}} \mathbb{P}_{2n+1}\Big(\bigcap_{i=1}^{_{2n+1}} \{\textbf{x}_{i-1} \sim\textbf{x}_{i}\} \Big)\mathbbm{1}_{x_{2n}\not\in  2\bar{R}_{{x}_{2n+1}}} \\
        &\hspace{220pt} \times \mathbbm{1}_{x_{2n+1}\not\in  2\bar{R}_{{x}_{2n}}} \,\text{d} \lambda (x_{1}) \dots\text{d}\lambda(x_{2n+1})  \\
        &\leq  c u^{2n+1} \hspace{-10pt}\sum\limits_{_{j_{_{0}},j_{_{2}},\dots,j_{_{2n}}\in\mathbb{N}^d} }
        p_{j_{_{2n+1}}}\prod\limits_{m=0}^{n}p_{j_{_{2m}}} \int\limits_{\mathbb{R}^d} \dots  \int\limits_{\mathbb{R}^d}\int\limits_{ {\mathbb{R}^d\setminus} B_1(x_{2n})} \hspace{-5pt} \mathbb{P}_{2n}\Big(\bigcap_{i=1}^{2n} \{\textbf{x}_{i-1} \sim\textbf{x}_{i}\} \Big) \\[7pt]
        &\hspace{70pt}\times{|x_{2n} - x_{2n+1}| ^{-(\alpha_k- \varepsilon)}  \prod\limits_{s=1}^{d-k} \frac{j_{_{2n}}^{_{(s)}}+|x_{2n} - x_{2n+1}|^{\beta}}{|x_{2n} - x_{2n+1}|} \,\text{d} \lambda (x_{1}) \dots\text{d}\lambda(x_{2n+1})}.
        \end{align*}
        We get the inequality since we first start again by looking at the last connection of the path. We therefore first count how many vertices are connected to $\textbf{x}_{2n}$ by an edge. As before we look at 
        the $k$-th diameter of $\bar{R}_{x_{2n+1}}$ is big enough so that an intersection of $\bar{R}_{x_{2n+1}}$ and $\bar{R}_{x_{2n}}$ is possible. For that we use similar considerations as before, namely
        \begin{equation*}
            \dist(\bar{R}_{x_{2n}},x_{2n+1}) \asymp |x_{2n} - x_{2n+1}|.
        \end{equation*}
        This leads to the sum with the index $k$ and to the term with the exponent $-(\alpha_k-\varepsilon)$. Moreover we bound the $(d-1)$ dimensional Lebesgue measure of the set of orientations that results in an intersection from above by assuming again that the $\bar{R}_{x_{2n}}$ is oriented in such a way that the largest face of it is perpendicular to the orientation of ${x_{2n+1}-x_{2n}}$. This and the use of Remark~\ref{rem:winkel} leads to the product term. Using now substitution we can rewrite the last expression to 
        \begin{align*}
        &c u^{2n+1} \hspace{-10pt} \sum\limits_{_{j_{_{0}},j_{_{2}},\dots,j_{_{2n}}\in\mathbb{N}^d} }
        p_{j_{_{2n+1}}}\prod\limits_{m=0}^{n}p_{j_{_{2m}}} \prod\limits_{s=1}^{ {d}} j_{_{2n}}^{_{(s)}}  \int\limits_{\mathbb{R}^d} \dots  \int\limits_{ {\mathbb{R}^d\setminus} B_1(0)}  \mathbb{P}_{2n}\Big(\bigcap_{i=1}^{2n} \{\textbf{x}_{i-1} \sim\textbf{x}_{i}\} \Big) \\
        &\hspace{170pt}\times|\tilde{x}_{2n+1}| ^{-\alpha_k+ \varepsilon +(d-k)(-1+\beta)} \,\text{d} \lambda (x_{1}) \dots\text{d}\lambda(\tilde{x}_{2n+1})\\
        &\leq cu^{2n+1} p_{j_{_{2n+1}}}\prod\limits_{m=0}^{n}p_{j_{_{2m}}}\prod\limits_{s=1}^{ {d}} j_{_{2n}}^{_{(s)}} \int\limits_{\mathbb{R}^d}\dots\int\limits_{\mathbb{R}^d}\int\limits_{1}^{\infty}\mathbb{P}_{2n}\Big(\bigcap_{i=1}^{2n} \{\textbf{x}_{i-1} \sim\textbf{x}_{i}\} \Big) \\
        &\hspace{180pt}\times r^{-\alpha_k+ \varepsilon +(d-k)(-1+\beta)+d-1}\,\text{d} \lambda (x_{1}) \dots\text{d}\lambda({x}_{2n}) \text{d}r\\
        &\leq c u^{2n+1} p_{j_{_{2n+1}}}\prod\limits_{m=0}^{n}p_{j_{_{2m}}}\prod\limits_{s=1}^{ {d}} j_{_{2n}}^{_{(s)}} \int\limits_{\mathbb{R}^d}\dots\int\limits_{\mathbb{R}^d}\mathbb{P}_{2n}\Big(\bigcap_{i=1}^{2n} \{\textbf{x}_{i-1} \sim\textbf{x}_{i}\} \Big)  \,\text{d} \lambda (x_{1}) \dots\text{d}\lambda({x}_{2n})
        \end{align*}
        In the first inequality we use polar coordinates and in the second inequality we integrate over $r$, which leads to a finite term since $\beta< {\tfrac{\alpha_k-k}{d-k}}$. 
        Using now inequality \textit{(IH)} of Proposition~\ref{Prop:IH} leads to the following upper bound
        \begin{equation*}
        u^{n+1}  c^{n+1}\sum\limits_{_{j_{_{0}},j_{_{2}},\dots,j_{_{2n}}\in\mathbb{N}^d} }p_{j_{_{2n}}} \prod\limits_{s=1}^{ {d}} j_{_{2n}}^{_{(s)}}\prod\limits_{k=0}^{n-1} p_{j_{_{2k}}}\biggl( \prod\limits_{i=1}^{ {d}} j_{_{2k}}^{_{(i)}} +  \prod\limits_{i=1}^{ {d}} j_{_{2k+2}}^{_{(i)}}  + 8 \prod\limits_{i=1}^{ {d}}j_{_{2k}}^{_{(i)}}j_{_{2k+2}}^{_{(i)}} \biggr). 
    \end{equation*}
    All together this yields the upper bound for the expected number of vertices that are connected to the origin via a path of length $2n+1$ 
    \begin{equation*}
                \mathbb{E}_0 \Bigl[\sum\limits_{\textbf{x}\in\mathcal{X}} \mathbbm{1}_{\textbf{0} \overset{2n+1}{\sim}\textbf{x}} \Bigr]
                \leq u^{n+1}c^{n+1}10^{n+1} \mathbb{E}\bigl[V^2\bigr]^{n+2}
    \end{equation*}
    Note that as in the base clause we just focus on the case that the $k$th diameter is of order $|x|+|y|$ and not on the case that it is of order $|x|+|y|)^{1/\delta}$ or that the $k$th diameter is bigger then $(|x|+|y|)^{1/\beta}$ as all calculation results in bigger upper bounds or in the bound $\beta<\tfrac{\alpha_k}{2d}$ as it is explained at the beginning of the proof of Lemma~\ref{Lemma:S1toS8}.
    \end{proof}
    With all these inequalities we are able to prove Theorem~\ref{Thm:two}.
    
    \begin{proof}[Proof of Theorem~\ref{Thm:two}]
    With Proposition~\ref{Prop:IH} and \ref{Prop:n_odd} we get upper bounds for the expected number of vertices that are connected to the origin with a path of length $n\in\mathbb{N}$.  {These bounds are only useful if $\mathbb{E}[V^2]<\infty$, which we now verify. {Almost surely the assumption of Theorem~\ref{Thm:two} holds for some $k\in \mathcal K$, and for this $k$} we have $\bar D^{_{(j)}}\leq \bar D^{_{(1)}}\leq D^{_{(k)}}/c+1$ for $j\leq k$ and $\bar D^{_{(j)}}\leq (D^{_{(k)}})^{\beta}/c+1$ for $j>k$, so that}
    \begin{equation*}
         {V=\prod_{i=1}^d \bar D^{_{(i)}}\leq c\,{\sum_{k\in \mathcal K}}\bigl(D^{_{(k)}}+1\bigr)^{k+(d-k)\beta}.}
    \end{equation*}
    {As $D^{_{(k)}}$ is regularly varying with index $-\alpha_k$, all its moments of order strictly smaller than $\alpha_k$ are finite. Since $2k+2(d-k)\beta<\alpha_k$ is equivalent to $\beta<\frac{\alpha_k-2k}{2(d-k)}$, assumption $(c)$ yields $\mathbb{E}[V^2]<\infty$, as $\mathcal K$ is finite}.
    Choosing $u< (10c\mathbb{E}[V^2])^{ {-1}}$
    gives us $
        \sum_{n\in\mathbb{N}} \mathbb{E}_0[\sum_{\textbf{x}\in\mathcal{X}} \mathbbm{1}_{\textbf{0}\overset{n}{\sim} \textbf{x}}] < \infty$,
    as desired. We therefore have non-robustness if $$0\mathrel{ {\le}} \beta <\min\big\{ \frac{\alpha_k-2k}{2(d-k)},\frac{\alpha_k}{2d}\big\}.
    \vspace{-22pt}$$   
    \end{proof}

   \section{Examples}\label{Sect:examples}
    
    In this section we take a closer look at the examples presented in Section~\ref{Sect:introduction} and identify explicitly the parameter regime given by Theorem \ref{Thm:one}. As we will see, the criteria in the theorem are sharp for some of the examples, but not all.
    \subsection{Ellipsoids with long and short axes}
    We consider the case with $m$ short axes of length one and $d-m$ long axes of length~$R$. As stated in Proposition \ref{Prop:one} it is known that the grain distribution is dense if the volume of the convex body has infinite expectation. This is equivalent to requiring that $R^{d-m}$ has infinite expectation which is guaranteed if $\alpha < d-m$. \smallskip
    
    Looking now at Theorem~\ref{Thm:one} we see that robustness  is given if $\alpha < \min \{2(d-m), d\}$. This is clear by looking at the tail indices of the distribution of the diameters. Let $k\leq d-m$. We have for $x>0$ and $k\leq d-m$ that
    \begin{equation*}
        \mathbb{P}(D^{{(k)}} \geq x) = \mathbb{P}(R \geq x),
    \end{equation*}
    i.e. $D^{{(k)}}$ has a regularly varying tail distribution with index $-\alpha$. For $k> d-m$ we have $D^{{(k)}}=1$ almost surely. So $D^{_{(k)}}$ has a tail distribution with index $-\infty$ for $k>d-m$.  If we are now asking for $\alpha_k< \min\{2k,d\}$ for some $k\in \{1,\dots,d\}$ one can see that this is the same as requiring $\alpha < 2(d-m)$ if $d-m \leq m $ or $\alpha < d$ if $d-m>d/2$. This is equivalent to $\alpha < \min\{2(d-m),d\}$.\smallskip
    
     {The criterion of Theorem~\ref{Thm:two}, applied with $k=d-m$ and $\beta=0$, results in  $\alpha_{d-m}=\alpha>2(d-m)$,} and Proposition~\ref{Prop:Gouere}   gives us $\alpha>d$ to ensure $D^{{(1)}}\in\mathcal{L}^d$. In other words, we require $\alpha >  \min\{2(d-m), d\}$. Ignoring the boundary cases, we see that our criterion is sharp.

    \subsection{Ellipsoids with independent axes}
    In this setting we have random ellipsoids which are generated via independent random radii $R_1,\dots,R_d$ with regularly varying tail indices $-\beta_1,\dots,-\beta_d$. Without loss of generality we have $\beta_i \leq \beta_{i+1}$ for all $i\in\{1,\dots,d-1\}$. 
    We are now interested in the tail indices of the diameters $D^{_{(1)}},\dots,D^{_{(d)}}$. For $x>0$ we have
    \begin{equation*}
        \mathbb{P}(D^{_{(1)}}\geq x) = 1-\mathbb{P}(D^{_{(1)}} < x) = 1- \prod\limits_{k=1}^d (1- \mathbb{P}(R_k \geq x))
    \end{equation*}
    by using the independence of the radii. Using now the calculation rule of regularly varying functions (see for example \cite{Bingham1989}) 
    we have that $D^{_{(1)}}$ is regularly varying with index~$-\beta_1$. Using again the independence of the radii and considering the event $\{D^{_{(k)}}\geq x\}$ leads us to the tail indices of the other diameters; more precisely $D^{_{(k)}}$ is regularly varying with tail index $-\sum_{i=1}^k \beta_i$. 
    Using Proposition~\ref{Prop:one} we get that the grain distribution is sparse if $\vol(C) \in \mathcal{L}^1$. In our case we have due to independence 
    \begin{equation*}
        \mathbb{E}[\vol(C)] = \prod\limits_{i=1}^d \mathbb{E}[R_i],
    \end{equation*}
    which is finite if $\min\limits_{i=1}^d \beta_i >1$. \pagebreak[3]
    
    Theorem~\ref{Thm:one} gives the condition for the $\alpha_k:= \sum_{i=1}^k \beta_i$ to ensure robustness. This can be equivalently stated as
    \begin{enumerate}[(i)]
        \item There exists $1\leq k \leq d$ such that $\alpha_k < \min\{2k,d\}$.
        \item There exists $1\leq i\leq d$ such that $\beta_i <2$.
    \end{enumerate}
     {(ii) follows from (i) since, if $\alpha_k \geq \min\{2k,d\}$ for every $k\in\{1,\dots,d\}$, we have in particular that $\beta_1 = \alpha_1 \geq 2$, and due to $\beta_{i+1}\geq \beta_i$ then $\beta_i \geq2 $ for every $i\in\{1,\dots,d\}$. Conversely, if $\beta_i<2$ for some $i$, then $\beta_1<2$ by the ordering, and hence $\alpha_1=\beta_1<2=\min\{2,d\}$, which is (i) with $k=1$.}
    Proposition~\ref{Prop:Gouere} gives us that $\beta_1>d$ implies the grain distribution is non-robust.
    
    %
    \subsection{Ellipsoids with strongly dependent axes}\label{Sect:dependent_axes}
    In this model we have ellipsoids with diameters which are defined via $D$ a Pareto distributed random variable on $( {1},\infty)$ with tail index $-\alpha$ for $\alpha\geq 1$. The diameters of $C\subset \mathbb{R}^3$ are given as
    \[
     D^{_{(1)}}=D, D^{_{(2)}}= D^{\gamma}, D^{_{(3)}}=1
    \]
    with $0\leq \gamma\leq 1$. We have $\vol(C)=cD^{1+\gamma}$ for some fix $c\in(0,\infty)$. So that $C$ is sparse if $D^{1+\gamma}\in \mathcal{L}^1$, which is satisfied for $\alpha >1+\gamma$.
\smallskip

    By definition, we have that $D^{_{(1)}}$ is regularly varying with index $-\alpha$ and $D^{_{(2)}}$ with index $-\alpha/\gamma$. Theorem~\ref{Thm:one} gives us robustness for $\alpha<2$ or $\alpha/\gamma < 3$. In addition to that the conditions of Theorem~\ref{Thm:one'} are satisfied so that $C$ is also robust for $\alpha<2+\gamma$. This leads in total to robustness if $\alpha < \max\{2+\gamma,3\gamma\}$. 
\smallskip

    We apply Theorem~\ref{Thm:two} with $k=1$ and $\beta=\gamma$. Its assumptions on the diameters hold, since $D^{_{(2)}}=D^{\gamma}=(D^{_{(1)}})^{\gamma}$ and $D^{_{(3)}}=1$. Condition $(c)$ reads $\gamma<\min\{\frac{\alpha-2}{4},\frac{\alpha}{6}\}$, which is equivalent to $\alpha>2+4\gamma$ since $\gamma\leq 1$. Hence Theorem~\ref{Thm:two} gives non-robustness for $\alpha>2+4\gamma$.
    Finally, Proposition~\ref{Prop:Gouere} shows that $C$ is non-robust if $\alpha>3$ as this results in $D^{_{(1)}}\in\mathcal{L}^d$. Altogether we have non-robustness if $\alpha>\min\{3,2+4\gamma\}$.

    \subsection{Random triangles}
    By our triangle construction $D^{_{(1)}}$ is the length of the hypotenuse, i.e. $D^{_{(1)}}=R$ is regularly varying with tail index $-\alpha=:-\alpha_1$. As the volume of a triangle is half of the product of one side and the corresponding height and the height orthogonal to the hypotenuse is the second diameter, we get
    \begin{equation*}
        \lambda(C) =\tfrac14 R^{1+\beta}= \tfrac12D^{_{(1)}}D^{_{(2)}},
    \end{equation*}
    so that $D^{_{(2)}}=\frac12 R^{\beta}$. Our proposition states that the grain distribution is sparse if $\vol(C)\in\mathcal{L}^1$. In this example this is the case whenever $\alpha > \beta+1$.  {To obtain robustness, Theorem~\ref{Thm:one} requires that there exists  $k\in\{1,2\}$  such that $\alpha_k<\min\{2k,2\}=2$, which holds if $\alpha <2$ as $\alpha_1=\alpha$. To obtain non-robustness we use Proposition~\ref{Prop:Gouere}: as $d=2$ we have $D^{_{(1)}}=R\in\mathcal{L}^d$ whenever $\alpha>2$, so that $C$ is non-robust if $\alpha>2$. Note that Theorem~\ref{Thm:two}, applied with $k=1$ and the given $\beta$, only yields non-robustness if $\beta<\min\{\tfrac{\alpha-2}{2},\tfrac{\alpha-2}{1},\tfrac{\alpha}{4}\}$, i.e.\ if $\alpha>2+2\beta$, which is weaker.}
    \subsection{Marginals do not determine robustness}\label{Sect:indeterminate}

In this section we show that for appropriately chosen exponents $\alpha_1$, $\alpha_2$ and $\alpha_3$, the 3-dimensional Poisson Boolean model with regularly varying diameters can exhibit both robustness, as well as non-robustness, depending on the correlation structure between the diameters. This shows that beyond the universal criteria of robustness some higher order properties beyond the marginal distributions determine whether the grain distribution is robust or non-robust.\smallskip

Recall that we set $0<\gamma<\frac12$ and $2<\alpha<2+\gamma$ and define $D$ Pareto distributed on $(1,\infty)$ with tail index $-\alpha$ and $X\in\{H,T\}$ an independent fair coin flip. 
We take an ellipsoid $C\subset \mathbb R^3$ with
    \begin{align*}
    D^{_{(1)}}=D, D^{_{(2)}}=D^{_{(3)}}=1
    & \text{ if $X=H$,}\\
    D^{_{(1)}}=D^{_{(2)}}=D^{\gamma}, D^{_{(3)}}=1
    & \text{ if $X=T$.}
    \end{align*}
    Then, $D^{_{(1)}}$ is regularly varying with index $-\alpha$ and $D^{_{(2)}}$ with index $-\alpha/\gamma$. \medskip

    {We now check the condition of Theorem~\ref{Thm:two}. In the case $\{X=H\}$ we let $k=1$ and in the case $\{X=T\}$ we let $k=2$. With $\beta=0$ we get non-robustness if $\alpha>2$, $\alpha/\gamma>4$ and the second moment of the volume is finite.  Our grain distribution satisfies the moment condition as}
\begin{equation*}
     {\mathbb{E}[\vol(C)^2]\leq c\,\bigl(\tfrac{1}{2}\mathbb{E}[D^2]+ \tfrac{1}{2}\mathbb{E}\bigl[D^{4\gamma}\bigr]\bigr),}
\end{equation*}
{which is finite for $\alpha>2>4\gamma$.}
\smallskip

If instead we consider an ellipsoid with
\begin{align*}
    D^{_{(1)}}=D, D^{_{(2)}}= D^{\gamma}, D^{_{(3)}}=1
    & \text{ if $X=H$,}\\
    D^{_{(1)}}=D^\gamma, D^{_{(2)}}= D^{_{(3)}}=1
    & \text{ if $X=T$,}
    \end{align*}
    then the diameters of the resulting grain have the same marginal distribution. 
    We use Theorem~\ref{Thm:one'} with $k=1$,  $\gamma_1=\alpha$ and $\gamma_2=\gamma$, so that
    \begin{align*}
         \mathbb{P}_C&\big(D^{_{(2)}}_C
            \ge c (D^{_{(1)}}_C)^{\gamma},
            D^{_{(1)}}_C \ge x\big)\\
            & {= \mathbb{P}_C\big(D^\gamma
            \ge c D^{\gamma},
D \ge x, X=H\big)+ \mathbb{P}_C\big(1
            \ge c (D^{\gamma})^{\gamma},
D^\gamma \ge x,X=T\big)}
     \end{align*}
    is regularly varying with index $-\alpha$ as the first summand is regularly varying with index $-\alpha$ and the second one is zero. By  choice of parameters we have $\gamma_1=\alpha \mathrel{ {<}} 2 +  \gamma_2 = 2+\gamma$, which results in robustness.\bigskip
    
{\textbf{Acknowledgment:} The material of this paper is taken from the PhD thesis~\cite{Korfhage2025}.
MK and PM were supported by DFG project 444092244 ``Condensation in random geometric graphs" within the priority programme SPP~2265.}

\end{document}